\documentclass[english]{amsart} 

\usepackage{ulem}
\usepackage{graphicx}
\usepackage{amsmath,amsfonts,amsthm}
\usepackage{amssymb}
\usepackage{multirow,array}
\usepackage{comment}
\usepackage{enumitem}
\usepackage{emptypage}
\usepackage{hyperref}
\usepackage{color}
\usepackage{mathrsfs}

\usepackage{url}

\usepackage[T1]{fontenc}
\usepackage{marvosym}
\usepackage{amsfonts}
\usepackage{amssymb}%
\usepackage{graphicx}

\usepackage{mathtools,bm}

\usepackage{shuffle}

\usepackage{mathrsfs}
\usepackage{tikz}
\usepackage{mhequ}

\usepackage{geometry}
\usetikzlibrary{arrows,snakes,backgrounds}

\usetikzlibrary{snakes}
\usetikzlibrary{decorations}
\usetikzlibrary{positioning}
\usetikzlibrary{shapes}

\usepackage[textsize=tiny]{todonotes}

\newcommand{\Om}{\Omega}
\newcommand{\ome}{\omega}

\newcommand{\vep}{\varepsilon}

\newcommand{\be}{\begin{equation}}
\newcommand{\ee}{\end{equation}}
\newcommand{\bea}{\begin{eqnarray}}
\newcommand{\eea}{\end{eqnarray}}

\newcommand{\beaa}{\begin{eqnarray*}}
\newcommand{\eeaa}{\end{eqnarray*}}

\newcommand{\lan}{ \langle }
\newcommand{\ran}{ \rangle}
\newcommand{\gam}{\gamma}
\newcommand{\lam}{ \lambda}

\newcommand{\MF}{\mathcal{F}}

\newcommand{\MM}{\mathbb{M}}

\newcommand{\ML}{\mathcal{L}}
\newcommand{\MY}{\mathcal{Y}}
\newcommand{\MZ}{\mathcal{Z}}

\newcommand{\mC}{\mathbb{C}_{T,d}}
\newcommand{\D}{\mathbb{D}_{T,d}}
\newcommand{\hD}{\hat{\mathbb{D}}_{T,d}}
\newcommand{\hC}{\hat{\mathbb{C}}_{T,d}}

\newcommand{\ty}{\tilde{y}}
\newcommand{\hY}{\hat{Y}}
\newcommand{\hZ}{\hat{Z}}
\newcommand{\hf}{\hat{f}}
\newcommand{\hPhi}{\hat{\Phi}}

\newcommand{\hu}{\hat{u}}

\newcommand{\td}{\tilde{d}}

\newcommand{\tB}{\tilde{B}}
\newcommand{\teta}{\tilde{\eta}}
\newcommand{\txi}{\tilde{\xi}}
\newcommand{\tX}{\tilde{X}}
\newcommand{\tY}{\tilde{Y}}
\newcommand{\tU}{\tilde{U}}
\newcommand{\tMY}{\tilde{\MY}}
\newcommand{\tty}{\tilde{y}}
\newcommand{\tZ}{\tilde{Z}}

\newcommand{\tome}{\tilde{\omega}}
\newcommand{\tMF}{\tilde{\MF}}
\newcommand{\tc}{\tilde{c}}
\newcommand{\tP}{\tilde{P}}
\newcommand{\tf}{\tilde{f}}

\newcommand{\tOmega}{\tilde{\Omega}}

\newcommand{\bE}{\bar{E}}

\newcommand{\bB}{\bar{B}}
\newcommand{\bareta}{\bar{\eta}}
\newcommand{\bxi}{\bar{\xi}}

\newcommand{\bX}{\bar{X}}

\newcommand{\gamp}{\gam'}

\newcommand{\etap}{\eta'}

\newtheorem{thm}{Theorem}[section]
\newtheorem{lem}[thm]{Lemma}
\newtheorem{coro}[thm]{Corollary}
\newtheorem{rem}[thm]{Remark}
\newtheorem{prop}[thm]{Proposition}
\newtheorem{example}[thm]{Example}
\newtheorem{defi}[thm]{Definition}

\newcommand{\E}{\mathbb{E}}
\newcommand{\tE}{\tilde{\mathbb{E}}}
\newcommand{\R}{\mathbb{R}}

\newcommand{\bS}{\mathbb{S}}
\newcommand{\bH}{\mathbb{H}}

\newcommand{\FC}{\mathscr{C}}

\newcommand{\op}{\mathcal{P}}
\newcommand{\FD}{\mathscr{D}}

 \newcommand{\gl}{{\gamma_t, \eta_t}}

\newcommand{\Z}{\mathbb{Z}}
\newcommand{\bx}{\mathbf{x}}

\newcommand{\Bf}{\mathbf{f}}

\begin{document}

\title{Classical solution of path-dependent mean-field semilinear PDEs}  



\author{Shanjian Tang}
\address{S. Tang, School of Mathematical Sciences, Fudan University, Handan Road 220, 200433, Shanghai, PRC.}
\email{sjtang@fudan.edu.cn}

\author{Huilin Zhang}
\address{H. Zhang, Research Center for Mathematics and Interdisciplinary Sciences, Shandong University, Binhai Road 72, 266237, Qingdao, PRC.}
\email{huilinzhang@sdu.edu.cn}
    



\keywords{path-dependent, classical solution, mean-field PDE, It\^o-Dupire formula} 

\subjclass{60G22, 60H10, 34C29} 




%


\begin{abstract}The paper concerns classical solution of path-dependent partial differential equations (PPDEs) with coefficients depending on both variables of path and path-valued measure, which are crucial to understanding large-scale mean-field interacting systems in a non-Markovian setting. We construct classical solutions of the PPDEs via solution of the forward and backward stochastic differential equations. To accommodate the intricacies introduced by the appearance of the path  in the coefficients, we develop a novel technique known as the ``parameter frozen'' approach to the PPDEs.

\end{abstract}

\maketitle

\tableofcontents


\newcommand{\ABS}[1]{\left(#1\right)} 
\newcommand{\veps}{\varepsilon} 



\section{Introduction}

\ \ \ \ \ \ Denote by $\mC$ the space of continuous functions on $[0,T]$ with values in $\R^d$ and by $\op^C_2$ the totality of probability measures on $ \mC$ with finite second order moments. Given functions $(b_1,\sigma_1)$ on $\mC$ and  $(b_2,\sigma_2)$ on $\op^C_2,$ we investigate a path-dependent mean-field PDE given by
\be\label{1}
\left\{
\begin{array}{l}
\partial_t u(t,\ome, \mu ) + \frac12 \text{Tr}\ [ \partial_{\ome}^2 u(t,\ome,\mu)\sigma_1(\ome_t)\sigma_1(\ome_t)^T]  + \partial_\ome u(t,\ome,\mu) b_1(\ome_t)\\[2mm]
\quad + \frac12  \text{Tr}\ [ \int_{\mC} \partial_{\tome} \partial_\mu u(t,\ome, \mu, \tome )\mu(d\tome)\sigma_2(\mu_t)\sigma_2(\mu_t)^T] + \int_{\mC} \partial_{\mu} u(t,\ome,\mu,\tome ) \mu(d\tome)b_2(\mu_t) \\[2mm]
\ \ \ + f(t,\ome, u(t,\ome,\mu), \sigma_1(\ome_t)\partial_\ome u(t,\ome,\mu), \mu, \ML_{u(t,W^\mu,\mu)})=0,\\
 \\
 u(T, \ome, \mu)= \Phi(\ome, \mu), \ \ \ (t,\ome,\mu)\in [0,T]\times \mC \times \op^C_2.
\end{array}
\right.
\ee
In this equation, (functional) derivatives $\partial_\ome$ and $\partial_\mu$ are taken in the spirit of Dupire and Lions (see the subsequent definitions \eqref{def-dx} and \eqref{fre}), respectively, and $W^\mu$ represents the canonical processes on $\mC$ under $\mu.$ The study of mean-field PDEs (or master equations) is crucial in understanding large systems in physics, games, and other areas of applied mathematics. While classical mean-field theory has its roots in statistical physics, quantum mechanics, and quantum chemistry (see Kac \cite{K56}, McKean \cite{M67}, Sznitman \cite{S84,S85,S86,S91}), recent developments have extended its applications to areas like stochastic differential games, partial differential equations (PDEs), and stochastic control, impacting fields such as engineering and economics (see e.g. \cite{L}, \cite{C12}, \cite{BFY13}, \cite{CD1}, \cite{CD2}, \cite{GLL}), just to mention a few.\\

 \ \ \ \ \ \ Mean-field PDEs have been studied  in various frameworks: Bensoussan et al. \cite{BFY14} consider the regular case when measure variables are restricted on those measures of square integrable density functions,  Cardaliaguet \cite{C12} gives a viscosity solution for  first-order HJB equations on a Wasserstein space,  Gomes and Saude \cite{GS14} survey well-posedness of HJB-FP equations for reduced mean-field games,  Buckdahn et al. \cite{BLP17} and Chassagneux et al. \cite{CCD} study classical solutions for second order master equations through stochastic differential equations (SDEs) and forward backward stochastic differential equations (FBSDEs) respectively, Carmona and Delarue \cite{CD3} consider the mean-field games and corresponding master equation with common noise,  Cardaliaguet et al. \cite{CDLL} give an analytic approach for master equations, Pham and Wei \cite{PW18} study the dynamic programming principle for Bellman master equation, Gangbo et al. study the well-posedness of master equations under non-monotonic conditions, etc. However, all these works consider the state-dependent case, which means $(\ome, \mu)$ in Equation~\eqref{1} take values in $\R^d \times \op_2(\R^k)$. Here,  $\op_2(\R^k)$ is the set of probability measures on $\R^k$ with finite second order moments. In practice, numerous problems could be non-Markovian or path-dependent: to mention a few, prices of  exotic options (e.g. Asian, chooser, lookback and barrier options \cite{D09}, \cite{CV91}, \cite{IM17}, \cite{H07}), stochastic differential game and stochastic control with delayed information (\cite{B92}, \cite{GM06}, \cite{S20}, \cite{Z17}, \cite{SZ19}), rough volatility \cite{GJR14}, \cite{BFG16}, etc. In particular,  Peng in his ICM 2010 lecture \cite{P} introduces the connection between non-Markovian FBSDEs and so called path-dependent PDEs (PPDEs), the latter of which is regarded as a crucial tool in the non-Markovian control theory.\\

\ \ \ \ \ \  Dupire \cite{D09} introduces a functional It\^o formula to incorporate the calculus of path-dependent functionals, which is subsequently developed by Cont-Fourni\'e \cite{CF10,CF13} and references therein (on the other hand, see another approach to path-dependent problems of  Flandoli and Zanco \cite{FZ16} by lifting the primal problem into a functional one in Banach spaces). In contrast to the classical approach of functional analysis (see e.g. Ahn \cite{A}), Dupire's approach is featured by the finite dimensional {\it vertical derivative} (see the following definition \eqref{def-dx}), and is admitted to solve non-Markovin problems (see e.g. \cite{TZ15}, \cite{S20}). Concerning the well-posedness of PPDEs, Peng and Wang \cite{PW16} consider smooth solutions of parabolic PPDEs; Ekren et al. \cite{EKTZ14,ETZ1,ETZ2} study the viscosity solution of quasilinear and fully nonlinear PPDEs; Cosso et al. \cite{CFGRT} treat PPDEs as the Hilbert space valued equations and build the viscosity solution; Peng and Song \cite{PS15} introduce a new path derivative and build Sobolev solutions for corresponding parabolic fully-nonlinear PPDEs via $G-$BSDEs \cite{P19}; Wu and Zhang \cite{WZ20} solve
a master equation with solutions in a form of $V(t, \mu)$, $\mu\in \op_2^C$. Recently, new viscosity solutions are introduced from different viewpoints by Zhou \cite{Z20}, Bouchard et al. \cite{BLT21} and Cosso et al. \cite{CGRR23}. \\

\ \ \ \ \ \ \ Although several  definitions of viscosity solutions are available,  the understanding of a smooth solution seems to be still very limited. The  well-understood smooth solution of PPDE
\be\label{ppde-1}
\left\{
\begin{array}{l}
\partial_t u(t,\ome  ) + \frac12 \text{Tr}\ [ \partial_{\ome}^2 u(t,\ome )\sigma(\ome_t)\sigma(\ome_t)^T]  + \partial_\ome u(t,\ome ) b (\ome_t)\\[2mm]
\ \ \ + f(t,\ome, u(t,\ome ), \sigma(\ome_t)\partial_\ome u(t,\ome) )=0,\\
 \\
 u(T, \ome)= \Phi(\ome ), \ \ \ (t,\ome)\in [0,T]\times \mC,
\end{array}
\right.
\ee
seems to be restricted within the case where $(b,\sigma)=(0,I)$ (though  $(b,\sigma)$ can be extended to be state-dependent in some sense). Here and in the following, we write $\ome(t)$ to denote the value of path $\ome$ at time $t,$ and $\ome_t=\ome_t(\cdot)=\ome(t\wedge .)$ to denote the path up to time $t.$ The problem comes from the definition of vertical derivatives. To be more precise, consider path-dependent SDE
\be 
\begin{split}
X^{\gam_t}(s)= \gam(t)+ \int_t^s b(X^{\gam_t}_r) dr + 	\int_t^s \sigma(X^{\gam_t}_r) d B(r),
\end{split}
\ee 
and a functional $\Phi$ defined on c\`adl\`ag space. Then 
we ask  whether or when we have the vertical differentiability of $\Phi(X^{\gam_t})$ in $\gam_t$, which is necessary to study the smooth solution of \eqref{ppde-1} via FBSDEs, even when both path functionals $b$ and $\sigma$  have smooth vertical derivatives of any order! The existing results exclude the case of a general forward system with nontrivial coefficients, and seem to be no help to study the corresponding control problems.\\

\ \ \ \ \ \ The paper focuses  smooth solution of  path-dependent PDEs. In contrast to the state-dependent case \cite{CCD}, smooth solution of equation \eqref{1} by FBSDEs meets with new issues. The first comes from the weak formulism of vertical derivatives (see identities \eqref{def-dx} and \eqref{fre} for details). Dupire's vertical derivative~\cite{D09} is defined in a finite-dimensional space, but depends on the ``cut-off'' time for functionals. In particular, to show the horizontal differentiability of the decoupling field $u$ on $[0,T]\times \mC,$ we have that for any $t,h \in [0,T]$ and $\gam_t \in \mC,$
\be
u(t+h, \gam_t)-u(t,\gam_t)= \left[u(t+h, \gam_t)- u(t+h, B^{\gam_t}_{t+h})\right]+\left[u(t+h, B^{\gam_t}_{t+h})-u(t,\gam_t)\right],
\ee
where 
$$B^{\gam_t}_{t+h}(s):= \gam(s)1_{s\le t} + (\gam(t)+ B(s)-B(t)) 1_{t+h>s \ge t} + (\gam(t)+ B(t+h)-B(t))1_{s \ge t+h},$$
 and $B$ is a Brownian motion. Then to apply It\^o's formula to compute the first difference of the right hand side of the last identity, the path ``differentiability'' of flow $u$ on $[t,t+h]$ is required. Such a differentiability is no longer the vertical derivative of $u$ since it is taken before the ``cut-off'' time $t+h$. To handle this issue, we introduce a new notation called ``strong vertical derivative'' (SVD) (see Definition \ref{svd}), built upon Dupire's vertical derivative, which restricts functionals to be vertically differentiable before the cut-off time. On one hand, the definition of SVDs is general enough to include all interesting continuously vertical differentiable functionals (see Example \ref{ex-path}). On the other hand, the SVD can be viewed as a pathwise definition for the Malliavin derivative (see e.g. \cite{N06}) on the c\`adl\`ag path space (see subsequent Remark \ref{withM} for details). Secondly, the existence of the derivative with respect to measure in Lions' sense usually requires the separability of the measurable space. However, in view of Dupire's vertical derivative and FBSDE theory, we work with the space of c\`adl\`ag functions under the uniform norm instead of Skorokhod norm. This leaves us without the general existence result for measure derivatives, and consequently we work with smooth coefficients such that we can construct derivatives via FBSDEs. Thirdly, as mentioned before, although there are many developments in viscosity solution theory for PPDEs, there is very few tool for smooth solutions even in the semi-linear case. 
 To study \eqref{1} via FBSDEs argument, we propose a novel approach involving a "coefficient frozen" strategy to handle complexities arising from path-dependent coefficients, contributing to the resolution of PPDEs with nontrivial coefficients. The argument is general enough to incorporate the mean-field path-dependent case (see Section \ref{semi} for details).\\

 \ \ \ \ \ \ A key contribution of the paper lies in establishing the well-posedness of the path-dependent mean-field equation with path-dependent coefficients, along with introducing and applying the SVD concept. In addition, the paper provides an It\^o formula and partial It\^o formula, which are fundamental in the  study of  path-dependent mean-field problems. The ``parameter frozen'' strategy proves instrumental in handling PPDEs with path-dependent coefficients. Our results  not only help to  understand path-dependnent mean-field equations, but also offer insights on numerical computations and the approximation of equilibrium in finite systems in view of the corresponding regularity needed (Fischer \cite{F14}, Lacker \cite{L14}). \\

\ \ \ \ \ \ The rest of the paper is organized as follows. In Section 2, we introduce notations of SVD with respect to paths and measures on path space, and build in the framework functional It\^o calculus incorporating paths and path measures. In Section 3, we show the  differentiability and regularity of associated FBSDE solutions. In Section 4,  we prove the existence and uniqueness of smooth solutions for path-dependent mean-field PDEs.



\section{Basic setup and It\^o calculus for functionals of both path and path-measure}

\subsection{The canonical setup}

For any fixed $T>0$, we denote by $\mC=C([0,T], \R^d) $ the canonical space and equip it with the supreme norm $\| \cdot \|_{[0,T]}.$  $W$ is the canonical process and  $\{\MF_t^W\}_{0\le t\le T}$ is the natural filtration. For any $(t,\omega)\in [0,T] \times \mC,$  ${\omega}_t$ is the cut-off path, meaning that $\ome_t\in \mC$ such that
\be
\omega_t(r)=\omega(r)1_{[0,t)}(r)+ \omega(t)1_{[t,T]}(r),\ \ r\in[0,T];
\ee
and $\ome(t)$ is the state of $\omega$ at time $t$. Let $\op^C_2$ be the set of probability measures on $(\mC, \MF_T^W)$ with finite second order moments, i.e. $\mu \in \op_2^C $ iff $|||\mu|||^2:=\E^\mu[\|W\|_{[0,T]}^2] < \infty.$  For $\mu \in \op^C_2, $  $\mu_t \in \op^C_2$ is the distribution of stopped process $W_t$ under $\mu.$ For any $\mu, \nu \in \op^C_2,$ we define the following classical 2-Wasserstein distance
\be
W_2(\mu, \nu)= \inf_{P\in \op(\mu, \nu)} \left(\int_{\mC \times \mC} \|u-v \|_{[0,T]}^2 \ dP(u, v) \right)^\frac12,
\ee
where $\op(\mu, \nu)$ is the set of all probability measures on $(\mC \times \mC, \MF^W_T \times \MF^W_T)$ with marginal measures $\mu$ and $\nu.$
To introduce functional derivative in the space of paths, we consider the space of c\`adl\`ag paths $\D:=D([0,T], \R^d),$ which can be equipped with the uniform topology $\|\cdot \|_{[0,T]},$ or the Skorohod topology
\be\label{skr}
d(\omega, \ome'):= \inf_{\lambda \in \Lambda_{[0,T]}} \sup_{t\in [0,T]} (|t-\lambda(t)| + |\ome(t)-\ome'(t)|),
\ee
where $\Lambda_{[0,T]}$ is the set of all strictly increasing continuous mappings on $[0,T]$ with $\lambda(0)=0$ and $\lambda(T)=T.$ In the following, we equip $\D$ with the uniform topology unless stated otherwise.
With the space $\mC$ being replaced with $\D$, notations such as $\op^D_2$ and $W_2(\mu, \nu)$ are self-explained.

\ \ \ \ \ \ Suppose that $(\Omega,\MF, P )$ is an atomless probability space supporting a $d$-dimensional Brownian motion $B,$ and $\{\MF_t\}_{t\in[0,T]}$ is the natural augmented filtration. For any $t\in[0,T]$ and $r\in [t,T],$ we define $\MF_r^t$ as the $\sigma$-algebra generated by $\{B(s)-B(t); t\le s \le r\}$ and completed under $P$.
For any (stopped up to time $t$) process $X_t,$ we denote by $\ML_{X_t}$ the law of the process $X_t$ and $\ML_{X(t)}$ the law of the random variable $X(t)$. In the following, we use notation $\MM_2^C$ ($\MM_2^D$, resp.) as the collection of measurable continuous processes (c\`adl\`ag processes, resp.) with laws in $\op_2^C$ ($\op_2^D$, resp.). Since for any $\mu \in \op_2^D,$ we can always find an atomless probability space $(\Omega,\MF_t, P )$ such that there exists a c\`adl\`ag process $\eta$ on this probability space with law $\mu,$ we will always suppose for any $\mu \in \op_2^D,$ $(\Omega,\MF, P )$ is rich enough to support a c\`adl\`ag process $\eta$ such that $\ML_{\eta}=\mu.$ Moreover, for any progressively measurable process $X$ and random variable $\xi$ on $(\Omega,\MF, P )$, we define the following norms if they are finite: for any $ t\in[0,T],$ $p\in \mathbb{N^+},$
\be
\|X\|_{\bS^p,[t,T]}^p:= \E^P[\|X\|^p_{[t,T]}],\ \ \|X\|_{\bH^p, [t,T]}^p:= \E^P[ (\int_t^T|X(r)|^2 dr)^{\frac p 2} ],\ \ \|\xi\|_{L^p}^p:= \E^P[|\xi|^p].
\ee
We write $\bS^p([t,T],\R^k)$, $\bH^p([t,T],\R^k)$ and $L^p(\MF_T,\R^k)$ for spaces of progressively measurable processes on $[t,T]$ and random variables with values in $\R^k$ and finite corresponding norms. Denote by $C^n(\R^m, \R^k)$ ($C^n_b(\R^m, \R^k)$, resp.) the space of (bounded, resp.) continuous functions from $\R^m$ to $\R^k$ with (bounded, resp.) continuous derivatives up to order $n.$ Usually, we omit $\R^k$ in $\bS^p([t,T],\R^k), \bH^p([t,T],\R^k), L^p(\MF_T,\R^k), C(\R^m, \R^k)$ when $k=1,$ and also omit the time interval $[t,T]$ if no confusion raised. Moreover, for $(Y,Z) \in \bS^p([t,T], \R^m) \times \bH^p([t,T], \R^n),$ we write
\be
\|(Y,Z)\|_{\bS^p\times \bH^p} := \left( \|Y \|_{\bS^p}^p +\|Z\|_{\bH^p}^p \right)^{\frac 1 p}.
\ee


\subsection{Strong vertical derivatives with respect to  path and path-measure}
Denote by $\hD$ the product space $[0,T] \times \D \times \op^D_2$ and by $\mathscr{D}$ the space of functionals on $\hD.$
A functional $f\in \FD $ is said to be {\it non-anticipative} if for any $(t, \ome, \mu),$ $f(t, \ome, \mu)=f(t, \ome_t, \mu_t)$, where $\mu_t$ is the law of $\eta_t$ with $\ML_\eta=\mu$. For non-anticipative $f\in \FD,$ we call $f$ continuous on $\hD$ and write $f\in \FC(\hD)$ if $f$ is continuous in the product space $[0,T] \times \D \times \op^D_2 $ equipped with the premetric:
\be\label{dp}
d_{p}((t,\ome,\mu),(t',\ome',\mu')) :=|t-t'|+\|\ome_t-\ome_{t'}\|+ W_2(\mu_t,\mu_{t'}).
\ee

For any non-anticipative $f  \in \FD, $ the horizontal derivative is defined as
\be\label{def-dt}
\partial_t f(t, \ome ,\mu  ):= \lim_{h\rightarrow 0^+} \frac{1}{h}[f(t+h, \ome_t,\mu_t )-f(t, \ome_t,\mu_t )], \ \forall \ (t,\ome,\mu) \in \hD.
\ee
For any $(t,x) \in [0,T] \times \R^d,$ define $\ome^{t,x}  \in \D$ by
\be
\ome^{t,x} := \ome  + x 1_{[t,T]}.
\ee
For any fixed $(t,\mu) \in [0,T] \times \op_2^D,$ $f(t, \cdot, \mu):   \D \mapsto \R$
is called {\it vertically differentiable} at $(t,\ome )$ (or $\ome_t$ for short), if $f(t, \ome^{t,x}, \mu)$ is differentiable at $x=0$, i.e. there exits $\partial_{\ome} f(t,\ome,\mu ) \in \R^d$ such that
\be\label{def-dx}
f(t, \ome+x 1_{[t,T]},\mu )= f(t, \ome ,\mu ) + \partial_{\ome} f(t,\ome,\mu ) x + o(|x|),\ \ \ \forall \ x\in \R^d,
\ee
and $\partial_{\ome} f(t,\ome,\mu )$ is then called the vertical derivative.
Now we introduce the notation of SVDs for the FBSDE argument in Section 4.

\begin{defi}\label{svd}
Suppose that $f  :[0,T] \times \D   \mapsto \R.$ For any $\tau \le t,$ we call $f$ strongly vertically differentiable at $(\tau,t, \ome)$ (or $\ome_\tau$ for short), if there exits $\partial_{\ome_\tau} f(t,\ome  ) \in \R^d$ such that
\be
f(t, \ome+x 1_{[\tau,T]}  )= f(t, \ome   ) + \partial_{\ome_\tau} f(t,\ome  )  x + o(|x|),\ \ \  \forall \ x\in \R^d.
\ee
In this case, $\partial_{\ome_\tau} f(t,\ome  )$ is called the strong vertical derivative (SVD) of $f$ at $(\tau,t,\ome)$. Moreover, if $f$ is strongly vertically differentiable at $(\tau,t,\ome)$ for any $\tau \le t,$ we call $f$ \it{strongly vertically differentiable} at $(t,\ome)$  (or $\ome_t$ for short).

\end{defi}

\begin{rem}\label{fulltime-SVD}
	Indeed, we can consider the SVD for any $\tau \in [0,T].$ For non-anticipating functionals we care about in this paper, we have 
	$$
	\partial_{\ome_\tau}f(t,\ome)= \partial_{\ome_\tau}f(t,\ome)1_{[0,t]}(\tau).
	$$ 
\end{rem}

\ \ \ \ \ \ Clearly, $f$ is strongly vertical differentiable at $\ome_t$ if and only if the mapping $x \mapsto f(t, \ome^{\tau,x})$ is differentiable at $x=0$ for any $\tau \le t$. In particular, if $f$ is non-anticipative and strongly vertically differentiable, $f$ is vertically differentiable and its vertical derivative at $(t,x)$ agrees with its strong vertical derivative at $(t,t,\ome).$ For the SVD $\partial_{\ome_\tau}f(t,\ome ),$ we can further define its SVDs in the same spirit: for any $\tau' \le t,$ define $\partial_{\ome_{\tau'}}\partial_{\ome_\tau}f(t,\ome )$ as the SVD of $\partial_{\ome_\tau}f(t,\ome )$ at $(\tau',t,\ome)$. In the following, we only need to consider the case $\tau'=\tau.$ We call $f$ has continuous SVDs or $\partial_{\ome_\tau}f(t,\ome )$ is continuous if $\partial_{\ome_\tau} f$ is continuous with respect to the metric: for any $(\tau,t,\ome )$ and $(\tau', t',\ome' )$ with $\tau \le t,\ \tau' \le t',$
\be
d_{sp}((\tau,t,\ome ), (\tau',t',\ome' )):=|\tau-\tau'|+|t-t'|+\|\ome_t-\ome'_{t'}\| .
\ee
Here are examples of strongly vertically differentiable functionals.

\begin{example}\label{ex-path}
Let $f:[0,T]\times \D \longmapsto \R$ and $(t,\ome) \in [0,T]\times \D.$

\begin{itemize}
\item[$(i)$]If $f(t,\ome) =F(t,\ome(t))$ for a function $F \in C^{1,k}([0,T]\times \R^d)$,
    then we have that for any $\tau_1,\tau_2,\cdots,\tau_j \in [0,t],$ $j\le k,$
\be
\partial_t f(t,\ome)= \partial_t F(t, \ome(t)), \quad \partial_{\ome_{\tau_j}} \cdots \partial_{\ome_{\tau_1}} f(t, \ome )= D_x^{  j} F(t,\ome(t)),
\ee
and thus $f$ has continuous strong vertical derivatives up to order $k$.

\item[$(ii)$] Suppose that
    $
    f(t,\ome) = \int_0^t F(r,\ome(r)) dr
    $
    with $F \in C^{1,k}([0,T]\times \R^d)$. Then for any $\tau_1,\tau_2,\cdots,\tau_j \in [0,t],$ $j\le k,$
\be
\partial_t f(t,\ome)=   F(t, \ome(t)), \ \ \ \partial_{\ome_{\tau_j}} \cdots \partial_{\ome_{\tau_1}} f(t, \ome )= \int_\tau^t D_x^{ j} F(r,\ome(r))dr,
\ee
with $\tau=\max_{1\le i\le j }\{\tau_i\}.$
Thus $f$ has continuous SVDs up to order $k$.

\item[$(iii)$]For a partition $0=t_0 < t_1 <\cdots <t_n=T ,$ and a continuously differentiable function
    $
    F : \underbrace{\R^d \times \R^d \times \cdots \R^d    }_{n}   \mapsto \R
    $, let
    \be
    f(T,\ome ):=F(\ome(t_1), \ome(t_2)-\ome(t_1),\cdots, \ome(T)-\ome(t_{n-1}) ).
    \ee
Then $f$ is strongly vertically differentiable at $(T,\ome)$: for $t>0,$
\be\nonumber
\partial_{\ome_t}f(T,\ome )=\sum_{j=1}^{n} \partial_{x_{j}} F(\ome(t_1), \ome(t_2)-\ome(t_1),\cdots, \ome(T)-\ome(t_{n-1}) )1_{(t_{j-1},t_j]}(t).
\ee

\item[$(iv)$]For fixed $t_0\in (0,T)$ and $F \in C^1(\R^d)$,
    define $
    f(T,\ome):= F( \ome(  t_0)).
    $
Thus $f$ has SVDs
\be
\partial_{\ome_{t}} f(T, \ome )=D_x F(\ome(t_0)) 1_{[0,t_0]}(t).
\ee
%

\item[$(v)$]For a given partition of $[0,T]:$ $0=t_0 < t_1 <\cdots <t_n=T $ and smooth functions $\{f_i\}_{i=0}^{n-1}$ on $\R^d,$ consider
\be
f(t,\ome):= \sum_{i=0}^{n-1} f_i(\ome(t_i)) 1_{[t_i,t_{i+1})} (t).
\ee
Then $f$ is strongly vertically differentiable at $\ome_t$ with
\be
\partial_{\ome_\tau} f(t,\ome)= \sum_{i=0}^{n-1} D f_i(\ome(t_i)) 1_{[t_i,t_{i+1})} (t) 1_{[0,t_i]} (\tau), \ \ \forall \tau \le t.
\ee

\end{itemize}

\end{example}

\begin{rem}\label{withM}
The relation between vertical derivative and Malliavin derivative is considered in \cite{CF13}, where an equivalence is built through martingale representation in both frameworks (see \cite[Theorem 6.1]{CF13}). According to $(iii)$ of Example \ref{ex-path}, the SVD is related to Malliavin derivatives restricted in the cylinder random variables or processes. 

\end{rem}

\ \ \ \ \ \ The following lemma follows immediately from Definition \ref{svd},  and  will be frequently used.

\begin{lem}\label{lipbd}
Suppose that $f:[0,T] \times \D \mapsto \R$ is strongly vertically differentiable, and uniformly Lipschitz continuous in $\ome:$
 \be
 |f(t,\ome)-f(t,\ome')|\le C \| \ome_t-\ome'_{t} \|, \quad \forall  (t,\ome, \ome') \in [0,T]\times \D\times \D.
 \ee
Then we have
$
|\partial_{\ome_\tau} f(t,\ome)| \le C
$
for any $(t,\ome)\in [0,T] \times \D$ and $\tau\le t.$
\end{lem}

\ \ \ \ \ \ For a non-anticipative functional $f\in \FD$, consider its {\it lift}
$\Bf: [0,T] \times \D \times \MM^D_{2} \mapsto \R,$
\be
\Bf(t, \ome, \eta):=f(t, \ome, \ML_{\eta}).
\ee
In the spirit of Lions \cite{L} (also see \cite{WZ20} for derivative with respect to measure on the path space), we call $f$ Fr\'echet (vertically) differentiable at $(t,\mu)$ (or $\mu_t$ for short), if for any fixed $\ome,$ $\Bf$ is Fr\'echet (vertically) differentiable at $(t,\eta)$ (or $\eta_t$ for short) with $\ML_\eta=\mu$ in the following sense: there exits $D_{\eta}\Bf(t, \ome, \eta) \in L^2_P(\MF_t , \R^d)$ such that for any $\xi \in L^2_P(\MF_t, \R^d),$
\be\label{fre}
\Bf(t, \ome, \eta+\xi 1_{[t,T]})= \Bf(t, \ome, \eta ) + \E^P[D_{\eta}\Bf(t, \ome, \eta)  \xi] + o(\|\xi \|_{L^2}).
\ee
In particular, it means that the following G\^ateaux derivative exits
\be\label{gat}
 \lim_{h\rightarrow 0} \frac{1}{h}[\Bf(t, \ome, \eta+ h \xi 1_{[t,T]}) - \Bf(t, \ome, \eta )]= \E^P[D_\eta\Bf(t, \ome, \eta)\xi].
\ee
Moreover, if there exists a non-anticipative jointly measurable functional $\partial_{\mu}f:\hD \times \D \mapsto \R,$ such that
\be\label{rep-dmu}
D_\eta \Bf(t,\ome, \eta)= \partial_{\mu}f(t,\ome,\mu,\eta), \ \ \ P\text{-}a.s.,
\ee
we call $f$ vertically differentiable at $(t,\mu)$ and $\partial_{\mu}f(t,\ome,\mu,\tome)$ the vertical derivative of $f(t,\ome, \cdot)$ at $(t,\mu)$ (or $\mu_t$).


\begin{rem}\label{fregat} Consider the validity for notations of Fr\'echet and G\^ateaux differentiability. Denote by $\Bf$ the lift of $f\in \FD.$ For any $\xi \in L^2_P(\MF_t , \R^d),$
let $F(t, \ome, \eta, \xi):= \Bf(t, \ome, \eta+\xi 1_{[t,T]}).$ Then $\Bf$ is Fr\'echet differentiable at $(t,\eta)$ in the above sense is equivalent to that $F(t, \ome, \eta, \xi)$ is Fr\'echet differentiable at $\xi=0$ in the classical sense. Similar argument for G\^ateaux differentiability also holds.

\end{rem}

\begin{rem} \label{exi-dmu}
Consider the existence of the derivative functional $\partial_{\mu}f $. If the lift $\Bf(t,\ome,\eta)$ of $f(t,\ome,\mu)$ is Fr\'echet differentiable at $\eta_t,$ and the derivative $D_\eta \Bf(t,\ome, \eta)$ is continuous in the sense that $D_\eta \Bf(t,\ome,\eta^n) \stackrel{L^2}{\longrightarrow} D_\eta \Bf(t,\ome,\eta)$ as $\eta^n \stackrel{L^2}{\longrightarrow} \eta$ under the Skorohod topology \eqref{skr}, then according to \cite[Theorem 2.2]{WZ20}, $\partial_\mu f$ exists in the sense of \eqref{rep-dmu}. However, to build smooth solutions for \eqref{1}, we need our It\^o formula (Theorem \ref{itoformula} and Corollary \ref{sito}) to be applicable for the larger class of functionals, which only need to be continuous with respect to the uniform topology. Luckily, we can construct the derivative directly by corresponding FBSDEs.

\end{rem}

\ \ \ \ \ \ \ For the uniqueness of $\partial_\mu f(t,\ome,\mu,\cdot),$ in view of identity \eqref{rep-dmu}, we see that it is unique $\mu$-a.s. in $\D$. Then for any $\mu  \in \op^D_2$ such that $\text{supp}(\mu)=\D,$ if $\partial_\mu f(t,\ome,\mu,\tome)$ is continuous in $ \tome \in \D,$ $\partial_\mu f(t,\ome,\mu,\cdot)$ is unique on $\D.$ Moreover, suppose that $\partial_\mu f(t,\ome,\cdot,\cdot)$ is jointly continuous on $  \op_2^D \times \D.$ Then for any $\mu_0\in \op_2^D,$ $\partial_{\mu} f(t,\ome,\mu_0, \cdot)$ is unique on $\D$. Indeed, choose any $\eta \in \MM^D_2$ with $\ML_{\eta}=\mu_0 \in \op_2^D,$ and any $\eta'\in (\MM^D_2)',$ which is independent of $\eta,$ such that $\text{supp}(\ML_{\eta'})=\D.$ Then for any $\vep>0, $ the functional $\partial_\mu f(t,\ome,\ML_{\eta+\vep \eta'},\cdot)$ is unique on $\D.$ It follows from continuity of $\partial_\mu f(t,\ome,\cdot,\cdot)$ that $\partial_\mu f(t,\ome,\mu_0,\tome)$ is unique as the limit of $\partial_\mu f(t,\ome,\ML_{\eta+\vep \eta'},\tome)$ as $\vep$ goes to zero. In conclusion, we have the following lemma.

\begin{lem} \label{uni-dmu}
Suppose that for any fixed $(t,\ome)\in[0,T] \times \D,$ the functional derivative $\partial_{\mu}f(t,\ome,\cdot,\cdot)$ is jointly continuous in $  \op_2^D \times \D.$ Then for any $(t,\ome,\mu) \in \hD,$ $\partial_{\mu}f(t,\ome, \mu ,\cdot) $ is unique on $\D$.

\end{lem}

\begin{rem} \label{d-ext}
The definition of vertical derivative given by \eqref{fre} and \eqref{gat} has natural extension for Banach space valued functionals. For any $t\in[0,T],$ suppose that $f(t, \ome , \mu) $ takes values in a (stochastic) Banach space $E_t$ (e.g. $\bS^2({[t,T]}), \bH^2({[t,T]}), L^2(\MF_t)$). Indeed, $f(t,\ome,\mu)$ has the natural lift $\Bf(t, \ome, \eta) \in E_t$ with $\ML_{\eta}=\mu$. If the mapping from $L^2(\MF_t)$ to $E_t$
$$
\begin{array}{lccl}
  \Bf(t, \ome, \eta+\cdot 1_{[t,T]}):& L^2(\MF_t)&  \longmapsto & E_t\\
& \xi && \Bf(t, \ome, \eta+\xi 1_{[t,T]})
\end{array}
$$
is Fr\'echet (vertical) differentiable with derivative $D_{\eta} \Bf(t, \ome, \eta) \in L(L^2(\MF_t),E_t)$ at $\xi=0,$ we call $f(t, \ome , \cdot)$ Fr\'echet (vertically) differentiable at $\mu_t$. Moreover, if there exists a jointly measurable functional $U: \hD \times \D \mapsto E_t$ such that for any $\xi \in L^2(\MF_t)$,
$
D_{\eta} \Bf(t, \ome, \eta) ( \xi) = \E^{P} [U(t, \ome, \mu, \eta) \xi ] ,
$
we call $\partial_{\mu} f(t, \ome, \mu, \cdot):= U(t, \ome, \mu, \cdot)$ the vertical derivative of $f(t,\ome,\cdot)$ at $\mu_t.$

\end{rem}

\ \ \ \ \ \ Now we introduce SVDs with respect to path-measure.

\begin{defi}\label{svdm}
For any $\tau,t\in [0,T]$ with $\tau \le t$ and $\mu \in \op_2^D,$
we call a non-anticipative functional $f:[0,T]\times \op_2^D \mapsto \R$ Fr\'echet (strongly vertically) differentiable at $(\tau,t,\mu)$ if its lift $\Bf(t,  \eta)$ with $\ML_\eta=\mu$ is Fr\'echet (strongly vertically) differentiable: there exits $D_{\eta_\tau}\Bf(t,   \eta) \in L^2_P(\MF_t , \R^d)$ such that for any $\xi \in L^2_P(\MF_\tau, \R^d),$
\be\label{sfre}
\Bf(t,   \eta+\xi 1_{[\tau,T]})= \Bf(t,  \eta ) + \E^P[D_{\eta_\tau}\Bf(t,   \eta)   \xi] + o(\|\xi \|_{L^2}).
\ee
In particular, it means that the following G\^ateaux derivative exits,
\be\label{sgat}
\lim_{h\rightarrow 0} \frac{1}{h} [\Bf(t,   \eta+ h \xi 1_{[\tau,T]}) - \Bf(t,   \eta )]= \E^P[D_{\eta_\tau}\Bf(t,   \eta)\xi].
\ee
We call $f$ strongly vertically differentiable at $(t,\mu)$ or $\mu_t,$ if it is Fr\'echet differentiable at $(\tau,t,\mu)$ for any $\tau \le t $, and moreover, there exists a jointly measurable non-anticipative functional $\partial_{\mu_\tau} f: [0,T] \times \op_2^D \times \D \mapsto \R^d $ such that
\be
D_{\eta_\tau}\Bf(t,   \eta)= \partial_{\mu_\tau} f(t, \mu, \eta),\ \ \ P\text{-}a.s..
\ee
$\partial_{\mu_\tau} f(t, \mu, \cdot)$ is then called the strong vertical derivative of $f(t, \cdot)$ at $(\tau,t,\mu).$

\end{defi}

\begin{rem}\label{ue-sdmu}
For the existence and uniqueness of the SVD at $\mu_\tau,$ we have similar results as Remark \ref{exi-dmu} and Lemma \ref{uni-dmu}. In particular, if for any $t \in[0,T] ,$ $\partial_{\mu_\tau}f(t, \cdot,\cdot)$ is jointly continuous on $\op_2^D \times \D$, then the SVD is unique. Moreover, we can extend SVDs in path-measure to the (stochastic) Banach framework as Remark \ref{d-ext}.

\end{rem}

\ \ \ \   Given strongly vertically differentiable $f:[0,T] \times \op_2^D \mapsto \R$, for any $(t, \mu,\tome)\in[0,T] \times \op_2^D \times \D$ and $\tau \le t,$ we can further consider SVDs of $\partial_{\mu_\tau}f $ with respect to $  \mu_t $ and $ \tome_t$: for any $\tau' \le t,$ consider $\partial_{\tome_{\tau'}}\partial_{\mu_\tau}f(t, \mu,\tome)$ as the SVD of $\partial_{\mu_\tau}f(t, \mu,\tome)$ at $(\tau',t,\tome)$; $\partial_{\mu_{\tau'}}\partial_{\mu_\tau}f(t, \mu,\tome,\tome')$ as the SVD of $\partial_{\mu_\tau}f(t, \mu,\tome)$ at $(\tau',t,\mu).$ In the subsequent sections, we only need to consider the case $\tau'=\tau$ and the second order derivative $\partial_{\tome_{\tau'}}\partial_{\mu_\tau}f(t, \mu,\tome)$. Moreover, we call $f$ has continuous SVDs or $\partial_{\mu_\tau}f(t,\mu, \tome)$ is continuous if $\partial_{\mu_\tau} f$ is continuous with respect to the following premetric: for any $(t, \mu, \tome)$ and $(t', \mu', \tome')$ with $\tau \le t,\tau' \le t',$
\be
d_{sp}((\tau,t,\mu,\tome), (\tau',t',\mu',\tome')):=|\tau-\tau'|+|t-t'|+  W_2(\mu_t,\mu'_{t'})+\|\tome_t-\tome'_{t'}\|.
\ee
$f$ is said to have continuous SVDs in path-measure up to order $2$, if both $ \partial_{\mu_\tau}f$ and $\partial_{\tome_\tau} \partial_{\mu_\tau}f$ are continuous with respect to the above topology.

\begin{example}\label{ex-meas}
Here we consider $f:[0,T]\times \op_2^D \mapsto \R$ and $(t,\mu) \in [0,T]\times \op_2^D.$

\begin{itemize}
\item[$(i)$]Suppose that $F\in C^{1,2}([0,T]\times \R^d)$ with  $|D_x^2 F|$ being uniformly bounded, and
    $
    f(t,\mu):= \E^{\mu}[F(t, W(t))].
    $
    Then we have that
\beaa
&&\partial_t f(t,\mu)= \E^{\mu}[\partial_t F(t, W(t))], \ \ \ \partial_{ {\mu_\tau}} f(t, \mu,\tome )= D_x F(t,\tome(t)),\\
&&\quad \text{and}\quad \partial_{\tome_\tau}\partial_{ {\mu_\tau}} f(t, \mu,\tome )= D_x^2 F(t,\tome(t)),  \  \ \  \forall \tau \in [0,t].
\eeaa
Thus $f$ has continuous SVDs up to order $2$.

\item[$(ii)$]Let $F$ as defined in $(i)$ and $
    f(t,\mu):=\E^{\mu}[ \int_0^t F(r,W(r)) dr ].
    $
    Then for any $\tau \in [0,t],$
\beaa
&&\partial_t f(t,\mu)=  \E^{\mu}[  F(t, W(t))], \ \ \ \partial_{\mu_{\tau}} f(t, \mu, \tome )= \int_\tau^t D_x F(r,\tome(r))dr,\\
&&\quad \text{and}\quad \partial_{\tome_\tau}\partial_{\mu_{\tau}} f(t, \mu, \tome )=\int_\tau^t D^2_x F(r,\tome(r))dr.
\eeaa
Therefore, the functional $f$ also has continuous SVDs up to order $2$.

\item[$(iii)$]Let $F  \in C^1(\R^d)$ such that $|DF(x)| \le C(1+|x|)$ for some $C\ge 0.$ For fixed $t_0\in (0,T),$ consider
    $
    \Phi(T,\mu):= \E^\mu [ F( W(  t_0))] .
    $
      Then the SVD at $\mu_t$ is
$
\partial_{\mu_{t}} \Phi(T, \mu, \tome ):= D_x F(\tome(t_0)) 1_{[0,t_0]}(t).
$

\end{itemize}

\end{example}

\begin{example}
We consider non-anticipative functionals $f\in \FD$ by combining Example \ref{ex-path} and Example \ref{ex-meas}. For simplicity take $d=1$. Suppose that $F \in C^{1,2}_b([0,T] \times \R^5)$ and
$f_1,f_2,f_3,f_5 \in C^2_b(\R)$. $f_4 \in C^2_b(\R^2)$. Consider the following functional
\beaa
&f(t,\ome,\mu):= F\Big(t,\ome(t),\int_0^t f_1(\ome(r))dr , \E^\mu[f_2(W(t))], \E^\mu[\int_0^t f_3(W(r))dr],\\
 &\quad \quad \quad \quad E^\mu[f_4(W(t), \int_0^t f_5(W(r))dr  )] \Big), \quad \forall\  (t,\ome,\mu) \in \hD.
\eeaa
Then we check that $f$ has continuous horizontal derivatives and twice continuous SVDs in $ \ome_t$ and $ \mu_t.$ Indeed, for any $\tau \le t,$
\begin{equation*}
\begin{split}
&\partial_t f(t,\ome,\mu) = \partial_t F(t,x)+ \partial_{x_2}F(t,x) f_1(\ome(t))+ \partial_{x_4}F(t,x) \E^{\mu}\left[f_3(W(t))\right]\\
& \ \ \ \ \quad   \ \ \ \ \quad  \quad   + \partial_{x_5}F(t,x)\E^{\mu}\left[\partial_{y_2}f_4(Y)f_5(W(t))\right],\\
& \partial_{\ome_\tau}f(t,\ome,\mu)= \partial_{x_1}F(t,x) + \partial_{x_2} F(t,x) \int_\tau^t f_1'(\ome(r)) dr,\\
& \partial^2_{\ome_\tau}f(t,\ome,\mu)= \partial^2_{x_1}F(t,x) + \partial_{x_2}^2 F(t,x) \left(\int_\tau^t f_1'(\ome(r))  dr\right)^2  +\partial_{x_2} F(t,x) \int_\tau^t f_1^{(2)}(\ome(r)) dr,\\
& \partial_{\mu_\tau} f(t,\ome,\mu, \tome)= \partial_{x_3} F(t,x) f'_2(\tome(t)) + \partial_{x_4} F(t,x) \int_\tau^t f_3'(\tome(r))dr\\
& \ \ \ \ \quad \ \  \ \quad  \ \ \ \quad  \quad + \partial_{x_5} F(t,x) \Big[ \partial_{y_1}f_4 ( \ty) + \partial_{y_2} f_4 ( \ty ) \int_{\tau}^t f_5'(\tome(r))dr \Big],\quad \text{and}\\
& \partial_{\tome_\tau}\partial_{\mu_\tau} f(t,\ome,\mu, \tome) =  \partial_{x_3} F(t,x) f^{(2)}_2(\tome(t)) + \partial_{x_4} F(t,x) \int_\tau^t f_3^{(2)}(\tome(r))dr + \partial_{x_5} F(t,x) \Big[ \partial^2_{y_1}f_4 ( \ty ), \\
& \quad \ \ \quad \quad \ \quad  \quad \ \ \ \quad  \quad + 2 \partial_{y_2}\partial_{y_1} f_4 ( \ty ) \int_{\tau}^t f_5'(\tome(r))dr +\partial_{y_2}^2 f_4 ( \ty ) (\int_{\tau}^t f_5'(\tome(r))dr)^2 \Big],
\end{split}
\end{equation*}
where
\begin{align*}
(t,x)&=\left(t,\ome(t),\int_0^t f_1(\ome(r))dr , \E^\mu[f_2(W(t))], \E^\mu \Big[\int_0^t f_3(W(r))dr\Big],\
 \E^\mu\Big[f_4(W(t), \int_0^t f_5(W(r))dr  )\Big] \right),\\
  Y&= \left(W(t), \int_0^t f_5(W(r))dr  \right),\ \ \text{and} \quad
 \ty= \left( \tome(t),\int_0^t f_5(\tome(r))dr \right).
\end{align*}


\end{example}

\ \ \ \ \ \ In the following, for any $f\in \FD,$ we use generic notations $(\partial_{\ome}f, \partial^2_\ome f)$ (($\partial_{\ome_\tau}f, \partial_{\ome_\tau}^2 f)$, resp.) to denote the vertical derivative (SVD, resp.) in path, and $( \partial_{\mu}f, \partial_{\tome} \partial_{\mu}f)$ ($(\partial_{\mu_\tau}f, \partial_{\tome_\tau} \partial_{\mu_\tau}f)$, resp.) to denote the vertical derivative (SVD, resp.) in measure if there is no confusion. For product spaces $\hD \times \D$, $[0,T] \times \hD$ and $ [0,T]   \times \hD \times \mathbb{D}_{T,d} $, we equip them with the following premetrics respectively: for any $\bx:=(\tau,t, \ome, \mu, \tome)   $, $\bx':=(\tau', t', \ome', \mu', \tome') \in [0,T]   \times \hD \times \mathbb{D}_{T,d}$,
\be\label{productmetric}
\begin{split}
& d_m((t,\ome,\mu,\tome),(t',\ome',\mu',\tome')):= |t-t'|+ \|\ome_t- \ome'_{t'}\| + W_2(\mu_t,\mu'_{t'})+\|\tome_t-\tome'_{t'}\|,\\
&d_{sv}((\tau,t,\ome,\mu),(\tau',t',\ome',\mu')):=|\tau-\tau'|+|t-t'|+ \|\ome_t- \ome'_{t'}\| + W_2(\mu_t,\mu'_{t'}),\\
&d_{sm}(\bx, \bx'):=|\tau-\tau'|+|t-t'|+ \|\ome_t- \ome'_{t'}\| + W_2(\mu_t,\mu'_{t'})+\|\tome_t-\tome'_{t'}\|.
\end{split}
\ee

\begin{defi}\label{diff} Denote by $\FC( \hD )$ (or $\FC$ when there is no confusion), the subspace of $\FD$ which consists of all non-anticipative and continuous functionals with respect to the metric $d_p$ defined by \eqref{dp}. Furthermore,
\begin{itemize}
 \item[(i)] $\FC^{1,1,1}$ ($\FC^{1,1,1}_s$, resp.) is the subset of $\FC$ whose element  is continuously horizontally differentiable, (strongly, resp.) vertically differentiable w.r.t. both path and measure, with all derivatives being continuous with respect to the metric introduced in \eqref{productmetric}; 

     \item[(ii)]$\FC^{1,2,1}$ ($\FC^{1,2,1}_s$, resp.) is the subset of $ \FC^{1,1,1}$ ( $\FC^{1,1,1}_s$, resp.) whose element's  derivative $\partial_{\ome} f(t,\cdot,\mu,\tome) $ ( $\partial_{\ome_\tau} f(t,\cdot,\mu,\tome)$, $\tau \le t,$ resp.), $(t,\ome, \mu,\tome)\in \hD \times \D$, is further vertically differentiable (strongly vertically differentiable at $(\tau,t, \ome)$, resp.),  with all derivatives being continuous; 

      \item[(iii)]$\FC^{1,2,1,1}$ ($\FC^{1,2,1,1}_s$, resp.) is the subset of $  \FC^{1,2,1}$ ($\FC^{1,2,1}_s$, resp.) whose element's derivative functional $\partial_{\mu } f(t,\ome,\mu,\cdot) $  ( $\partial_{\mu_\tau} f(t,\ome,\mu,\cdot) $, $\tau \le t$, resp.), $(t,\ome,\mu, \tome)\in \hD \times \D$, is further vertically differentiable (strongly vertically differentiable at $(\tau,t,\tome)$, resp.), with all derivatives being continuous.
  \end{itemize}
Moreover, denote by $\FC^{1,1,1}_p$ the subset of $\FC^{1,1,1}$ such that the functional and all its first order derivatives have at most polynomial growth in the path variable: there exists $k \in \Z^+$, such that for $\phi=f, \partial_t f , \partial_{\ome } f,$ $\psi=\partial_{\mu }f $ and any $K>0,$

\be\label{polygrowth}
\begin{split}
 &|\phi(t,\ome, \mu) | \le C_{K} (1+\|\ome_t\|^k ),\ \ \
  |\psi(t,\ome,\mu,\tome)| \le C_{K} ( 1+ \|\ome_t \|^k+ \|\tome_t \|^k),\\
 &\ \ \ \ \forall   (t,\ome,\mu,\tome)\in \hD \times \D  \ \text{such that }|||\mu_t|||\le K,
\end{split}
\ee
for a constant $C_{K}$ depending only on $K.$ Notations such as $\FC_p,$ $\FC^{1,1,1}_{s,p}$ $\FC^{0,1,1}$ and $\FC^{1,2,1,1}_{p}$ are defined similarly.
 \end{defi}

\begin{rem}\label{state}
Assume that $f\in \FD$ is non-anticipative and has a state-dependent structure: $f(t,\ome,\mu)=\tf(t,\ome(t), \mu(t))$ for some function $\tf$ defined on $[0,T] \times \R^d \times \op_2(\R^d)$.  Then the horizontal differentiability and strongly vertical differentiability of $f$ is reduced to the differentiability of $\tf$ on $[0,T] \times \R^d \times \op_2(\R^d).$ Moreover,
\bea\nonumber
&&\partial_t f(t,\ome,\mu)= \partial_t \tf(t,\ome(t),\mu(t)), \ \  \partial_{\ome_\tau} f(t,\ome, \mu)= D_x f(t,\ome(t),\mu(t)),\quad \text{and}\\ \nonumber
&&\partial_{\mu_\tau}f(t,\ome,\mu,\tome)= \partial_{\nu} \tf(t,\ome(t),\mu(t), \tome(t)), \ \forall (t,\ome,\mu)\in [0,T] \times \D \times \op_2^D, \ \tau \le t,
\eea
where $\partial_{\nu} \tf$ is the Lions' derivative (see e.g. \cite{L}).
\end{rem}

\subsection{It\^o-Dupire formula}

Suppose that $(a,b)    $ is a bounded progressively measurable process on $(\Om, \MF, P) $ with values in $\R^m \times \R^{m\times d}.$ For any $(t, \gam)\in[0,T]\times \D,$ $X $ is the solution of SDE
\be\label{X}
\left\{
\begin{array}{l}
 d X(r)= a(r)dr + b(r) dB(r), \\
 X_t= \gam_t,\ \ \ r \ge t.
 \end{array}
 \right.
\ee
$(\Om', \MF' , P')$ is an atomless probability space with a $k$-dimensional Brownian motion $B'$ and $(c ,d ) $ is a bounded progressively measurable process on $(\Om', \MF' , P') $ with values in $  \R^n \times \R^{n \times k}.$ Given $\eta \in (\MM_2^D)' $, let $X' $ defined by SDE
\be\label{X'}
\left\{
\begin{array}{l}
 d X'(r)= c(r)dr + d(r) dB'(r), \\
 X'_t= \eta_t, \ \ \ r \ge t.
 \end{array}
 \right.
\ee
Moreover, let $({\tX}', \tc, \td, \tilde{B}', {\teta})$ be an independent copy of $(X',c,d,B',\eta)$, which means that $({\tX}', \tc, \td, \tilde{B}', {\teta})$ is defined in an independent probability space $(\tilde{\Omega},\tilde{\MF}, \tilde{P})$ from $(\Omega, \MF, P)$ and $(\Omega', \MF', P')$, and it has the same law as $(X',c,d,B',\eta)$.
Then we have the following It\^o-Dupire formula.

\begin{thm}\label{itoformula}
For any fixed $(t, \gam , \eta)\in [0,T] \times \D \times (\MM^D_2)',$ $X$ and $ X'$ are diffusion processes defined by \eqref{X} and \eqref{X'} respectively. Suppose that $f\in \FC^{1,2,1,1}_p(\hD)  $, and then we have

\begin{align}\nonumber
&f(s , X , \ML_{X'})- f(t,\gam, \ML_{\eta } )\\ \nonumber
&\ \ \ = \int_t^s \partial_{r}f(r, X, \ML_{X'}) dr + \int_t^s \partial_{\ome} f(r, X, \ML_{X'} ) d X(r)\\ \label{e-ito}
 &\ \ \ \ \ \ + \frac12  \int_t^{s } \text{Tr}\ [\partial_{\ome}^2 f(r, X_{r}, \ML_{X'} ) d\langle X\rangle(r) ]+ \E^{\tilde{P}'} [\int_t^s \partial_{\mu} f(r, X , \ML_{X'},\tilde{X}' )d \tilde{X}'(r) ] \\\nonumber
 &\ \ \ \ \ \ +  \frac12 \E^{\tilde{P}'} \int_t^{s } \text{Tr}\ [\partial_{\tome} \partial_\mu f(r, X, \ML_{X'} , \tilde{X}' ) \tilde{d}(r) \tilde{d}(r)^{ T} ] dr, \quad \forall s\ge t.
\end{align}

\end{thm}

\begin{proof}
Without loss of generality, assume $d=k=m=n=1$ and $s=T.$ Since both sides of identity \eqref{e-ito} depend on $(X',c,d, \eta)$ through its law, we assume that $(\Omega',\MF',P')$ is independent from $(\Omega, \MF, P)$ for simplicity of notations. Consider the following discretization of $X$ and $X':$ for any $n \ge 1, $ take $t=t_0 < t_1  < \cdots < t_n  =T$ as any partition of $[0,T]$ with vanishing modulus $\delta_n$. Define c\`adl\`ag processes $X^n, {X'}^n$ with $X^n_t = \gam_t, \ {X'}^n_t = \eta_t$ by
\begin{align*}
 & X^n(r):= \sum_{i=0}^{n-1} X(t_i) 1_{[t_i, t_{i+1}) }(r) + X(T)1_{\{T\}}(r),\\
 &    {X'}^n(r) := \sum_{i=0}^{n-1} X'(t_i) 1_{[t_i, t_{i+1}) }(r) + X'(T)1_{\{T\}}(r) , \quad r \ge t.
\end{align*}
Since $(a,b,c,d)$ is bounded, we see that for any $r\in[0,T],$
\begin{align}\label{con-xn}
&\E\|X^n \|^p_{\bS^p} \le \E\|X \|^p_{\bS^p} < \infty,  \quad  \lim_{n\rightarrow \infty}\|X^n_{t_i}- X_r\| = 0, \ P\text{-}a.s.,  \\ \label{con-x'}
&|||\ML_{{X'}^n}|||^2=\E\|{X'}^n \|^2_{\bS^2} \le \E\|X' \|^2_{\bS^2} < \infty,  \quad \lim_{n\rightarrow \infty} \|{X'}^n_{t_i}- X'_r\|  = 0, \ P'\text{-}a.s.,
\end{align}
where $i$ above satisfies $r\in[t_i,t_{i+1}).$ It follows from \eqref{con-x'} that
\be\label{con-dmu}
\lim_{n\rightarrow \infty} W_2(\ML_{X^{'n}_{t_i}}, \ML_{X'_r}) = 0.
\ee
Then we have
\be \label{dec}
\begin{split}
&f(T, X^{n}_T , \ML_{X^{'n}_T} )- f(t,  \gam_t , \ML_{ \eta_t })\\
 &\ \ \ = \sum_{i=0}^{n-1} [f(t_{i+1}, X_{t_{i+1}}^n, \ML_{(X^{'n})_{t_{i+1}}}) - f(t_{i }, X_{t_{i }}^n, \ML_{(X^{'n})_{t_{i }}}) ] \\
& \ \ \ =\sum_{i=0}^{n-1}\Big[ ( f(t_{i+1},X^{n}_{t_{i}}, \ML_{X^{'n}_{t_{i}}}) -  f(t_{i},X^{n}_{t_{i}}, \ML_{X^{'n}_{t_{i}}}) ) + ( f(t_{i+1},X^{n}_{t_{i+1}}, \ML_{X^{'n}_{t_{i}}})\\
&\ \ \ \ \ \ -  f(t_{i+1},X^{n}_{t_{i}}, \ML_{X^{'n}_{t_{i}}}) )
+ ( f(t_{i+1},X^{n}_{t_{i+1}}, \ML_{X^{'n}_{t_{i+1}}}) -  f(t_{i+1},X^{n}_{t_{i+1}}, \ML_{X^{'n}_{t_{i}}}) ) \Big].
\end{split}
\ee
\ \ \ \ \ \ Since
\be
\begin{split}
f(t_{i+1},X^{n}_{t_{i}}, \ML_{X^{'n}_{t_{i}}}) -  f(t_{i},X^{n}_{t_{i}}, \ML_{X^{'n}_{t_{i}}})  &= \int_{t_i}^{t_{i+1}} \partial_r f(r,X^{n}_{t_{i}}, \ML_{X^{'n}_{t_{i}}} ) dr\\
&  =\int_t^T\partial_r f(r,X^{n}_{t_{i}}, \ML_{X^{'n}_{t_{i}}} )1_{[t_i,t_{i+1})}(r) dr,
\end{split}
\ee
in view of inequalities \eqref{con-xn} and \eqref{con-dmu}, applying the dominated convergence theorem and passing to the limit for a subsequence, we have
\be\label{C}
\lim_{n\rightarrow \infty}\sum_{i=0}^{n-1}  \Big( f(t_{i+1},X^{n}_{t_{i}}, \ML_{X^{'n}_{t_{i}}}) -  f(t_{i},X^{n}_{t_{i}}, \ML_{X^{'n}_{t_{i}}}) \Big) = \int_t^T \partial_r f(r,X  , \ML_{X'  } ) dr, \ \ \ P\text{-}a.s..
\ee

\ \ \ \ \ \ For the second term on the right hand side of \eqref{dec},
since $f\in \FC^{1,2,1,1}_{p}, $ we have that $ \phi_i(\theta):= f(t_{i+1}, X^n_{t_{i}}+ \theta 1_{[t_{i+1}, T)}, \ML_{X^{'n}_{t_{i}}} )$ is twice continuously differentiable in $\theta$, and moreover,
\be
\phi'_i(\theta)= \partial_{\ome}f(t_{i+1}, X^n_{t_{i}}+ \theta 1_{[t_{i+1}, T)}, \ML_{X^{'n}_{t_{i}}}),\ \
\phi''_i(\theta)= \partial_{\ome}^2 f(t_{i+1}, X^n_{t_{i}}+ \theta 1_{[t_{i+1}, T)}, \ML_{X^{'n}_{t_{i}}}).
\ee
In the following, we will write $X_{i}:=X(t_i)\equiv X^n(t_i)$ and $\delta X_i:=X_{i+1}-X_i.$ Similar notations such as $X^{'}_i$ are self-explained.
Note that $X^n_{t_{i+1}}= X^n_{t_{i}}+ (X_{i+1}-X_i) 1_{ [t_{i+1}, T) }.$
Using the It\^o formula to $\phi(X(r)-X_i)$ on $r\in [t_i, t_{i+1}],$ we have
\begin{align}\nonumber
 &f(t_{i+1},X^{n}_{t_{i+1}}, \ML_{X^{'n}_{t_{i}}}) -  f(t_{i+1},X^{n}_{t_{i}}, \ML_{X^{'n}_{t_{i}}})\\ \label{lim1}
 &\ \ \ = \int_{t_i}^{t_{i+1}} \partial_{\ome} f(t_{i+1}, X^n_{t_i}+ (X(r)- X_i)1_{[t_{i+1}, T) }, \ML_{X^{'n}_{t_{i}}} ) d X(r)\\ \nonumber
&\ \ \ \ \ \ + \frac12 \int_{t_i}^{t_{i+1}} \partial_{\ome}^2 f(t_{i+1}, X^n_{t_i}+ (X(r)- X_i)1_{[t_{i+1}, T] }, \ML_{X^{'n}_{t_{i}}} ) d \lan X\ran (r).
\end{align}
Since $\| X^n_{t_i}+ (X(r)- X_i)1_{[t_{i+1}, T]} - X_r \|\rightarrow 0, \ P$-a.s. for any $r \in [t_i, t_{i+1})$, we have the following $P$-a.s. convergence   under the sup norm 
\[
\begin{split}
&\sum_{i=0}^{n-1} \partial_{\ome} f (t_{i+1},X^n_{t_i}+ (X(r)- X_i)1_{[t_{i+1}, T] }, \ML_{X^{'n}_{t_{i}}})1_{[t_i,t_{i+1})}(r)  \rightarrow \partial_{\ome} f (r, X, \ML_{X} ),	\\
&\sum_{i=0}^{n-1} \partial_{\ome}^2 f (t_{i+1},X^n_{t_i}+ (X(r)- X_i)1_{[t_{i+1}, T] }, \ML_{X^{'n}_{t_{i}}})1_{[t_i,t_{i+1})}(r)  \rightarrow \partial_{\ome}^2 f (r, X, \ML_{X} ),
\end{split}
\]
which implies $P$-a.s.
\[ 
\begin{split}
\int_t^T |\sum_{i=0}^{n-1} \partial_{\ome} f (t_{i+1},X^n_{t_i}+ (X(r)- X_i)1_{[t_{i+1}, T] }, \ML_{X^{'n}_{t_{i}}})1_{[t_i,t_{i+1})}(r)- \partial_{\ome} f (r, X, \ML_{X} )|^2 dr \rightarrow 0,\\
\int_t^T | \sum_{i=0}^{n-1} \partial_{\ome}^2 f (t_{i+1},X^n_{t_i}+ (X(r)- X_i)1_{[t_{i+1}, T] }, \ML_{X^{'n}_{t_{i}}})1_{[t_i,t_{i+1})}(r) - \partial_{\ome}^2 f (r, X, \ML_{X} )| dr \rightarrow 0.
\end{split}
\]
In view of the above convergence and identity \eqref{lim1}, passing to the limit in a subsequence, we have
\be \label{b}
\begin{split}
&\lim_{n\rightarrow \infty}\sum_{i=0}^{n-1} \Big( f(t_{i+1},X^{n}_{t_{i+1}}, \ML_{X^{'n}_{t_{i}}}) -  f(t_{i+1},X^{n}_{t_{i}}, \ML_{X^{'n}_{t_{i}}}) \Big)\\
&\ \ \ = \int_t^T \partial_\ome f(r, X , \ML_{X' } ) d X(r)  + \frac12  \int_t^T \partial_{\ome}^2 f(r, X  , \ML_{X' } ) d \lan X\ran (r),  \ \  P\text{-}a.s..
\end{split}
\ee

\ \ \ \ \ \ For the last term in the decomposition \eqref{dec}, we have
\begin{align*} \nonumber
&f(t_{i+1},X^{n}_{t_{i+1}}, \ML_{X^{'n}_{t_{i+1}}}) -  f(t_{i+1},X^{n}_{t_{i+1}}, \ML_{X^{'n}_{t_{i}}})\\
&\ \ \ = \int_0^1 \E' \Big[ \partial_{\mu} f(t_{i+1},X^{n}_{t_{i+1}}, \ML_{X^{'n}_{t_i}+ \theta (\delta X'_i) 1_{[t_{i+1, T})}},X^{'n}_{t_i}+ \theta (\delta X'_i) 1_{[t_{i+1, T})}  ) ( \delta X'_i)      \Big] d\theta\\
&\ \ \  = \int_0^1 \E' \Big[ \partial_{\mu} f(t_{i+1},X^{n}_{t_{i+1}}, \ML_{X^{'n}_{t_i}+ \theta ( \delta X'_i) 1_{[t_{i+1, T})}}, X^{'n}_{t_i} ) ( \delta X'_i)      \Big] d\theta\\ \nonumber
&\ \ \ \ \ \ +  \int_0^1 \int_0^1 \E' \Big[\partial_{\tome} \partial_{\mu} f(t_{i+1},X^{n}_{t_{i+1}}, \ML_{X^{'n}_{t_i}+ \theta ( \delta X'_i) 1_{[t_{i+1, T})}}, X^{'n}_{t_i}+ \lam \theta ( \delta X'_i)1_{[t_{i+1, T})} ) \theta ( \delta X'_i)^2      \Big] d\theta d \lambda. \\
\end{align*}
Since $\| X^{'n}_{t_i}+   \theta ( \delta X'_i)1_{[t_{i+1}, T)}- X'_r \| \rightarrow 0, \  P' $-a.s. for any $r\in[0 ,T]$
with $r\in[t_i,t_{i+1}]$,  we have
$$\lim_{n\rightarrow \infty} W_2(\ML_{X^{'n}_{t_i}+ \theta (\delta X'_i) 1_{[t_{i+1, T})}} , \ML_{X'_r}) = 0.$$
In view of \eqref{con-xn}, \eqref{con-dmu} and the dominated convergence theorem, we have
\begin{align*}
&\lim_{n\rightarrow \infty}\sum_{i=0}^{n-1} \int_0^1 \Big[\partial_{\mu} f(t_{i+1},X^{n}_{t_{i+1}}, \ML_{X^{'n}_{t_i}+ \theta ( \delta X'_i) 1_{[t_{i+1, T})}}, X^{'n}_{t_i} ) ( \delta X'_i)\Big] d\theta\\
&\ \ \ =  \int_t^T  \partial_{\mu} f(r , X^{\gam_t} , \ML_{X^{' } }, X^{'}  ) d X'(r), \  P \times P'\text{-}a.s..
\end{align*}
Then, according to Fubini's theorem, we have
\be \label{a1}
\begin{split}
&\lim_{n\rightarrow \infty} \sum_{i=0}^{n-1} \E'  \int_0^1 \Big[\partial_{\mu} f(t_{i+1},X^{n}_{t_{i+1}}, \ML_{X^{'n}_{t_i}+ \theta ( \delta X'_i) 1_{[t_{i+1, T})}}, X^{'n}_{t_i} ) ( \delta X'_i)\Big] d\theta\\
&\ \ \ = \E' [ \int_t^T  \partial_{\mu} f(r , X , \ML_{X^{' } }, X^{'}  ) d X'(r)], \ \ \ \ \ \ P\text{-}a.s..
\end{split}
\ee

\ \ \ \ \ \ By a similar argument as above, we have  
\be \label{a2}
\begin{split}
&\lim_{n\rightarrow \infty} \int_0^1\!\!\! \int_0^1 \E' \Big[\partial_{\tome} \partial_{\mu} f(t_{i+1},X^{n}_{t_{i+1}}, \ML_{X^{'n}_{t_i}+ \theta ( \delta X'_i) 1_{[t_{i+1, T})}}, X^{'n}_{t_i}+ \lam \theta ( \delta X'_i)1_{[t_{i+1, T})} ) \theta ( \delta X'_i)^2      \Big] d\theta d \lambda\\
&\ \ \ = \E'[ \int_t^T \partial_{\tome} \partial_\mu f(r, X, \ML_{X'} , X' )  dr  ],  \  P\text{-}a.s..
\end{split}
\ee
In view of \eqref{C}, \eqref{b}, \eqref{a1} and \eqref{a2},  taking $n \to \infty$ in \eqref{dec}, we obtain the desired identity.

\end{proof}

\ \ \ \ \ \ Note that $(\ome_s)_\tau=\ome_s$  and $(\mu_s)_\tau=\mu_s$  for any $\tau\ge s.$
In particular, if the non-anticipative functional $f$ is strongly vertically differentiable, we have the following partial It\^{o}-Dupire formula.

\begin{coro}\label{sito}
Suppose that $(X,  X')$ is defined as in Theorem \ref{itoformula} and $f  \in \FC^{0,2,1,1}_{s,p}(\hD)$.
Then we have that for any $t\le s \le v \le T,  $
\be \nonumber
\begin{split}
&f(v , X_s, \ML_{X'_s})- f(v,\gam_t, \ML_{\eta_t } )\\
 &\ \ \  =  \int_t^s \partial_{\ome_r} f(v, X_r, \ML_{X'_r} ) d X(r)+ \frac12  \int_t^{s } \text{Tr}\ [\partial_{\ome_r}^2 f(v, X_{r}, \ML_{X'_r} ) d\lan X \ran (r) ] \\ \nonumber
 &\ \ \ \ \ \ + \E^{\tilde{P}'} [\int_t^s \partial_{\mu_r} f(v, X_r , \ML_{X'}, \tilde{X}' )d \tilde{X}'(r) ] + \frac12 \E^{\tilde{P}'} \int_t^{s } \text{Tr}\ [\partial_{\tome_r} \partial_{\mu_r} f(v, X_r, \ML_{X'_r} , \tilde{X}'_r ) \tilde{d}(r) \tilde{d}(r)^T ] dr.
\end{split}
\ee

\end{coro}

\begin{proof}

Without loss of generality, assume $v=T.$ For any $r\in [t,s],$ let
\be
\tf(r,\ome,\mu):= f(T, \ome_r, \mu_r).
\ee
Obviously, $\tf$ is non-anticipative, and moreover, we have that for any $h\ge0,$
$$
\tf(r+h, \ome_r, \mu_r)=f(T, (\ome_r)_{r+h}, (\mu_r)_{r+h} )= f(T, \ome_r , \mu_r ) =\tf(r,\ome_r, \mu_r),
$$
which implies $\partial_r \tf(r,\ome_r, \mu_r )=0.$ Furthermore, it follows from definitions of vertical derivatives and strongly vertical derivatives that
\beaa
&\partial_{\ome} \tf(r,\ome,\mu)= \partial_{\ome_r} f(T, \ome_r, \mu_r), \ \  \ & \partial_{\ome}^2 \tf(r,\ome,\mu)= \partial_{\ome_r}^2 f(T, \ome_r, \mu_r),\\
&\partial_{\mu} \tf(r,\ome,\mu,\tome)= \partial_{\mu_r} f(T, \ome_r, \mu_r,\tome), \ \ & \text{and}\quad \partial_{\tome}\partial_{\mu} \tf(r,\ome,\mu,\tome)=  \partial_{\tome_r}\partial_{\mu_r}  f(T, \ome_r, \mu_r,\tome_r).
\eeaa
Applying Theorem \ref{itoformula} to $\tf(r,X,\ML_{X'})$ on $r\in[t,s]$, and we obtain the desired formula.

\end{proof}


\section{Solution of  semilinear path-dependent master equations }

\ \ \ \ \ \ In this section we show the well-posedness of \eqref{ppde-1}, during which we will exploit the regularity of corresponding FBSDEs (see Section 4.2). We leave the detailed proof of such regularity in Section 4.

\ \ \ \ \ \ To build smooth solutions to path-dependent mean-field PDE
\bea\label{pde-ext'}
\left\{
\begin{array}{l}
\partial_t u(t,\gam, \mu ) + \frac12 \text{Tr} \left[ \partial_{\ome}^2 u(t,\gam,\mu)\sigma_1(\gam_t)\sigma_1(\gam_t)^T  \right]+ \partial_{\ome}u(t,\gam,\mu) b_1(\gam_t)\\[2mm]
 \quad  +\frac12 \text{Tr} \left[ \E^{P}[\partial_{\tome} \partial_\mu u(t,\gam, \mu, \eta )]\sigma_2(\mu_t)\sigma_2(\mu_t)^T \right]+ \E^{P}[\partial_{\mu} u(t,\gam,\mu,\eta ) ] b_2(\mu_t)\\[2mm]
\quad + f(t,\gam, u(t,\gam,\mu),\sigma_1(\gam_t) \partial_\ome u(t,\gam,\mu), \mu, \ML_{u(t,\eta,\mu)})=0,\\
 \\
 u(T, \gam, \mu)= \Phi(\gam_T, \mu_T), \ \ \ (t,\gam,\mu) \in [0,T]\times \mC \times \op_2^C,
\end{array}
\right.
\eea
we firstly need to study the case when $(b_1,\sigma_1)=(b_2,\sigma_2)=(0,I)$. In the following, we usually write $f(\ome_t ,\mu_t ):=f(t,\ome, \mu )$ for simplicity when $f$ is non-anticipative. 

\ \ \ \ \ \ To consider the regularity of terminal functional $\Phi$. Let product spaces $[0,T] \times \D \times \op^D_2$ and $[0,T] \times \D \times  \op^D_2 \times \D $ be equipped with the following metrics respectively: for any $\bx:=(\tau,  \ome, \mu, \tome)   $, $\bx':=(\tau',  \ome', \mu', \tome') \in [0,T]   \times \D \times   \op^D_2 \times \mathbb{D}_{T,d}$,
\be\label{T-productmetric}
\begin{split}
&d_{T,sv}((\tau ,\ome,\mu),(\tau' ,\ome',\mu')):=|\tau-\tau'|+  \|\ome_T- \ome'_{T}\| + W_2(\mu_T,\mu'_{T}),\\
&d_{T,sm}(\bx, \bx'):=|\tau-\tau'| + \|\ome_T- \ome'_{T}\| + W_2(\mu_T,\mu'_{T})+\|\tome_T-\tome'_{T}\|.
\end{split}
\ee

\begin{defi}
We write $\Phi \in \FC_T(\hD)$ (or $\FC_T$ if no confusion raised) if $\Phi:\D \times \op_2^D \mapsto \R $ is continuous on $\D \times \op_2^D.$ Furthermore, we write
\begin{itemize}
\item[(i)] $\Phi \in \FC_{T,lip}$ if it is uniformly Lipschitz continuous on $\D \times \op_2^D$:
    $$
| \Phi(\ome_T, \mu_T)- \Phi(\ome'_T, \mu'_T) |\le C (\|\ome_T-\ome'_T\|+ W_2(\mu_T, \mu'_T)), \ \forall (\ome,\mu), (\ome',\mu') \in \D \times \op_2^D,
$$
 for some constant $C \ge 0$;

\item[(ii)] $\Phi \in \FC^{1,1}_{T,lip}$ if $\Phi \in \FC_{T,lip}$ and its SVDs $\partial_{\ome_\tau}\Phi$ and $\partial_{\mu_\tau}\Phi$  is continuous under the metric introduced in \eqref{T-productmetric} respectively. Moreover, SVDs are uniformly Lipschitz continuous with respect to $\tau \in [0,T]$ in $(\ome,\mu)\in \D \times \op_2^D$ and $(\ome,\mu,\tome) \in \D \times \op_2^D \times \D $, respectively;

\item[(iii)] $\Phi\in \FC^{2,1,1}_{T,lip}$ if $\Phi \in \FC^{1,1}_{T,lip}$ and for any $(\tau, \ome,\mu,\tome) \in \hD \times \D,$ its SVDs $\partial_{\ome_\tau}\Phi( \cdot,\mu_T)$ and $\partial_{\mu_\tau}\Phi(\ome_T,\mu_T,\cdot)$ are continuously strongly vertically differentiable at $(\tau,T, \ome)$ and $(\tau,T, \tome)$ under the metric $d_{T,sv}$ and $d_{T,sm}$, respectively. Moreover, all second-order derivatives are uniformly Lipschitz continuous with respect to the time parameter.

\end{itemize}

\end{defi}

\ \ \ \ \ \ To obtain the classical solution to \eqref{pde-ext'} with $(b_1,\sigma_1)=(b_2,\sigma_2)=(0,I),$ we introduce the following increasingly stringent assumptions.

\begin{itemize}
\item[ $\textbf{(H0)}$]
 $(i)$ The functional $\Phi\in \FC_{T,lip}(\hD)$ ; \
 $(ii)$ $f $ is a non-anticipative continuous function on $ [0,T] \times \D \times \R \times \R^d \times \op_2^D \times \op_2(\R)  $, and for any $(t,\ome,\mu)\in[0,T]\times \D \times \op_2^D,$ $f(t,\ome,\cdot,\cdot,\mu,\cdot)$ is continuously differentiable on $\R\times \R^d \times \op_2(\R)$. Moreover, for any $t \in [0,T]$, $f(t, \cdot, \cdot, \cdot , \cdot , \cdot )$ and $\partial_\nu f(t, \cdot, \cdot, \cdot , \cdot , \cdot, \cdot )$ are uniformly Lipschitz continuous.

\end{itemize}

\begin{itemize}
\item[ $\textbf{(H1)}$]
 $(i)$\ The functional $\Phi \in \FC^{1,1}_{T,lip}(\hD)$;
 $(ii)$\
  $f $ is a non-anticipative continuous function on $ [0,T] \times \D \times \R \times \R^d \times \op_2^D \times \op_2(\R)  $, and for any $(t,\ome,\mu)\in[0,T]\times \D \times \op_2^D,$ $f(t,\ome,\cdot,\cdot,\mu,\cdot)$ is  differentiable on $\R\times \R^d \times \op_2(\R)$
  with bounded derivatives. For any $( y,z,\nu) \in \R \times \R^d \times \op_2(\R),$ $f(t,\ome, y,z, \cdot, \nu)$ is strongly vertically differentiable at $\mu_t$ and $f(t,\cdot, y,z, \mu, \nu)$ is strongly vertically differentiable at $\ome_t$.
  Moreover, $\partial_{(y,z,\nu, {\ome_\tau}, {\mu_\tau})} f$ is continuous, and for any $\tau \le t,$ $(I,\partial_{(y,z,\nu, {\ome_\tau}, {\mu_\tau})}) f(t,\cdot )$ is uniformly Lipschitz continuous.

\end{itemize}

\begin{itemize}
\item[ $\textbf{(H2)}$]
 $(i)$\ $\Phi \in \FC^{2,1,1}_{T,lip}(\hD)$;
 $(ii)$\  $f:[0,T] \times \D \times \R \times \R^d \times \op_2^D \times \op_2(\R) \mapsto \R$ satisfies Assumption {\bf(H1)}(ii). Moreover, for any $(t,\ome,y,z,\mu,\nu)\in [0,T] \times \D\times \R \times \R^d \times \op_2^D \times \op_2(\R),$
 $(\partial_y f(t,\ome,\cdot,\cdot, \mu,\nu)  ,\ \partial_z f(t,\ome,\cdot,\cdot, \mu,\nu))$ is differentiable on $\R\times \R^d$; $(\partial_y f(t,\cdot,y,z,\mu,\nu)$,\ $\partial_z f(t,\cdot,y,z,\mu,\nu))$ is strongly vertically differentiable at $ (t,\ome)$; for any $\tau \le t$, $\partial_{\ome_\tau}f(t,\cdot,y,z,\mu,\nu)$ is differentiable at $(\tau,t,\ome)$; $\partial_\nu f(t,\ome,y,z,\mu,\nu,\cdot)$ is differentiable on $\R$; for any $\tome \in \D,$ $\partial_{\mu_\tau} f(t,\ome,y,z,\mu,\nu,\cdot)$ is differentiable at $(\tau,t,\tome)$. All second order derivatives are continuous and $(\partial_{y}^2, \partial_y \partial_z, \partial_z^2, \partial_{\ome_\tau}^2, \partial_{\ty}\partial_{\nu}, \partial_{\tome_\tau}\partial_{\mu_\tau})f(t,\cdot)$ is uniformly Lipschitz continuous.

\end{itemize}

\subsection{The decoupling field and its regularity}

\ \ \ \ \ \ Assume that {\bf(H2)} holds for $(\Phi,f).$ Recall for any $\gam, \ome \in \D,$  $\ome^{\gamma_t} \in \D$ with
\be
\ome^{\gamma_t}(\cdot):= \gamma_t(\cdot) + (\ome(\cdot)-\ome(t)) 1_{[t,T]}(\cdot).
\ee
For any $(t,\gam,\mu)\in\hD,$
let $Y^{\gl} $ solve the path-dependent BSDE
\be \label{e2'}
\begin{split}
Y^{\gamma_t, \eta_t}{(s)}&= \Phi(B_T^{\gamma_t}, \ML_{B_T^{\eta_t} }) + \int_s^T f( B_r^{\gamma_t}, Y^{\gamma_t, \eta_t}(r), Z^{\gamma_t,\eta_t}(r) ,  \ML_{B_r^{\eta_t}}, \ML_{Y^{\eta_t}(r)}  ) dr\\
&\ \ \ \ -  \int_s^T Z^{\gamma_t,\eta_t}(r) dB(r), \ \ \ \ \ \ s\in [t,T],
\end{split}
\ee
where $Y^{\eta_t}$ is the unique solution of the mean-field BSDE
\be \label{e1'}
\begin{split}
Y^{\eta_t}(s)&= \Phi(B_T^{\eta_t}, \ML_{B_T^{\eta_t} }) + \int_s^T f( B_r^{\eta_t}, Y^{\eta_t}(r),  Z^{\eta_t}(r)  , \ML_{B_r^{\eta_t}},  \ML_{Y^{\eta_t}(r) }) dr\\
&\ \ \ \ - \int_s^T Z^{\eta_t}(r) dB(r),\ \ \ \ \ \  \ s\in [t,T].
\end{split}
\ee
According to Lemma \ref{unibd} and Remark \ref{law-dep}, we know that $Y^{\gamma_t, \eta_t}(t)=Y^{\gamma_t, \ML_{\eta_t}}(t).$
For any $(t,\gam,\mu)\in\hD$ with $\mu=\ML_{\eta}$, define the decoupling field
\be\label{decop}                                                u(t,\gam, \mu ): = Y^{\gamma_t, \ML_{\eta_t}}(t).
\ee
By the well-posedness of \eqref{e2'} and \eqref{e1'}, we see that $u \in \FD$ and it is non-anticipative. On the other hand, for any $ v \ge t$, $(\ome,y,z) \in \D \times \R \times \R^d,$ denote
\be
\hat{\Phi}_{\mu_t}(\ome_T):= \Phi(\ome_T, \ML_{B^{\eta_t}_T} ),\ \ \ \hat{f}_{\mu_t}(v,\ome, y, z ):= f( v, \ome, y, z , \ML_{B_{v}^{\eta_t}}, \ML_{Y^{ \eta_t}({v})} ).
\ee
Let $\hY^{\gamma_{v}, \ML_{\eta_t}}$ is the unique solution of the following (path-dependent) BSDE: for $s \ge v,$
\be\label{hatu}
\hY^{\gamma_{v}, \eta_t}{(s)}= \hat{\Phi}_{\mu_t}(B_T^{\gamma_{v}}) + \int_s^T \hf_{\mu_t}(r, B_r^{\gamma_{v}}, \hY^{\gamma_{v}, \eta_t}(r), \hZ^{\gamma_{v},\eta_t}(r)   ) dr-  \int_s^T \hZ^{\gamma_{v},\eta_t}(r) dB(r).
\ee
According to Remark \ref{H2} and \cite[Theorem 3.9]{PW16}, there exists a non-anticipative mapping
$
\hu_{\mu_t}: [v,T] \times \D \mapsto \R,
$
such that for any $s\ge v,$
\be\label{uflow}
\hu_{\mu_t}(s,B^{\gam_{v}} ) =  \hY^{\gam_{v}, \mu_t}(s), \ \ \ \partial_{\gam_{v}} \hu_{\mu_t}(s,B^{\gam_{v}} ) = \hZ^{\gam_{v}, \mu_t}(s).
\ee
Moreover, $\hu_{\mu_t}$ is the classical solution of the following semilinear PPDE
\be\label{e-hu}
\left\{
\begin{array}{l}
\partial_v \hu_{\mu_t} (v, \gam ) +  \frac12 \text{Tr}\ [\partial^2_{\ome_v}\hu_{\mu_t}(v,\gam)] + \hf_{\mu_t}(v,\gam, \hu_{\mu_t}(v,\gam), \partial_{\ome_v}\hu_{\mu_t}(v,\gam) )=0, \\[2mm]
\hu_{\mu_t}(T, \gam)= \hPhi(\gam), \ \ \ v \ge t.
\end{array}
\right.
\ee
Indeed, denote $\hat{\eta}:= B^{\eta_t},$ and we have
\be
\hu_{\mu_t}(v, \gam)= Y^{\gam_{v}, \hat{\eta}_{v}}(v).
\ee
Concerning the relation among $Y^{\gl},$ $u(t,\gam, \mu)$ and $\hu_{\mu_t} (v, \gam ),$ we have

\begin{prop}\label{flow}
Assume that {\bf(H2)} holds for $(\Phi,f).$
For any $(t, \gam, \mu) \in \hD $ and $s\ge   t,$
\bea\label{u=hu}
&&u(t,\gam, \mu)= \hu_{\mu_t}(t,\gam), \\ \label{us=hu}
&& u(s,\ome_s^{\gam_t}, \ML_{B^{\eta_t}_s})=\hu_{\mu_t}(s,\ome^{\gam_t}_s) ,\ \ \forall \ome \in \mC,   \\ \label{us=ys}
&&  u(s,B_s^{\eta_t}, \ML_{B^{\eta_t}_s}) = Y^{\eta_t}(s).
\eea

\end{prop}

\begin{proof}

The first identity follows immediately from \eqref{decop} and \eqref{uflow}. By the uniqueness of BSDE \eqref{e1'}, we see that for any $t \le v \le s,$
\be
(B^{B^{\eta_t}_v}_s, Y^{B^{\eta_t}_v}(s) ) = (B^{\eta_t}_s , Y^{\eta_t}(s) ),
\ee
and in particular
\be
\ML_{(B^{B^{\eta_t}_v}_s, Y^{B^{\eta_t}_v}(s) )} = \ML_{(B^{\eta_t}_s , Y^{\eta_t}(s) ) } .
\ee
Then in view of the uniqueness of solutions of BSDE \eqref{hatu} and definition \eqref{uflow}, we have
\be
\hu_{\mu_t}(s, B_s^{\gam_t}) = \hu_{\ML_{B_v^{\eta_t}}} (s, B_s^{\gam_t}) .
\ee
In particular if $v=s,$ $\hu_{\mu_t}(s, B_s^{\gam_t}) = \hu_{\ML_{B_s^{\eta_t}}} (s, B_s^{\gam_t}) .$
On the other hand, by relation \eqref{u=hu}, we have
\be
\hu_{\ML_{B_s^{\eta_t}}} (s, B_s^{\gam_t})= u(s,B_s^{\gam_t}, \ML_{B_s^{\eta_t}} ),
\ee
and thus \eqref{us=hu} in view of the continuity of $u(t,\ome, \mu)$ in $\ome \in \mC$ (see Lemma \ref{unibd}) and the support theorem for diffusion processes.
Finally, since
\bea
Y^{\eta_t}(s)= Y^{\gam_t, \eta_t}(s)|_{\gam=\eta}=\hu_{\mu_t}(s, B^{\gam_t}_s )|_{\gam=\eta} = \hu_{\mu_t}(s, B^{\eta_t}_s ), 
\eea
 identity \eqref{us=ys} follows from  \eqref{us=hu}.

\end{proof}

\ \ \ \ \ \ To show that $u$ given by \eqref{decop} provide a smooth solution to the path-dependent mean-field PDE, we need the following regularity of $u$, which is a result of regularity of corresponding BSDEs proved in Section 4.

\begin{prop}\label{regular}

Suppose that $(f,\Phi)$ satisfies Assumption {\bf(H2)}. The decoupling field $u$ given by \eqref{decop} belongs to $ \FC^{0,2,1,1}_{s,p}(\hD)$.

\end{prop}

\begin{proof}

According to Lemma \ref{unibd}, $u(t,\gam, \mu)= Y^{\gl}(t)$ satisfies the polynomial growth condition in the sense of \eqref{polygrowth}. To prove that $u \in \FC_{p}$, we only need to show its continuity in $(t, \gam, \mu) \in \hD.$ For any $(t, \gam, \mu), (t', \gam', \mu')\in \hD,$ without loss of generality, assume $t \ge t'.$ We have
\begin{align} \nonumber
|  u(t, \gam, \mu)-   u(t', \gam', \mu')|&=|  Y^{\gam_{t}, \eta_{t}}(t)- Y^{\gam'_{t'}, \eta'_{t'}}(t') |\\  \nonumber
&\le  \E| Y^{\gam_{t}, \eta_{t}}(t)- Y^{\gam'_{t'}, \eta'_{t'}}(t)  | +\E| Y^{\gam'_{t'}, \eta'_{t'}}(t)- Y^{\gam'_{t'}, \eta'_{t'}}(t')  |\\  \nonumber
&\le   \E| Y^{\gam_{t}, \eta_{t}}(t)- Y^{\gam'_{t'}, \eta'_{t'}}(t)  | + \E| \int_{t'}^t f(\Theta^{\gamp_{t'},\etap_{t'} }_r, \ML_{\Theta^{\etap_{t'}}_r} ) dr-\int_{t'}^t Z'(r) dB(r)|\\ \label{162}  
&\le  \E| Y^{\gam_{t}, \eta_{t}}(t)- Y^{\gam'_{t'}, \eta'_{t'}}(t)  | + C(1+\|\gam_t\|+\|\eta_t\|_{\bS^2})(t-t')^{\frac12}.
\end{align}
It remains to prove $\E| Y^{\gam_{t}, \eta_{t}}(t)- Y^{\gam'_{t'}, \eta'_{t'}}(t)  | \rightarrow 0$ as $(t,\gam,\mu) \rightarrow (t',\gam', \mu').$ Set $Y':= Y^{\gam'_{t'}, \eta'_{t'}}$, $Y:=Y^\gl$, and $(\delta  Y, \delta Z):= (Y-Y', Z-Z')$, and omit subscripts $t$ and $t'.$
Then $(\delta Y, \delta Z)$ is the unique solution of BSDE
\beaa\nonumber
&& \delta Y (s) =  \Phi(B^{\gam}, \ML_{B^{\eta}})-\Phi(B^{\gam'}, \ML_{B^{\eta'}}) + \int_s^T [f(\Theta^{\gam,\eta}_r, \ML_{\Theta^{\eta}_r} ) - f(\Theta^{\gam',\eta'}_r, \ML_{\Theta^{\eta'}_r} ) ] dr
 - \int_s^T \delta Z(r) d B(r) \\
&&\ \ \ \ \ \ \ \ = : \delta \Phi + \int_s^T \Big( a_r \delta Y(r) + b_r \delta Z(r) + \tE[ \tilde{c}_r \delta \tY(r)] + \delta h_r \Big) dr  - \int_s^T \delta Z(r) d B(r) ,
\eeaa
where
\beaa
&&a_r:= \int_0^1 \partial_y {f} (B^{\gam}_r , Y'+ \theta(Y-Y'), Z, \ML_{\Theta^\eta_r} ) d\theta, \\
&&b_r:=  \int_0^1 \partial_z {f} (B^{\gam}_r , Y' , Z+\theta(Z-Z'), \ML_{\Theta^\eta_r} ) d\theta ,\\
&&\tilde{c}_r:= \int_0^1 \partial_{\nu} f ( B^{\gam}_r , Y' , Z', \ML_{B_r^{\eta}}, \ML_{Y^{\etap} + \theta(Y^\eta -Y^{\etap} )} ,   \tY^{\etap}+ \theta(\tY^\eta -\tY^{\etap} ) ) d \theta,\quad \text{and} \quad\\
&& \delta h_r: = f( B^{\gam}_r ,  Y', Z', \ML_{B^{\eta}_r} , \ML_{Y^{\etap}}  ) - f( B^{\gam'}_r ,  Y', Z', \ML_{B^{\eta'}_r} , \ML_{Y^{\etap}}  ) .
\eeaa
Applying Lemma \ref{linear} to the above BSDE, we have
\begin{align*}
 \| (\delta Y , \delta Z) \|_{\bS^2 \times \bH^2}^2 &\le C(\|\delta \Phi \|_{L^2}^2 + \|\int_t^T|\delta h_r| dr \|_{L^2}^2   )\\
&\le  C (\| \|B^{\gam}  - B^{\gam'} \| \|_{L^2}^2 + W_2(\ML_{B^\eta}, \ML_{B^{\etap}})^2 )\\
&\le  C (\|\gam_t -\gam'_{t'}\|^2 + W_2(\mu_t, \mu'_{t'})^2+ (t-t') ),
\end{align*}
where $C$ depends on $|||\mu_t|||+ |||\mu'_{t'}|||, $ $\|\gam_t\|$ and $ \|\gam'_{t'}\|,$ and thus the continuity of $u$ in view of \eqref{162}.

\ \ \ \ \ \ Since $u(t, \gam, \mu)= Y^{\gl}(t)=\hu_{\mu_t}(t,\gam),$ according to Proposition \ref{s-bdd1}, we see that for any $\tau \le t,$ $u(t, \gam, \mu)$ is twice strongly vertically differentiable at $(\tau,t,\gam)$, and moreover, $\partial_{\ome_\tau}u(t, \gam, \mu)=\partial_{\ome_\tau} Y^{\gl}(t)$ satisfies the polynomial growth condition. To show $ u(t, \gam, \mu) \in \FC_{s,p}^{0,1,0},$ we only need to prove that $\partial_{\gam_\tau}u(t,\gam,\mu)$ is continuous at any $(\tau,t,\gam,\mu).$ Indeed, for any $(\tau,t,\gam,\mu)$ and $(\tau',t',\gam',\mu')$ with $\tau \le t$, $\tau'\le t'$, denote solutions of equation \eqref{s-e4} corresponding to parameters $(\tau,t,\gam,\mu)$ and $(\tau',t',\gam',\mu')$ by
\be
(\MY,\MZ):=(\partial_{\ome_{\tau}}Y^{\gl}, \partial_{\ome_{\tau}}Z^{\gl} ),\ \ \   (\MY',\MZ'):=(\partial_{\ome_{\tau'}}Y^{\gam'_{t'},\eta'_{t'}}, \partial_{\ome_{\tau'}}Z^{\gam'_{t'},\eta'_{t'}} ).
\ee
Without loss of generality, let $t\ge t'$. By inserting the term $\MY'(t)$ and applying Proposition \ref{s-bdd1}, we have
\begin{align*}
&|  \partial_{\ome_\tau}u(t, \gam, \mu)-   \partial_{\ome_{\tau'}} u(t', \gam', \mu')|\\
&\quad =|  \partial_{\ome_\tau}Y^{\gam_{t}, \eta_{t}}(t)- \partial_{\ome_{\tau'}}Y^{\gam'_{t'}, \eta'_{t'}}(t') |\\
&\quad \le  \E| \partial_{\ome_\tau}Y^{\gam_{t}, \eta_{t}}(t)- \partial_{\ome_{\tau'}}Y^{\gam'_{t'}, \eta'_{t'}}(t)  | +\E| \partial_{\ome_{\tau'}}Y^{\gam'_{t'}, \eta'_{t'}}(t)- \partial_{\ome_{\tau'}}Y^{\gam'_{t'}, \eta'_{t'}}(t')  |\\
&\quad \le  \E| \partial_{\ome_\tau}Y^{\gam_{t}, \eta_{t}}(t)- \partial_{\ome_{\tau'}}Y^{\gam'_{t'}, \eta'_{t'}}(t)  | +  C(t-t')^{\frac12}.
\end{align*}
Set $(\delta \MY, \delta \MZ):=(\MY-\MY', \MZ-\MZ')$.  We see that $(\delta \MY, \delta \MZ)$ is the unique solution of the following BSDE
\begin{align*}
\delta \MY(s)=&\ [\partial_{\ome_\tau}\Phi- \partial_{\ome_{\tau'}}\Phi']+ \int_s^T[\partial_{\ome_\tau}f- \partial_{\ome_{\tau'}}f']dr + \int_s^T \partial_y f \delta \MY(r) dr\\
&+ \int_s^T (\partial_y f -\partial_y f')\MY'(r) dr  + \int_s^T (\partial_z f)^T \delta \MZ(r) d r+ \int_s^T (\partial_z f -\partial_z f')^T \MZ'(r) dr \\
&- \int_s^T \delta \MZ dB(r),
\end{align*}
where
\begin{align*}
&\partial_{\ome_\tau}\Phi:= \partial_{\ome_\tau}\Phi(B^{\gam_t}, \ML_{B^{\eta_t}}), \quad
\partial_{\ome_{\tau'}}\Phi':= \partial_{\ome_{\tau'}}\Phi(B^{\gam'_{t'}}, \ML_{B^{\eta'_{t'}}}),\quad
\partial_{(\ome_\tau,y,z)}f:= \partial_{(\ome_\tau,y,z)}f(\Theta^{\gl}, \ML_{\Theta^{\eta_t}}),\\
 & \quad \text{and} \quad \partial_{(\ome_{\tau'},y,z)}f':= \partial_{(\ome_{\tau'},y,z)}f(\Theta^{\gam'_{t'}, \eta'_{t'}}, \ML_{\Theta^{\eta'_{t'}}}).
 \end{align*}
In view of estimates in Lemma \ref{linear} and Proposition \ref{s-bdd1}, using Cauchy inequality, we have
\begin{align*}
|\MY(t)-\MY'(t)|^2 &\le \Big[\| \partial_{\ome_\tau}\Phi- \partial_{\ome_{\tau'}}\Phi'   \|^2_{L^2} + \| \int_t^T[\partial_{\ome_\tau}f- \partial_{\ome_{\tau'}}f']dr  \|^2_{L^2} \\
&\ \ \ \  + \| \int_t^T[\partial_{y}f- \partial_{y}f']\MY'dr  \|^2_{L^2} + \| \int_t^T[\partial_{z}f- \partial_{z}f']\MZ'dr  \|^2_{L^2} \Big]\\
&\le  \Big[\| \partial_{\ome_\tau}\Phi- \partial_{\ome_{\tau'}}\Phi'   \|^2_{L^2} +  \E[\int_t^T|\partial_{\ome_\tau}f- \partial_{\ome_{\tau'}}f'|^2dr ] \\
&\ \ \ \  + \E[ \int_t^T|\partial_{y}f- \partial_{y}f'|^4dr ] + \E[ \int_t^T|\partial_{z}f- \partial_{z}f'|^4dr ] \Big].
\end{align*}
Then the desired continuity follows from that of $(\partial_{\ome_\tau}\Phi,\partial_{(\ome_\tau,y,z)}f )$ and the bounded convergence theorem.
Similarly, we have $\partial_{\ome_\tau}u \in \FC_{s,p}^{0,1,0}$ and therefore $u\in \FC_{s,p}^{0,2,0}$.

\ \ \ \ \ \ For the differentiability with respect to the measure variable, according to Lemmas \ref{s-gat1} and \ref{s-dxu}, we have that for any $x\in \D,$
\be
\partial_{\mu_\tau} u(t, \gam , \mu, x) = \partial_{\mu_\tau} Y^{\gam_t, \mu_t, x_t }(t),\ \quad
\partial_{\tome_\tau} \partial_{\mu_\tau}u(t,\gam,\mu,x)=\partial_{\tome_\tau} \partial_{\mu_\tau} Y^{\gam_t, \mu_t, x_t }(t).
\ee
Here, $\partial_{\mu_\tau} Y^{\gam_t, \mu_t, x_t }$ solves BSDE \eqref{s-e-u} and $\partial_{\tome_\tau} \partial_{\mu_\tau} Y^{\gam_t, \mu_t, x_t }$ solves BSDE \eqref{e-durx2}. Following a similar argument as above, we see that $  u(t, \gam , \mu ) \in \FC^{0,2,1,1}_{s,p}.$

\end{proof}

\subsection{Solution of BSDEs as solution of path-dependent mean-field PDEs}\label{sol-to}
In this subsection we consider well-posedness of the path-dependent mean-field PDE
\bea\label{pde}
\left\{
\begin{array}{l}
\partial_t u(t,\gam, \mu ) + \frac12 \text{Tr} \left[ \partial_{\ome}^2 u(t,\gam,\mu)\right] + \frac12\text{Tr} \left[  \int_{\mC} \partial_{\tome} \partial_\mu u(t,\gam, \mu, \tome )\mu(d\tome)\right] \\[2mm]
\ \ \ + f(t,\gam, u(t,\gam,\mu), \partial_\ome u(t,\gam,\mu), \mu, \ML_{u(t,W^\mu,\mu)})=0,\\
 \\
 u(T, \gam, \mu)= \Phi(\gam_T, \mu_T), \quad \quad (t,\gam,\mu) \in [0,T] \times \mC \times \op_2^C,
\end{array}
\right.
\eea
where we recall that $W^\mu$ is the canonical process under $\mu.$ In applications, $(\gam, \mu)$ takes values in $\mC \times \op_2^C.$ Thus we need to give a description of equation \eqref{pde} restricted on $\mC \times \op_2^C.$ Denote by $\hC$ the product space $[0,T]\times \mC \times \op_2^C,$ and for a $n\times n$ matrix $A,$ we write $\text{Sym}(A):= \frac12 (A+ A^T).$ For  any $f\in \FD$, we write $ (\partial_{\ome_{\tau}}, \partial_{\ome_{\tau}}^2, \partial_{\mu_\tau}, \partial_{\tome_\tau}\partial_{\mu_\tau}) f:=(\partial_{\ome_\tau} f, \partial_{\ome_{\tau}}^2 f, \partial_{\mu_\tau} f, \partial_{\tome_\tau}\partial_{\mu_\tau}f )$ if the right hand side exists.

\begin{defi}\label{f-ext} Denote by $\FC^{1,2,1,1}_{s,p}(\hC)$ the set of functionals
 $f :\hC \mapsto \R$ such that there exists an extension $F \in \FC^{1,2,1,1}_{s,p}(\hD)$ with $f=F$ on $\hC$. In this case, for any $(t,\ome,\mu,\tome)\in \hC \times \mC $ and $\tau \le t,$ we write
\be\label{derivatives}
\begin{split}
&\partial_t f(t,\ome,\mu):= \partial_t F(t,\ome,\mu), \  (\partial_{\ome_{\tau}}, \partial_{\ome_{\tau}}^2) f(t,\ome_\tau ,\mu):=  (\partial_{\ome_{\tau}}, \partial_{\ome_{\tau}}^2)  F(t,\ome_\tau ,\mu),\
\\
& \quad \text{and} \quad  (\partial_{\mu_\tau}, \partial_{\tome_\tau}\partial_{\mu_{\tau}} ) f(t,\ome,\mu_{\tau},\tome_{\tau}):= (\partial_{\mu_\tau}, \partial_{\tome_\tau}\partial_{\mu_{\tau}} )  F(t,\ome,\mu_{\tau},\tome_{\tau}).\\ 
\end{split}
\ee
Notations such as $\FC^{1,2,1,1}(\hC), \FC^{1,2,1,1}_{p}(\hC)$ and $ \FC^{0,2,1,1}_{s,p}(\hC)$ are defined in a similar way.
\end{defi}

\ \ \ \ \ \ In view of It\^o-Dupire formulas given in Theorem \ref{itoformula} and Corollary \ref{sito}, we have

\begin{coro}

 For any $(t, \gam , \mu, \eta)\in   \hC \times \MM_2^C,$ $X$ and $ X'$ are diffusion processes given by \eqref{X} and \eqref{X'} respectively.

\begin{itemize}
\item[(i)] Suppose that $f\in \FC^{1,2,1,1}_{s}(\hC)$ ($\FC^{1,2,1,1} (\hC)$, resp.). For any $(\tau,\tome) \in [0,t] \times \mC ,$ derivatives
$
    \partial_t f(t,\ome,\mu) ,\ (\partial_{\ome_\tau}, \partial_{\ome_\tau}^2 )f(t,\ome_\tau ,\mu),\ \partial_{\mu_\tau}f(t,\ome  ,\mu_\tau,\tome_\tau), \   \text{and } \text{Sym}(\partial_{\tome_\tau}\partial_{\mu_\tau}f(t,\ome  ,\mu_\tau, \tome_\tau)),
$
($\partial_t f,\ (\partial_{\ome}, \partial_{\ome}^2 )f,\ \partial_{\mu}f,\   \text{and }  \text{Sym}(\partial_{\tome}\partial_{\mu}f)$, resp.)
defined as in \eqref{derivatives} do not depend on the choice of the extended functional.
 
\item[(ii)] Suppose that $f\in \FC^{1,2,1,1}_{p}(\hC).$ For any $s\ge t,  $ we have
\be  \label{e-ito'}
\begin{split}
&f(s , X , \ML_{X'})- f(t,\gam, \ML_{\eta } )\\
&\ \  = \int_t^s \partial_{r}f(r, X, \ML_{X'}) dr + \int_t^s \partial_{\ome} f(r, X, \ML_{X'} ) d X(r)\\
 &\ \ \ \ \ \ + \frac12  \int_t^{s } \text{Tr}\ [\partial_{\ome}^2 f(r, X , \ML_{X'} ) d \lan X \ran (r) ]+ \E^{\tilde{P}'} [\int_t^s \partial_{\mu} f(r, X , \ML_{X'},\tilde{X}' )d \tilde{X}'(r) ] \\
 &\ \ \ \ \ \ +  \frac12 \E^{\tilde{P}'} \int_t^{s } \text{Tr}\ [\partial_{\tome} \partial_\mu f(r, X, \ML_{X'} , \tilde{X}' ) \tilde{d}(r) \tilde{d}(r)^T ] dr.
 \end{split}
\ee

\item[(iii)] Suppose that $f\in \FC^{0,2,1,1}_{s,p}(\hC).$ For any $t\le s \le s' ,  $ we have the partial It\^o-Dupire formula
\be \label{se-ito'}
\begin{split}
&f(s' , X_s, \ML_{X'_s})- f(s',\gam_t, \ML_{\eta_t } )\\
&\ \  =  \int_t^s \partial_{\ome_r} f(s', X_r, \ML_{X'_r} ) d X(r)\\
 &\ \ \ \ \ \ + \frac12  \int_t^{s } \text{Tr}\ [\partial_{\ome_r}^2 f(s', X_{r}, \ML_{X'_r} ) d \lan X \ran(r) ]+ \E^{\tilde{P}'} [\int_t^s \partial_{\mu_r} f(s', X_r , \ML_{X'_r}, \tilde{X}'_r )d \tilde{X}'(r) ] \\
 &\ \ \ \ \ \ +  \frac12 \E^{\tilde{P}'} \int_t^{s } \text{Tr}\ [\partial_{\tome_r} \partial_{\mu_r} f(s', X_r, \ML_{X'_r} , \tilde{X}'_r ) \tilde{d}(r) \tilde{d}(r)^T ] dr.
 \end{split}
\ee

\end{itemize}

\end{coro}

\begin{proof}
Since for any $n \times n$ matrix $A$ and symmetric $n\times n$ matrix $B$, $\text{Tr}[AB]$ depends only on Sym$(A)$,
$(ii)$ and $(iii)$ follow from $(i)$, Theorem \ref{itoformula} and Corollary \ref{sito} directly. 
To end the proof, we only need to show $(i)$. Indeed, the uniqueness of $\partial_t f(t,\gam,\mu)$ follows from its definition. For the uniqueness of $(\partial_{\gam_\tau},\partial_{\gam_\tau}^2)f (t,\gam_\tau, \mu)$, without loss of generality, assume $\tau=t.$ Otherwise consider the non-anticipative path-dependent function $\tf_{\mu_t}(\tau,\gam ):= f(t,\gam_\tau, \mu )$ instead of $f(t,\gam_\tau, \mu ).$ For any $(\gam, \mu) \in \mC \times \op_2^C,$ take $c(\cdot)=d(\cdot)=0$ in equation \eqref{X'} for $X'$ and $a=0$ in equation \eqref{X} for X. For any extension $F$ of $f,$ applying It\^o formula \eqref{e-ito} to $F(s,X, \ML_{\eta_t})$ on $s\in [t,T]$, we have
\be\label{d-b}
\begin{split}
f(T , X, \ML_{\eta_t})- f(t, \gam_t, \ML_{\eta_t})=& \ \int_t^{T} \partial_r f(r,X, \ML_{\eta_t}) dr + \int_{t}^{T} \partial_{\ome_r} F(r,X, \ML_{\eta_t})b(r) dB(r)\\
&\ \ + \frac12 \int_t^{T}  \text{Tr}\ [ \partial_{\ome_r}^2 F(r,X, \ML_{\eta_t})b(r)b(r)^T] dr.
\end{split}
\ee
In view of identity \eqref{d-b} and the Doob-Meyer theorem for semimartingales, we obtain the uniqueness of $\partial_{\ome_t} F(t,\gam_t,\mu_t)$ and $\text{Sym}(\partial_{\ome_t}^2 F(t,\gam_t, \mu_t))$. For the uniqueness of $\partial_{\mu_\tau}f$ and $\text{Sym}(\partial_{\tome_{\tau}}\partial_{\mu_\tau}f)$, again we assume $\tau=t$. Otherwise consider $\bar{f}_{\ome_t}(\tau,\mu)=f(t,\ome,\mu_\tau)$, and then by definition
$$
\partial_{\mu_\tau}f(t,\ome,\mu_\tau,\tome_{\tau})= \partial_{\mu_\tau} \bar{f}_{\ome_t}(\tau,\mu,\tome),\ \
\text{Sym}(\partial_{\tome_{\tau}}\partial_{\mu_\tau}f(t,\ome,\mu_\tau,\tome_{\tau}))= \text{Sym}(\partial_{\tome_{\tau}}\partial_{\mu_\tau} \bar{f}_{\ome_t}(\tau,\mu,\tome)).
$$
Then the uniqueness of $\partial_\mu f $ and $\text{Sym}(\partial_{\tome}\partial_\mu f) $ follows from \cite[Theorem 2.9]{WZ20}.

\end{proof}

\begin{rem}
The uniqueness of $\partial_{\tome}\partial_\mu f$ can be proved via a similar argument as above from the uniqueness of $\partial_{\mu} f$ under a stronger assumption on the regularity of $f$. However, our It\^o-Dupire formulas and analysis below only depend on $\text{Sym}(\partial_{\tome}\partial_\mu f)$. Indeed, equation \eqref{pde} also only depends on $\text{Sym}(\partial_{\tome}\partial_\mu u)$ instead of $\partial_{\tome}\partial_\mu u.$

\end{rem}

\ \ \ \ \ \ For the case $(\gam,\mu) \in  \mC \times \op_2^C$, we assume
\begin{itemize}
\item[ $\textbf{(A2)}$] $\Phi:\mC \times \op_2^C \mapsto \R$ such that there exists $\Phi'$ satisfying {\bf(H2)}(i) and $\Phi=\Phi'$ on $\mC \times \op_2^C$. $f : [0,T] \times \mC \times \R \times \R^d \times \op_2^C \times \op_2(\R) \mapsto \R$ such that there exists $f'$ satisfying {\bf(H2)}(ii) and $f=f'$ on $[0,T] \times \mC \times \R \times \R^d \times \op_2^C \times \op_2(\R).$\\

\end{itemize}

\ \ \ \ \ \ A functional $u\in \FC^{1,2,1,1}_p(\hC)$ is called a classical solution to equation \eqref{pde} if it satisfies equation \eqref{pde}. The following theorem states the uniqueness of solutions for equation \eqref{pde}.

\begin{thm}
Suppose that $u^1$ and $ u^2$ are two classical solutions to the path-dependent master equation \eqref{pde}. Then $u^1=u^2$.

\end{thm}

\begin{proof}

Apply It\^o formula \eqref{e-ito'} to $u_i(r,B^{\gam_t}, \ML_{\tilde{B}^{\teta_t}})$ on $r \in[t,s],$ $i=1,2,$ and we obtain that
\begin{equation*}
\begin{split}
du_i(r,B^{\gam_t}, \ML_{\tilde{B}^{\teta_t}} )=&\ \partial_r u_i(r,B^{\gam_t}, \ML_{\tilde{B}^{\teta_t}} ) dr + \partial_{\ome}u_i(r,B^{\gam_t}, \ML_{\tilde{B}^{\teta_t}}) dB(r)+ \frac12 \text{Tr} [ \partial^2_{\ome} u_i(r,B^{\gam_t}, \ML_{\tilde{B}^{\teta_t}})] dr\\
& + \frac12 \tE\left[  \text{Tr}  [\partial_{\tome}\partial_{\mu} u_i(r,B^{\gam_t}, \ML_{\tilde{B}^{\teta_t}}, \tilde{B}^{\teta_t} ) d r ]\right].
\end{split}
\end{equation*}
In view of equation \eqref{pde}, we have
\begin{equation*}
\begin{split}
du_i(r,B^{\gam_t}, \ML_{\tilde{B}^{\teta_t}} )=& \ -f(r,B^{\gam_t},u_i(r,B^{\gam_t}, \ML_{\tilde{B}^{\teta_t}}), \partial_\ome u_i(r,B^{\gam_t}, \ML_{\tilde{B}^{\teta_t}}), \ML_{\tilde{B}^{\teta_t}}, \ML_{u_i(r, {B}^{\eta_t}, \ML_{\tilde{B}^{\teta_t}} )} )dr \\
&\  + \partial_{\ome}u_i(r,B^{\gam_t}, \ML_{\tilde{B}^{\teta_t}})dB(r).
\end{split}
\end{equation*}
Then processes $(Y^{\gl}, Z^{\gl})$ and $(Y^{\eta_t},Z^{\eta_t})$ given by
\beaa
&&(Y^{\gl}(s),Z^{\gl}(s)):= (u_i(s,B^{\gam_t}, \ML_{\tilde{B}^{\teta_t}}), \partial_{\ome}u_i(s,B^{\gam_t}, \ML_{\tilde{B}^{\teta_t}}) ), \quad \text{and}\\
&&(Y^{\eta_t}(s),Z^{\eta_t}(s)):= (u_i(s,B^{\eta_t}, \ML_{\tilde{B}^{\teta_t}}), \partial_{\ome}u_i(s,B^{\eta_t}, \ML_{\tilde{B}^{\teta_t}}) ), \ \ \ \ s \ge t,
\eeaa
define solutions to equations \eqref{e2'} and \eqref{e1'}, respectively.
By the uniqueness of solutions for BSDEs \eqref{e2'} and \eqref{e1'}, our conclusion follows.

\end{proof}

%
%
%
%
%
%

\ \ \ \ \ \  Now we show the existence of a classical solution to \eqref{pde} via FBSDEs.

\begin{thm}\label{mainthm}
Suppose that $(f, \Phi)$ satisfies Assumption {\bf(A2)} and $u$ is given by \eqref{decop}. Then $ u  $ restricted on $\hC$ is a classical solution of \eqref{pde}.

\end{thm}

\begin{proof}

In view of Proposition \ref{regular}, we have $u\in \FC_{s,p}^{0,2,1,1,}(\hC).$ For any $(t,\gam, \mu)\in \hC$
 and $h > 0,$
\be\label{u-dec}
\begin{split}
&u(t+h,\gam_t,\mu_t)-u(t,\gam,\mu) \\
&\ \ \ = u(t+h,\gam_t,\mu_t)
-\E\left[ u(t+h,B^{\gam_t},\ML_{B^{\eta_t}})\right]  + \E\left[ u(t+h,B^{\gam_t},\ML_{B^{\eta_t}})\right] -u(t,\gam,\mu).
\end{split}
\ee

Applying partial It\^o formula \eqref{se-ito'} to $u(t+h,B^{\gam_t}_r,\ML_{B^{\eta_t}_r})$ on $r\in [t,t+h]$, we have
\be
\begin{split}\label{u-1}
&u(t+h,\gam_t,\mu_t)- u(t+h,B^{\gam_t},\ML_{B^{\eta_t}})\\
&\ \ \ = -\int_t^{t+h} \partial_{\ome_r} u(t+h,B^{\gam_t}_r, \ML_{B^{\eta_t}_r} ) dB(r) - \frac12 \int_t^{t+h} \text{Tr}\ [ \partial_{\ome_r}^2 u(t+h,B^{\gam_t}_r,\ML_{B^{\eta_t}_r} ) ] dr\\
&\ \ \ \ \ \  - \frac12 \tE  \int_t^{t+h}\text{Tr}\ [ \partial_{\tome_r}\partial_{\mu_r} u(t+h,B^{\gam_t}_r,\ML_{B^{\eta_t}_r},\tB^{\teta_t}_r ) ]dr . \end{split}
\ee
On the other hand, in view of identities \eqref{us=hu}, \eqref{uflow}, and BSDE \eqref{e2'}
\be\label{u-2}
\begin{split}
&u(t+h,B^{\gam_t},\ML_{B^{\eta_t}}) -u(t,\gam,\mu) = Y^{\gl}(t+h)-Y^\gl(t)\\
&\quad = -\int_t^{t+h} f(r,B^{\gam_t}, Y^{\gl},Z^{\gl},\ML_{B^{\eta_t}}, \ML_{Y^{\eta_t}}) dr + \int_t^{t+h} Z^\gl(r)dB(r).
\end{split}
\ee
Putting \eqref{u-1} and \eqref{u-2} to \eqref{u-dec}, and taking expectation $\E$, we obtain  
\be \label{du}
\begin{split}
u(t+h,\gam_t,\mu_t)-u(t,\gam,\mu)&= - \frac12 \int_t^{t+h} \E\left[ \text{Tr} [ \partial_{\ome_r}^2 u(t+h,B^{\gam_t}_r,\ML_{B^{\eta_t}_r} ) ]\right] dr\\
 &\ \ \ - \frac12 \int_t^{t+h} \E\left[  \tE \text{Tr}[ \partial_{\tome_r}\partial_{\mu_r} u(t+h,B^{\gam_t}_r,\ML_{B^{\eta_t}_r},\tB^{\teta_t}_r ) ]\right] dr \\
 &\ \ \  - \int_t^{t+h} \E\left[  f(r,B^{\gam_t}, Y^{\gl},Z^{\gl},\ML_{B^{\eta_t}}, \ML_{Y^{\eta_t}})\right] dr.
 \end{split}
\ee
Moreover, in view of \eqref{uflow}, \eqref{us=hu} and \eqref{us=ys}, we have  
\bea
&&Y^{\gl}(r)=u(r,B^{\gam_t}, \ML_{B^{\eta_t}}), \ Z^{\gl}({r})= \partial_{\gam_t} u(r,B^{\ome_t}, \ML_{B^{\eta_t}}),\\
&&\quad \text{and } \quad Y^{\eta_t}(r)= u(r, B^{\eta_t}, \ML_{B^{\eta_t}}).
\eea
Then dividing both sides of \eqref{du} by $h$ and taking $h\rightarrow 0^+$, according to the dominated convergence theorem and Proposition \ref{regular}, we obtain  
\begin{align*}
\partial_t u(t,\gam ,\mu)&= - \frac12 \text{Tr}\ \left[ \partial_{\ome}^2 u(t,\gam,\mu)\right] - \frac12  \text{Tr} \left[ \E^{\mu}[\partial_{\tome} \partial_\mu u(t,\gam, \mu, W )]\right]\\
&\ \ \ - f\left(t,\gam, u(t,\gam,\mu), \partial_\ome u(t,\gam,\mu), \mu, \ML_{u(t,B^{\eta_t},\mu)}\right).
\end{align*}

\end{proof}

\subsection{Classical solution of semi-linear path-dependent PDEs}\label{semi}

\ \ \ \ \ \ As stated in the introduction, a classical solution to a semi-linear equation \eqref{1} suffers from several problems if one tries to build the solution via the classical argument of FBSDE theory. In the following, we approximate the classical solution to \eqref{1} via a sequence of solutions to corresponding ``coefficient frozen'' equations. For simplicity of technique and notations, here we only consider the measure independent case and assume $f$ is independent of $z$-variable, which is also new even restricted in the path-dependent setting,
\be\label{semieq}
\left\{
\begin{array}{l}
\partial_t u(t,\ome ) + \frac12 \text{Tr}\ [ \partial_{\ome}^2 u(t,\ome )\sigma (\ome_t)\sigma(\ome_t)^T]  + \partial_\ome u(t,\ome ) b(\ome_t)+ f(t,\ome, u(t,\ome ))=0,\\
 \\
 u(T, \ome, \mu)= \Phi(\ome ), \ \ \ (t,\ome )\in [0,T]\times \mC .
\end{array}
\right.
\ee 
However, see Remark \ref{gsemieq} for the general and mean-field case. 
Instead of considering a forward SDE, we 
 consider the following ``coefficient frozen'' SDE
\be
\left\{
\begin{array}{l}
X^{\vep, \gam_t}(s):= \gam(t) + b(\gam_{t-\vep}) (s-t) + \sigma(\gam_{t-\vep}) (B_s- B_t)  , \quad s\ge t>0,   \\[2mm]
X^{\vep, \gam_t}_t= \gam_t .
\end{array}\right.
\ee
Here and in the following $\vep>0$ is a small parameter.
According to the above definition, we have that $X^{\vep, \gam_t}$ is independent of $\MF_t,$ and for any strongly vertically differentiable functional $\Phi: \D \rightarrow \R,$ 
\be\label{conv-derivative}
\partial_{\gam_{\tau}} [ \Phi(X^{\vep, \gam_t}) ]= \partial_{\ome_{\tau}} \Phi(X^{\vep, \gam_t}), \ \ \ \forall \tau \in (t-\vep, t].
\ee

In view of the FBSDE argument, we consider the following path-dependent BSDE,
\be 
Y^{\vep, \gam_t}(s)= \Phi(X_T^{\vep, \gam_t}) + \int_s^T f(r, X_r^{\vep, \gam_t}, Y^{\vep, \gam_t}(r)) dr - \int_t^T Z^{\vep, \gam_t}(r) dB(r), \ \ \ s \ge t.
\ee

In view of \eqref{conv-derivative}, the following result follows from the differentiability of the corresponding BSDEs shown in Proposition \ref{s-bdd1}. 
\begin{coro}\label{local-derivative}
Assume that $(f,\Phi)$ satisfies assumption \textbf{(A2)} and fix $t\in (0,T].$
For any $(\tau, \gam) \in (t-\vep, t]\times \D,$ $Y^{\vep, \gam_t}(s)$ is twice strongly vertically differentiable at $(\tau,t,\gam)$ for any $s \ge t.$ Moreover, derivatives $\partial_{\gam_\tau} Y^{\vep, \gam_t} $ and  $\partial_{\gam_\tau}^2 Y^{\vep, \gam_t}$ are solutions to \eqref{172} and \eqref{173} respectively.
\be\label{172}
\begin{split}
\partial_{\gam_{\tau}} Y^{\vep, \gam_t}(s)= &\ \partial_{\ome_{\tau}} \Phi(X^{\vep, \gam_t}_T) + \int_s^T \Big[\partial_{\ome_\tau} f (r,X_r^{\vep, \gam_t}, Y^{\vep, \gam_t}(r)) \\
&\ + \partial_y f(r,X_r^{\vep, \gam_t}, Y^{\vep, \gam_t}(r)) \partial_{\gam_{\tau}} Y^{\vep, \gam_t}(r)) \Big] dr - \int_s^T \partial_{\gam_\tau}Z^{\vep, \gam_t }(r) dB(r).
\end{split}
\ee

\be\label{173}
\begin{split}
\partial_{\gam_{\tau}}^2 Y^{\vep, \gam_t}(s)= &\ \partial_{\ome_{\tau}}^2 \Phi(X^{\vep, \gam_t}_T)  + \int_s^T \Big[ \partial_y f(r,X_r^{\vep, \gam_t}, Y^{\vep, \gam_t}(r)) \partial_{\gam_\tau}^2 Y^{\vep,\gam_t}(r) \\
&\ + \partial_{\ome_\tau}\partial_y f(r,X_r^{\vep, \gam_t}, Y^{\vep, \gam_t}(r)) \partial_{\gam_{\tau}} Y^{\vep, \gam_t}(r)) +\partial_{\ome_\tau}^2 f (r,X_r^{\vep, \gam_t}, Y^{\vep, \gam_t}(r)) \Big] dr\\
&\ - \int_s^T \partial_{\gam_\tau}^2 Z^{\vep, \gam_t }(r) dB(r).
\end{split}
\ee

\end{coro}

Let 
\be\label{vepu}
u^\vep(t,\gam):= Y^{\vep, \gam_t}(t),  \ \ \ \forall (t,\gam) \in [0,T] \times \D,
\ee
and it follows from the above corollary that $u^\vep(t,\gam) $ is twice strongly vertically differentiable at $(\tau,t, \gam)$ for any $\tau \in (t-\vep, t].$ Then according to It\^{o}'s formula \eqref{se-ito'}, we have that for any $h\in (0,\vep),$
\be\label{expan1}
\begin{split}
&u^{\vep}(t+h,X^{\vep, \gam_t}_{t+h})- u^\vep (t+h,\gam_t)\\
&\ \ \ = \int_t^{t+h} \partial_{\ome_r} u (t+h, X^{\vep, \gam_t}_r) b(\gam_{t-\vep}) d r+ 
\int_t^{t+h} \partial_{\ome_r} u (t+h, X^{\vep, \gam_t}_r) \sigma(\gam_{t-\vep}) d B(r)
\\
&\ \ \ \ \ \ \ +\frac12 \text{Tr}  \int_t^{t+h} \partial^2_{\ome_r} u (t+h, X^{\vep, \gam_t}_r)\sigma^2(\gam_{t-\vep})   dr.
\end{split}
\ee
Here and in the following we assume $d=1$ for simplicity.
On the other hand, according to a classical argument as in \cite[Lemma 4.4]{PW16} (also see Proposition \ref{flow}), we have 
\be\label{pathflow}
u^{\vep}(s,X_s^{\vep,\gam_t})= Y^{\vep, \gam_t}(s), \ a.s. ,\ \ \ \forall s \ge t.
\ee
Then in view of \eqref{expan1} and \eqref{pathflow}, we have 
\be
\begin{split}
&u^\vep (t+h, \gam_t)- u^\vep (t,\gam_t) =  u^\vep  (t+h, \gam_t)- \E u^{\vep}(t+h,X_{t+h}^{\vep,\gam_t}) + \E u^{\vep}(t+h,X_{t+h}^{\vep,\gam_t})-u^\vep (t,\gam_t)\\
&\ \ \  =-\E\int_t^{t+h} \Big[ \partial_{\ome_{r}} u^\vep (t+h,X^{\vep,\gam_t}_r) b(\gam_{t-\vep}) + \frac12 \partial_{\ome_r}^2 u^\vep (t+h, X^{\vep,\gam_t}_r) \sigma^2(\gam_{t-\vep})\Big] dr\\
&\ \ \ \ \ \ \ - \E \int_t^{t+h} f(r, X^{\vep,\gam_t}_r, Y^{\vep,\gam_t}(r))dr.
\end{split}
\ee

Divide by $h$ both sides of the above identity and take $h$ to zero, we have 
\be\label{approxeq}
\partial_t u^\vep (t, \gam)= -\partial_\ome u^\vep (t, \gam ) b(\gam_{t-\vep }) - \frac12 u^\vep (t, \gam_t ) \sigma^2(\gam_{t-\vep})- f(t, \gam_t, u^\vep (t, \gam)).
\ee

\ \ \ \ \ \ \ To build a solution to \eqref{semieq}, we only need to show $(b,\sigma)(\gam_{t-\vep})$ converges to $(b,\sigma)(\gam_{t})$, and $(I, \partial_\ome, \partial_\ome^2, \partial_t)u^\vep $ also converges to a limit as $\vep$ vanishes. In the following we denote by ``$\lesssim$'' that the left hand is bounded by the right hand side up to a generic constant. Firstly we show that $u^\vep $ converges to a limit $u$. To this end, we assume that 
\begin{itemize}
\item[ $\textbf{(A3)}$] $(i)$. $(b, \sigma):\D \rightarrow \R^d \times \R^{d\times d}$ is locally Lipschitz continuous, i.e. for $\phi= b$ or $\sigma,$
\be
|\phi (\ome^1)- \phi(\ome^2) | \le  (1+ \|\ome^1\|^k+\|\ome^2\|^k ) \|\ome^1- \ome^2\|,
\ee
for some $k \ge 1.$ Moreover, $b$ and $\sigma$ are predictable in the sense of
\be\label{predictable}
(b,\sigma)(\gam_t)=(b,\sigma)(\gam_{t-}),\ \ \ \text{for any }(t,\gam) \in (0,T] \times \D,
\ee
where $\gam_{t-}(s)=\gam_t(s-)$ for any $s\in (0,T].$\\[1mm]
$(ii)$. $(f,\Phi):[0,T] \times \D \times \R \rightarrow \R$ satisfies assumption (A2), and moreover, for any $  (t,\ome,y)\in [0,T] \times \mC \times \R$ and $\{\tau,\tau'\} \in [t,T],$ we have
\begin{align}\nonumber 
&|\partial_{\ome_\tau}\Phi(\ome)-\partial_{\ome_{\tau'}}\Phi(\ome)|+ |\partial^2_{\ome_\tau}\Phi(\ome)-\partial^2_{\ome_{\tau'}}\Phi(\ome)|\lesssim (1+ \|\ome\|^k) \rho(|\tau-\tau'|) , \\ \nonumber 
&|\partial_{\ome_\tau}f(t,\ome,y)-\partial_{\ome_{\tau'}}f(t,\ome,y)| + |\partial^2_{\ome_\tau}f(t,\ome,y)-\partial^2_{\ome_{\tau'}}f(t,\ome,y)| 
 \lesssim (1+ \|\ome\|^k+|y|) \rho(|\tau-\tau'|) , 
\end{align}
where $\rho:[0,\infty) \rightarrow [0,\infty)$ is continuous at $0$ with $\rho(0)=0.$

\end{itemize}

\begin{example}
For a functional $g$ on $\D$ with property \eqref{predictable}, we must have $g(\ome_t)= G(\ome_{t-})$ for a path functional $G$ and vice versa. The benchmark example is 
$g(\ome)= \int_0^T F(r, \ome(r)) dr$ for a continuous function $F$ on $[0,T] \times \R$.
	
\end{example}

Let 
\be\label{newu}
X^{\gam_t}:= X^{0,\gam_t}, \ \ \ Y^{\gam_t}:= Y^{0,\gam_t}, \ \ \ \text{and } \ u (t,\gam):= Y^{ \gam_t}(t), \ \ \ \forall (t,\gam) \in [0,T]\times \D.
\ee
Firstly we show that $u^{\vep }$ converges to $u$ as $\vep $ goes to null. In the following we denote by 
$$
\text{Osc}(\gam,t,\vep):= \sup_{u,v \in [t-\vep, t)} |\gam(u)- \gam(v)|,  \ \ \ \forall \  (t,\gam) \in [0,T] \times \D.
$$

\begin{lem}\label{appro-u}
Suppose that $(b,\sigma, f, \Phi)$ satisfies assumption (A3). Then for any $(t,\gam) \in [0,T] \times \D,$ we have 
\be
\E\|X^{\vep,\gam_t}- X^{\gam_t}\|^2 +\E \| Y^{\vep, \gam_t} - Y^{\gam_t}\|^2 \lesssim (1+\|\gam_t\|^k)^2 \text{Osc}(\gam,t,\vep)^2
\ee
In particular, for any $(t,\gam ) \in [0,T] \times \mC,$
$$
|u^{\vep}(t,\gam)- u(t,\gam)| \lesssim (1+\|\gam_t\|^k) \text{Osc}(\gam,t,\vep).
$$ 

\end{lem}

\begin{proof}
Note that 
$$
X^{\vep, \gam_t}(s)-X^{\gam_t}(s)= (b(\gam_{t-\vep})-b(\gam_{t}))(s-t)+ (\sigma(\gam_{t-\vep})- \sigma(\gam_{t}))(W(s)-W(t)).
$$
It follows from the Burkholder-Davis-Gundy inequality that 
\be
\E\|X^{\vep, \gam_t}-X^{\gam_t}\|^2 \lesssim (1+\|\gam_t\|^k)^2 \| \gam_{t-\vep}- \gam_t\|^2 \lesssim (1+\|\gam_t\|^k)^2 \sup_{u,v \in [t-\vep, t)} |\gam(u)- \gam(v)|^2.
\ee
According to Lemma \ref{linear}, we have the estimate for $\E \| Y^{\vep, \gam_t} - Y^{\gam_t}\|^2.$

\end{proof}

\ \ \ \ \ \  Next we show $u(t,\gam)$ is vertically differentiable. Indeed, in view of \cite[Theorem 4.5]{PW16} and Remark \ref{H2}, thanks to the predictable assumption \eqref{predictable}, we have that $Y^{\gam_t}$ is twice vertically differentiable at $(t,\gam)$ following similar argument. Remark that $Y^{\gam_t}$ may not be strongly vertically differentiable. Then $u(t,\gam)$ is twice vertically differentiable and its derivatives $(\partial_\ome u, \partial_\ome^2 u)(t,\ome)$ satisfy
\be
(\partial_{\ome} u, \partial^2_{\ome} u)(t,\gam)=( \partial_{\gam_t} Y^{\gam_t}(t),\partial^2_{\gam_t} Y^{\gam_t}(t)),
\ee
where $\partial_{\gam_t} Y^{\gam_t}$ and $\partial^2_{\gam_t} Y^{\gam_t}$ are solutions to \eqref{172} and \eqref{173} respectively with $\vep=0.$ In conclusion, we have

%

\begin{lem}\label{ver-de-con}
$u$ is continuously twice vertically differentiable. Moreover, for any $(t,\gam) \in [0,T]\times \D,$ 
\be(\partial_{\ome}u^\vep (t,\gam),\partial_{\ome}^2u^\vep (t,\gam)) \rightarrow (\partial_{\ome}u(t,\gam), \partial_{\ome}^2u(t,\gam)), \ \ \ \text{as} \ \vep \rightarrow 0.
\ee
	
\end{lem}

\begin{proof}
The first part of this theorem follows from the argument above this lemma. For the second part, in a similar spirit as the proof of \eqref{3.7} of Proposition \ref{s-bdd1}, we have 
\be
\|\partial_\ome u^{\vep}(t,\gam)- \partial_\ome u (t,\gam)\| \le \|X^{\vep, \gam_t}- X^{\gam_t}\|_{\bS^2} \lesssim (1+\|\gam_t\|^k) \text{Osc}(\gam,t,\vep),
\ee
which converges to zero as $\vep $ goes to null since $\gam \in \D.$
Then convergence of $\partial_\ome^2 u$ follows similarly.

%
%

\end{proof}

\ \ \ \ It remains to prove the horizontal differentiability of $u$. To this end, we need the following estimates.

\begin{lem}\label{4.13}

For any $\{t,t'\} \in (0,T],$ and $\{\gam,\gam'\} \in \D,$ we have for any $p \ge 2,$
\begin{align}\label{est-xt}
\| X^{\vep, \gam_t}-  X^{\vep, \gam'_{t'}}\|_{\bS^{p}} \le C (1+\|\gam_t\|^{k+1}+\|\gam'_{t'}\|^{k+1}) (\| \gam_t- \gam'_{t'}\|+ |t-t'|^\frac12),\\
\| Y^{\vep, \gam_t}-  Y^{\vep, \gam'_{t'}}\|_{\bS^p} \le C (1+\|\gam_t\|^{k+1}+\|\gam'_{t'}\|^{k+1}) (\| \gam_t- \gam'_{t'}\|+ |t-t'|^\frac12),
\end{align}
where the generic constant $C$ is independent of $\vep.$
	
\end{lem}

\begin{proof}

We only show the first estimate and the second one follows from the first and classical argument as shown in Lemma \ref{unibd}. Without loss of generality, assume $t'>t.$ Indeed, for $s \le t,$
$|X^{\vep, \gam_t}(s)-X^{\vep, \gam'_{t'}}(s)| =|\gam(s)-\gam'(s)|$. For $s\in[t,t')$
\be
\begin{split}
|X^{\vep, \gam_t}(s)-X^{\vep, \gam'_{t'}}(s)| &\le |\gam(t)-\gam'(s)|+|b(\gam_{t-\vep})|(t'-t)+|\sigma(\gam_{t-\vep})| \ |W_s- W_t|\\
&\lesssim  \|\gam_t- \gam'_{t'}\| + (1+\|\gam_t\|^{k+1}) (|t'-t| +|W_s- W_t|).
\end{split}
\ee
For $s\in [t',T],$ 
\be
\begin{split}
&|X^{\vep, \gam_t}(s)-X^{\vep, \gam'_{t'}}(s)|\\
 &\ \ \ \ \le |\gam(t)-\gam'(t')|+|b(\gam_{t-\vep})|(t'-t)+|\sigma(\gam_{t-\vep})| \ |W_{t'}- W_t|\\
&\ \ \ \ \ \ \  + (b(\gam_{t-\vep})-b(\gam'_{t'-\vep}))(s-t') +(\sigma(\gam_{t-\vep})-\sigma (\gam'_{t'-\vep}))|W_s-W_{t'}|\\
&\ \ \ \  \lesssim  (1+\|\gam_t\|^{k+1} + \gam'_{t'}\|^{k+1}) (|t'-t| +|W_{t'}- W_t| + \|\gam_t- \gam'_{t'}\| (1+|W_s-W_{t'}|)),
\end{split}
\ee
which implies \eqref{est-xt} by Burkholder-Davis-Gundy inequality.
	
\end{proof}

\begin{prop}
Suppose that $(b,\sigma, f, \Phi)$ satisfies assumption \textbf{(A3)}.
For any $\tau \in (t-\vep,t],$ let $\partial_{\gam_\tau}Y^{\vep, \gam_t}$ and $\partial^2_{\gam_\tau}Y^{\vep, \gam_t}$ be solutions of linear BSDEs \eqref{172} and \eqref{173} respectively. Then for any $\gam, \gam' \in \D,$ $t,t' \in (0,T]$, $\tau \in (t-\vep,t],$ and $\tau' \in (t'-\vep,t'],$ we have 
\begin{align}\label{202}
&	\|\partial_{\gam_\tau}Y^{\vep, \gam_t}-\partial_{\gam'_{\tau'}}Y^{\vep, \gam'_{t'}} \|_{\bS^2,[t' \vee t,T]}\\ \nonumber
&\ \ \ \ \ \  \le C (1+\|\gam\|^\ell +\|\gam'\|^\ell) (\|\gam_t- \gam'_{t'}\|+|t-t'|^\frac12 + \rho(|\tau-\tau'|)),\\ \label{203}
&	\|\partial^2_{\gam_\tau}Y^{\vep, \gam_t}-\partial^2_{\gam'_{\tau'}}Y^{\vep, \gam'_{t'}} \|_{\bS^2,[t' \vee t,T]}\\ \nonumber
&\ \ \ \ \ \  \le C (1+\|\gam\|^\ell +\|\gam'\|^\ell) (\|\gam_t- \gam'_{t'}\|+|t-t'|^\frac12 + \rho(|\tau-\tau'|)),
\end{align}
where $\ell>k+1$ and the generic constant $C$ is independent of $\vep.$
	
\end{prop}

\begin{proof}

Without loss of generality, assume $t \le t'.$ In the following we write $(X,X',Y,Y',Z,Z' )$ short for $(X^{\vep, \gam_t},X^{\vep, \gam'_{t'}},Y^{\vep, \gam_t},Y^{\vep, \gam'_{t'}},Z^{\vep, \gam_t},Z^{\vep, \gam'_{t'}}).$ For any $s \ge t',$ let 
$$
(\delta y(s),\delta z(s)):=(\partial_{\gam_\tau}Y(s)-\partial_{\gam'_{\tau'}}Y'(s),\partial_{\gam_\tau}Z(s)-\partial_{\gam'_{\tau'}}Z'(s)),
$$
and it is the solution to linear BSDE
\be
\begin{split}
\delta y(s)=& \ [\partial_{\ome_\tau} \Phi (X )-\partial_{\ome_{\tau'}} \Phi (X')] + \int_s^T \Big[ \left(\partial_{\ome_{\tau}} f(r,X(r),Y(r))-\partial_{\ome_{\tau'}}f(r,X'(r),Y'(r))\right)\\
& + \partial_y f(r,X(r),Y(r)) \delta y(r) + \left( \partial_y f(r,X(r),Y(r))-\partial_y f(r,X'(r),Y'(r)) \right)Y'(r) \Big] dr\\
& - \int_s^T \delta z(r)dB(r).	
\end{split}
\ee
In view of Lemma \ref{linear}, we have 
\begin{align}\nonumber
 \|\delta y\|^2_{\bS^2,[t',T]}  \lesssim & \  \|\partial_{\ome_\tau} \Phi (X )-\partial_{\ome_{\tau'}} \Phi (X')\|_{L^2}^2 + \E \left[\int_{t'}^T |\partial_{\ome_{\tau}} f(r,X(r),Y(r))-\partial_{\ome_{\tau'}}f(r,X'(r),Y'(r))| dr\right]^2	\\ \label{decdy}
 &\  + \E \left[\int_{t'}^T |\partial_{y} f(r,X(r),Y(r))-\partial_{y}f(r,X'(r),Y'(r))| |Y'(r)| dr \right]^2.
\end{align}
For the first term on the right hand side of the above inequality, we have
\be
\begin{split}
&\|\partial_{\ome_\tau} \Phi (X )-\partial_{\ome_{\tau'}} \Phi (X')\|_{L^2}^2\\
&\ \ \ \ \ \  \lesssim \|\partial_{\ome_\tau} \Phi (X )-\partial_{\ome_{\tau'}} \Phi (X)\|_{L^2}^2+\|\partial_{\ome_{\tau'}} \Phi (X )-\partial_{\ome_{\tau'}} \Phi (X')\|_{L^2}^2\\
&\ \ \ \ \ \  \lesssim   (1+\|\gam_t\|^{k+1}+\|\gam'_{t'}\|^{k+1})^2 ( \rho(|\tau-\tau'|)^2+\| \gam_t- \gam'_{t'}\|^2+ |t-t'|), \\
\end{split}
\ee
where we apply assumption \textbf{(A3)} and Lemma \ref{4.13} in the last inequality. Similarly, we have 
\be
\begin{split}
& \E \left[\int_{t'}^T |\partial_{\ome_{\tau}} f(r,X(r),Y(r))-\partial_{\ome_{\tau'}}f(r,X'(r),Y'(r))| dr\right]^2\\
&\ \ \ \ \ \  \lesssim    (1+\|\gam_t\|^{k+1}+\|\gam'_{t'}\|^{k+1})^2 ( \rho(|\tau-\tau'|)^2+\| \gam_t- \gam'_{t'}\|^2+ |t-t'|).	
\end{split}
\ee
For the last term on the right hand side of \eqref{decdy}, by the Cauchy-Schwartz inequality, we have 
\be
\begin{split}
& \E \left[ \int_{t'}^T |\partial_{y} f(r,X(r),Y(r))-\partial_{y}f(r,X'(r),Y'(r))| |Y'(r)| dr\right]^2\\
&\ \ \ \ \ \  \lesssim   \left[ \E[\| X-X'\|^4+ \|Y-Y' \|^4] \right]^{\frac12} [\E\|Y'\|^4]^{\frac12}\\
&\ \ \ \ \ \  \lesssim    (1+\|\gam_t\|^{\ell}+\|\gam'_{t'}\|^{\ell}) ( \rho(|\tau-\tau'|)^2+\| \gam_t- \gam'_{t'}\|^2+ |t-t'|),
\end{split}
\ee
with an integer $\ell.$ Then inequality \eqref{202} is implied by the above estimates. \eqref{203} follows similarly.

\end{proof}

\ \ \ \ \ \ Now we are ready to give the main theorem of this subsection. To have the uniqueness of solutions for \eqref{semieq}, we assume that 
\begin{itemize}
\item[ $\textbf{(A4)}$] $(b,\sigma)$ is Lipschitz on $\D, $ i.e. for $\phi= b, \sigma,$
$
|\phi(\gam)-\phi(\gam')| \lesssim \|\gam-\gam'\|.
$

\end{itemize}

\begin{thm}

Suppose assumption \textbf{(A3)} holds. Then there exists a classical solution to equation \eqref{semieq}. Moreover, if \textbf{(A4)} also holds for $(b,\sigma),$ the solution is unique on $\FC^{1,2}_p(\mC).$

\end{thm}

\begin{proof}

Firstly we show that $u$ given by \eqref{newu} is a classical solution to \eqref{semieq}. In view of Lemma \ref{ver-de-con}, we need to show $u$ is horizontally differentiable. Recall that $u^\vep $ is given by \eqref{vepu}. For any $h>0,$
\be\label{decomhu}
\begin{split}
&	u(t+h, \gam_t)- u(t,\gam_t)\\
&\ \ \ 	=[u(t+h, \gam_t)- u^\vep(t+h, \gam_t)]+ [u^\vep(t+h, \gam_t)- u^\vep (t,\gam_t)] + [u^\vep (t,\gam_t) - u(t,\gam_t)].\\	
\end{split}
\ee
According to Lemma \ref{appro-u}, we have 
\be\label{208}
|u(t+h, \gam_t)- u^\vep(t+h, \gam_t)| \lesssim (1+\|\gam_t\|^k) \text{Osc}(\gam_t,t+h,\vep) =0,
\ee
whenever $\vep <h.$	Similarly, for the third term on the right hand side of \eqref{decomhu},
\be\label{209}
|u^\vep (t,\gam_t) - u(t,\gam_t)| \lesssim (1+\|\gam_t\|^k) \text{Osc}(\gam_t,t,\vep)
\ee
Since $\gam \in \D,$ for any $h,$ we can choose $\vep $ small such that $\text{Osc}(\gam_t,t,\vep) =o(h).$
Since $u^\vep$ is horizontally differentiable, in view of \eqref{approxeq}, we have 
\be\label{210}
\begin{split}
&u^\vep(t+h, \gam_t)- u^\vep (t,\gam_t) = h\int_0^1 \partial_t u^{\vep}(t+\lambda h, \gam_t) d \lambda 	\\
 = & h \int_0^1\Big[-\partial_{\ome} u^\vep (t+\lambda h, \gam_t)b(\gam_{t-\vep}) - \frac12 \partial_\ome^2 u^\vep(t+\lambda h, \gam_t)\sigma(\gam_{t-\vep})^2 -f(\gam_t, u^\vep(t+\lambda h, \gam_t) ) \Big].
\end{split}
\ee

Take \eqref{208}, \eqref{209} and \eqref{210} to \eqref{decomhu}, and divide both sides by $h.$ Then we have 
\be
\begin{split}
\frac{u(t+h, \gam_t)- u(t,\gam_t)}{h}  =& \ o(1) +  \int_0^1\Big[-\partial_{\ome} u^\vep (t+\lambda h, \gam_t)b(\gam_{t-\vep}) - \frac12 \partial_\ome^2 u^\vep(t+\lambda h, \gam_t)\sigma(\gam_{t-\vep})^2 \\
& \ -f(t, \gam_t, u^\vep(t+\lambda h, \gam_t) ) \Big].
\end{split}
\ee
Let $h$ go to zero, and then $\vep$ go to zero. According to Lemma \ref{appro-u} and Lemma \ref{ver-de-con}, the right hand side of the above identity converges to $-\partial_{\ome} u (t,\gam_t) b(\gam_t)- \frac12 \partial_\ome^2 u (t , \gam_t)\sigma(\gam_{t })^2 - f(t,\gam_t, u(t,\gam_t))$, which implies $u$ is horizontally differentiable and $u$ satisfy PPDE \eqref{semieq}. 
Now we show the uniqueness of equation \eqref{semieq}. Suppose that $\bar{u} \in \FC^{1,2}_p $ is a classical solution. Consider a path-dependent SDE
\be
\left\{
\begin{array}{l}
\bar{X}(s):= \gam(t) + \int_t^s b(\bX_r) dr + \int_t^s \sigma(\bX_r) dB(r) , \quad s\ge t>0,   \\[2mm]
\bX_t= \gam_t .
\end{array}\right.
\ee

According to assumption \textbf{(A4)}, the last equation has a unique solution $\bX$. Applying functional It\^o's formula to $\bar{u}(s, \bX_s),$ we have
\be
\begin{split}
d\bar{u}(s,\bX_s)= &\ [\partial_s \bar{u}(s,\bX_s)+ \partial_\ome \bar{u}(s,\bX_s) b(\bX_s) + \frac12 \partial^2_\ome \bar{u}(s,\bX_s) \sigma^2(\bX_s) ] ds\\
&\ + \ \partial_\ome \bar{u}(s,\bX_s) \sigma(\bX_s)d B(s),\\
=& \ -f(s,\bar{X}_s, \bar{u}(s, \bX_s)) ds +\partial_\ome \bar{u}(s,\bX_s) \sigma(\bX_s)d B(s).
\end{split}
\ee
Let 
\be
( Y,  Z)(s):= (\bar{u}(s, \bX_s), \partial_\ome \bar{u}(s,\bX_s) \sigma(\bX_s)), 
\ee
and it gives a solution to the following BSDE,
\be
dY(s)= -f (s, \bX_s, Y(s)) ds + Z(s)dB(s),\ \ \ 
Y(T)=\Phi (\bX_T).
\ee
According to the well-posedness of the above BSDE, for any classical solution $\tilde{u},$ 
$
\tilde{u}(s, \bX_s)= Y(s).
$  
In particular, $\tilde{u}(t,\gam_t)= Y(t)= \bar{u}(t,\gam_t)$ which concludes the uniqueness.

\end{proof}

\begin{rem}\label{gsemieq}

$(i)$
For the general case when $f$ depends also on $\partial_\ome u$, in view of \eqref{uflow} and regularity results on $Z$ of BSDEs (see \cite{MZ}), a smooth solution can be constructed in a similar spirit under the condition that $Z$ has continuous paths.
$(ii)$ For the mean-field equation \eqref{1}, similar as shown in Section \ref{sol-to}, we can introduce the following diffusion 
\be\label{X2}
\left\{
\begin{array}{l}
X(s)= \eta(t)+ b_2(\mu_t) (s-t)+ \sigma_2(\mu_t) (B(s)-B(t)),\\
X_t =\eta_t.
\end{array}
\right.
\ee
Similarly as the pure path-dependent case above, we apply the approximation argument with ``coefficient frozen'' process
\be\label{X2}
\left\{
\begin{array}{l}
X^{\vep,\mu_t}(s)= \eta(t)+ b_2(\mu_{t-\vep}) (s-t)+ \sigma_2(\mu_{t-\vep}) (B(s)-B(t)),\\
X_t =\eta_t,
\end{array}
\right.
\ee
where $\ML(\eta)=\mu.$
Under similar assumptions as \textbf{(A3)} and \textbf{(A4)} with adaptation in a mean-field setting, one can construct a unique classical solution to \eqref{1}.

\end{rem}

%
%
%
%
%

\subsection{Some typical cases}

\ \ \ \ \ \ In view of Remark \ref{state}, the path-dependent mean-field equation \eqref{pde} involves many interesting special cases. In the following we list some typical ones, where we always assume that $(f,\Phi)$ satisfies Assumption {\bf(A2)} and $(t,\ome,y,z,\mu,\nu)\in [0,T]\times \mC \times \R \times \R^d \times \op_2^C \times \op_2(\R).$

\bigskip

$(i)$ \textbf{The state-dependent master equation}. Suppose that $(f,\Phi)$ has a state-dependent form:
\bea
& f(t,\ome, y, z, \mu, \nu )= F(t,\ome(t), y, z, \mu(t), \nu),\\
& \Phi(T, \ome , \mu)= G(\ome(T), \mu(T)),
\eea
for functional $F:[0,T] \times \R^d \times \R \times \R^d \times \op_2(\R^d) \times \op_2(\R) \mapsto \R$ and $G:\R^d \times \op_2(\R^d) \mapsto \R.$ In this case, the differentiability of $(f,\Phi)$ is equivalent to the differentiability of $(F,G)$ in its corresponding domain, and path-dependent equation \eqref{pde} has the form
\bea\label{state1}
\left\{
\begin{array}{l}
\partial_t u(t,\gam, \mu ) + \frac12 \text{Tr}\ [ \partial_{\ome}^2 u(t,\gam,\mu)] + \frac12 \text{Tr}\left[ \int_{\mC} \partial_{\tome} \partial_\mu u(t,\gam, \mu, \tome )\mu(d\tome)\right] \\[2mm]
\ \ \ \ + F(t,\gam(t), u(t,\gam,\mu), \partial_\ome u(t,\gam,\mu), \mu(t), \ML_{u(t,W^\mu,\mu)})=0,\\
 \\
 u(T, \gam, \mu)= G(\gam(T), \mu(T) ), \ \ \ (t,\gam,\mu)\in \hC.
\end{array}
\right.
\eea
Since the corresponding FBSDE is Markovian, we see that $u(t,\gam,\mu)=U(t,\gam(t),\mu(t))$ for a smooth (indeed $\FC^{1,2,1,1}$ in view of Definition \ref{diff} with obvious adjustment) functional $U:[0,T] \times \R^d \times \op_2(\R^d) \mapsto \R$ thanks to Remark \ref{state}. Then we obtain well-posedness of the (state-dependent) master equation considered in \cite{BLP17}, \cite{CCD},
\bea\label{state2}
\left\{
\begin{array}{l}
\partial_t U(t,a, \lambda ) + \frac12 \text{Tr}\ [ \partial_{a}^2 U(t,a,\lambda)] + \frac12 \int_{\R^d} \partial_{\tilde{a}} \partial_\lambda U(t,a, \lambda, \tilde{a} )\lambda(d\tilde{a}) \\[2mm]
\ \ \ \  + F(t,a, U(t,a,\lambda), \partial_a U(t,a,\lambda), \lambda , \ML_{u(t,\xi,\lambda)})=0,\\
 \\
 U(T, a, \lambda)= G(a, \lambda ),\ \ \ \ (t,a,\lambda)\in [0,T] \times \R^d \times \op_2(\R^d),
\end{array}
\right.
\eea
where $\xi$ is a random variable on $\R^d$ with law $\lambda.$

\bigskip

$(ii)$ \textbf{The PPDE}. Suppose that $(f,\Phi)$ does not depend on measures:
\bea
& f(t,\ome, y, z, \mu, \nu )= H(t,\ome , y, z ),\\
& \Phi(T, \ome , \mu)= I(\ome_T ),
\eea
with $H:[0,T] \times \mC \times \R \times \R^d   \mapsto \R$ and $I: \mC \mapsto \R.$ Then equation \eqref{pde} is written as PPDE,
\bea\label{ppde}
\left\{
\begin{array}{l}
\partial_t u(t,\gam  ) + \frac12 \text{Tr}\ [ \partial_{\ome}^2 u(t,\gam)]
 + H(t,\gam, u(t,\gam), \partial_\ome u(t,\gam) )=0,\\
 \\
 u(T, \gam)= I(\gam_T ),\ \ \ \ (t,\gam )\in [0,T] \times \mC.
\end{array}
\right.
\eea
Then Theorem \ref{mainthm} recovers the well-posedness of PPDEs shown in \cite[Theorem 4.5]{PW16} under a stronger assumption in view of the integrability of BSDE \eqref{e1'}.
\bigskip

$(iii)$ \textbf{The measure-dependent master equation}. Suppose that $(f,\Phi)$ does not depend on the path/state variable and has the following structure
\bea
& f(t,\ome, y, z, \mu, \nu )= J(t,y,   \mu  ),\\
& \Phi(T, \ome , \mu)= K(\mu_T ),
\eea
where $J: [0,T] \times \R  \times \op^C_2 \mapsto \R$ and $K:\op^C_2 \mapsto \R.$ In this case our path-dependent mean-field equation is reduced to
\be
\left\{
\begin{array}{l}
\partial_t u(t,\mu  ) + \frac12 \text{Tr}\left[ \int_{\mC} \partial_{\tome} \partial_\mu u(t,  \mu, \tome )\mu(d\tome)\right]
 + J(t,  u(t, \mu),   \mu  )=0,\\
 \\
 u(T,   \mu)= K(  \mu_T ),\ \ \ \ (t, \mu)\in [0,T]  \times \op_2^C.
\end{array}
\right.
\ee
Such form of master equation is introduced in \cite{WZ20} for a closed-loop control problem with control being the form of $\alpha_t=\alpha(t,\ML_{X_t}).$
\bigskip

$(iv)$ \textbf{Path-state mixed cases}. Suppose that $(f,\Phi)$ has the following form
\bea
& f(t,\ome, y, z, \mu, \nu )= L(t,\ome, y, z,  \mu(t)    ),\\
& \Phi(T, \ome , \mu)= M(\ome_T, \mu(T)  ),
\eea
where $L:[0,T]\times\mC \times \R \times \R^d \times  \op_2(\R^d) \mapsto \R$ and $M: \mC \times \op_2(\R^d) \mapsto \R$. Then we have a unique smooth solution for the following mean-field equation
\bea\label{pathstate}
\left\{
\begin{array}{l}
\partial_t u(t,\gam, \mu ) + \frac12 \text{Tr}\ [ \partial_{\ome}^2 u(t,\gam,\mu)] + \frac12 \text{Tr}\left[ \int_{\mC} \partial_{\tome} \partial_\mu u(t,\gam, \mu, \tome )\mu(d\tome)\right] \\[2mm]
 \ \ \ \ + L(t,\gam , u(t,\gam,\mu), \partial_\ome u(t,\gam,\mu), \mu(t) )=0,\\
 \\
 u(T, \gam, \mu)= M(\gam_T, \mu(T) ), \ \ \ (t,\gam,\mu)\in \hC.
\end{array}
\right.
\eea
In view of the corresponding FBSDE \eqref{e2'}, we see that $u(t,\gam,\mu)=U(t,\gam ,\mu(t))$ for a functional $U:[0,T] \times \mC \times \op_2(\R^d) \mapsto \R.$ Since $u \in \FC^{1,2,1,1},$ $U $ is the unique classical solution to the master equation
\bea\label{pathstate1}
\left\{
\begin{array}{l}
\partial_t U(t,\gam, \lambda ) + \frac12 \text{Tr}\ [ \partial_{\ome}^2 U(t,\gam,\lambda)] + \frac12 \text{Tr}\left[ \int_{\R^d} \partial_{\tilde{a}} \partial_\lambda U(t,\gam, \lambda, \tilde{a} )\lambda(d \tilde{a})\right] \\[2mm]
\ \ \ \ + L(t,\gam , U(t,\gam,\lambda), \partial_\ome U(t,\gam,\lambda), \mu(t) )=0,\\
 \\
 U(T, \gam, \lambda)= M(\gam_T, \lambda  ), \ \ \ (t,\gam,\lambda)\in[0,T] \times \mC \times \op_2(\R^d).
\end{array}
\right.
\eea

\ \ \ \ \ \ On the other hand, if $(f,\Phi)$ has the the following structure
\bea
&& f(t,\ome, y, z, \mu, \nu )= N(t,\ome(t), y, z,  \mu,\nu    ) \quad \quad \text{and} \\
&& \Phi(T, \ome , \mu)= P(\ome(T), \mu_T )
\eea
for some  functionals $N:[0,T] \times \R^d \times \R \times \R^d \times \op_2^C \times \op_2(\R) \mapsto \R$ and $P: \R^d \times \op_2^C \mapsto \R, $ then $u(t,\ome, \mu)= V(t,\ome(t),\mu_t)$ for a functional $V:[0,T] \times \R^d \times \op_2^C \to \R.$ Then,  $V$ is the unique smooth solution to the master equation
\bea\label{pathstate2}
\left\{
\begin{array}{l}
\partial_t V(t,a, \mu ) + \frac12 \text{Tr}\ [ \partial_{a}^2 V(t,a,\mu)] + \frac12 \text{Tr}\left[  \int_{\mC} \partial_{\tome} \partial_\mu V(t,x, \mu, \tome )\mu(d \tome) \right] \\[2mm]
\ \ \ \  + N(t, a , V(t,a,\mu), \partial_a V(t,a,\mu), \mu, \ML_{V(t,W^{\mu}(t), \mu)}  )=0,\\
 \\
 V(T, a, \mu)= P(a, \mu_T  ), \ \ \ (t, a,\mu)\in[0,T] \times \R^d \times \op_2^C.
\end{array}
\right.
\eea

\bigskip

$(v)$ \textbf{A non-smooth case}. 
For any $t_0 \in (0, T)$ and $F \in C_b^3(\R^d),$ consider PPDE
\be\label{ex-pde}
\left\{
\begin{array}{l}
\partial_t u(t,\ome)+ \frac12 \text{Tr} \left[ \partial^2_{\ome} u(t,\ome)\right]=0, \\[2mm]
u(T,\ome)=F( \ome(t_0)),\ \ \ (t,\ome)\in[0,T]\times \mC.
\end{array}
\right.
\ee
In this case we have
$$
\partial_{\ome_t} F(\ome(t_0)) = DF(\ome(t_0)) 1_{[0,t_0]}(t), \quad \forall (t, \ome)\in [0,T]\times  \mC, 
$$
which is not continuous on $[0,T] \times \D$. Therefore,  {\bf(A2)} is not satisfied and the preceding PPDE has no smooth solution. In particular, when $F(x)= a x$ for some $a\in \R^d,$ by resolvability of the corresponding BSDE, the functional $u(t,\ome):= a  \ome(t \wedge t_0), (t, \ome)\in [0,T]\times  \mC,$ is the unique viscosity solution in the sense of \cite{EKTZ14}.

\ \ \ \ \ \ In a similar way, for any $(F, G)\in C_b^3(\R^d)\times  C_b^3(\R^d, \R^d)$, the path-dependent master equation
    \be\label{ex-pde2}
\left\{
\begin{array}{l}
\partial_t u(t,\mu)+ \frac12  \text{Tr} \left[ \partial_{\ome'} \partial_{\mu} u(t,\mu, \ome')\right]=0,\\[2mm]
u(T,\mu)= F\left( \E^{\mu}[G( W(t_0))] \right),\ \ \ (t,\mu)\in [0,T]\times \op_2^C
\end{array}
\right.
\ee
has no  smooth solution. In particular,  if $\Phi(T,\mu)= \E^{\mu}[aW(t_0)],$ $u(t,\mu):= \E^{\mu}[a W(t\wedge t_0)]$ is the unique viscosity solution of equation \eqref{ex-pde2} in the sense of \cite{WZ20}.\\

\section{Differentiability of solutions of path-dependent mean-field BSDEs}

\ \ \ \ \ \ In the following, for any process $(X,Y,Z)$ on the probability space $(\Omega, \MF,P)$, we denote by  $(\tX,\tY,\tZ)$ an independent copy of $(X,Y,Z)$, which means that $(\tX,\tY,\tZ)$ is defined in an independent probability space $(\tOmega, \tMF,\tP)$ and has the same law as $(X,Y,Z).$ The following linear mean-field BSDEs and estimates are frequently used in subsequent discussions. Different from a classical linear BSDE, all linear coefficients are not necessarily  bounded. For simplicity, we only address the  one-dimensional case. Similar assertions  in this section are still true  in the multi-dimensional case.


\begin{lem}\label{linear}
Let $\xi \in L^2(\MF_T)$ and $t\in[0,T)$. Suppose that $(\alpha ,\beta )\in \bH^2([t,T], \R \times \R^d )$ is bounded, $ c  \in \bH^2([t,T],\R^k)$, and $h $ is a real valued progressively measurable process such that $\int_t^T|h(r)|dr \in L^2(\MF_T).$ For any $(r, x)\in [t,T] \times \R^k ,$ $g(\cdot,  x) \in \bH^2([t,T] )$ and $g(r,  \cdot)$ is uniformly Lipschitz continuous:
$$
\sup_{r \in [t,T]}|g(r,x)-g(r,y)| \le L |x-y|, \quad  \forall y \in \R^k,\ P\text{-}a.s.
$$
for a constant $L.$ Then the following linear mean-field BSDE: $s\in[t,T],$
\be\label{e7}
Y(s) = \xi + \int_s^T \Big( \alpha(r) Y(r) + \beta(r) Z(r) + \tE[g(r, \tc(r)) \tY(r) ] + h(r) \Big) dr - \int_s^T Z(r) dB(r),
\ee
with $(\tc, \tY)$ being an independent copy of $(c, Y),$ has a unique solution $(Y, Z) \in \bS^2([t,T]) \times \bH^2([t,T],\R^d)$. Moreover,  we have
\be\label{l-bd}
\|(Y,Z )\|_{\bS^2 \times \bH^2}^2 \le C (\|\xi\|_{L^2}^2 + ||\int_t^T|h(r)| dr ||^2_{L^2} ) e^{C(\|c\|_{\bH^2}  + \|g(\cdot, 0)\|_{\bH^2}  ) }
\ee
for a constant $C$ depending on the bound of $\alpha, \beta$ and $L.$
In particular,  if $g$ is uniformly bounded, we have
\be
\|(Y,Z )\|_{\bS^2 \times \bH^2}^2 \le C(\|\xi\|_{L^2}^2 + ||\int_t^T|h(r)| dr ||^2_{L^2} ).
\ee


\end{lem}

\begin{rem} Since neither  of $g(t,x) $ and $g(r,c(r))$ is bounded or uniformly integrable for any $c(r)\in \bH^2([t,T],\R^k)$,  the well-posedness of the mean-field BSDE is not an immediate consequence of existing  works such as \cite{BLP09}.
\end{rem}

\begin{proof}

For any $y\in \bH^2, $ consider the following classical linear BSDE
\be\label{e8}
Y(s) = \xi + \int_s^T \Big( \alpha(r) Y(r) + \beta(r) Z(r) + \tE[g(r, \tc(r)) \tty(r) ] + h(r) \Big) dr - \int_s^T Z(r) dB(r),
\ee
where $(\tc,\ty)$ is an independent copy of $(c,y)$.
To prove that  it is well-posed on $[t,T]$, we only need to show  
\be
\E \left[  \int_t^T \Big|\tE[g(r, \tc(r)) \tty(r) ] \Big| dr   \right]^2 < \infty.
\ee
Indeed, by the uniformly Lipschitz continuity of $g$, we have
\be\label{gy}
\begin{split}
&\E\Big|\int_t^T |\tE[g(r, \tc(r)) \tty(r) ] | dr \Big|^2    \le C \E\Big[ \int_t^T |\tE|g(r,0) \tty(r)| dr +   \tE | \tc(r)\tty(r) | dr \Big]^2 \\
 &\ \ \le C \E \Big[\int_t^T |g(r,0) \tE[\tty(r)]| dr  \Big]^2 + C\Big[  \int_t^T \tE[ |\tc(r) \tty(r)|]dr  \Big]^2\\
 &\ \ \le  C \Big[ ( \E \int_t^T |g(r,0)|^2 dr)(\int_t^T [\tE|\tty(r)|]^2 dr) + (\int_t^T \tE|\tc(r)|^2 dr ) (\int_t^T \tE|\tty(r)|^2 dr ) \Big]\\
 & \ \ \le C \Big[ \|g(\cdot,0)\|_{\bH^2}^2  \|y\|_{\bH^2}^2 + \|c\|_{\bH^2}^2 \|y\|_{\bH^2}^2   \Big]
\le  C (\|g(\cdot,0)\|_{\bH^2}^2+\|c\|_{\bH^2 }^2) \|y\|_{\bH^2}^2,
\end{split}
\ee
where we have used   in the third  inequality  the H\"older inequality to integrals over $[t,T]$ and $[t,T]\times \Omega$. Then for any $y\in \bH^2, $ BSDE~\eqref{e8} has a unique solution $(Y, Z) \in \bH^2 \times \bH^2$. The solution mapping $\Phi: y\mapsto Y$ defines a transformation on $\bH^2$,  and turns out to  be a contraction under the following equivalent norm
\be
\|Y\|^2:= \E \int_t^T e^{As-\int_s^T (\| g(r,0) \|^2_{L^2}+ \| c(r) \|^2_{L^2}) dr } |Y(s)|^2   ds, \quad Y\in \bH^2
\ee
with  $A$ being a constant to be determined later. In fact, 
take any $y^{(j)}\in \bH^2 $ and denote  by $(Y^{(j)}, Z^{(j)})$ the corresponding solution of classical BSDE \eqref{e8},  with $j=1,2$. Set $(\Delta Y ,\Delta Z):= (Y^{(1)} - Y^{(2)}, Z^{(1)} - Z^{(2)})$, $ \Delta y :=y^{(1)} - y^{(2)} $, and $f(r):=\| g(r,0) \|^2_{L^2} + \| c(r) \|^2_{L^2} $.
Applying  It\^o's formula to $e^{As-\int_s^T f(r) dr } |\Delta Y(s)|^2$ on $s\in [t,T]$, we have
\be\nonumber
\begin{split}
  - e^{At-\int_t^T f(r) dr } |\Delta Y(t)|^2& = \int_t^T (A+ f(s) ) e^{As-\int_s^T f(r) dr } |\Delta Y(s)|^2 ds\\
 &\ \ \ \ + 2 \int_t^T e^{As-\int_s^T f(r) dr } \Delta Y d(\Delta Y) + \int_t^T e^{As-\int_s^T f(r) dr } |\Delta Z|^2 ds.
\end{split}
\ee
Therefore,
\begin{align*}
&e^{At-\int_t^T f(r) dr } |\Delta Y(t)|^2 + \int_t^T (A+ f(s) ) e^{As-\int_s^T f(r) dr } |\Delta Y(s)|^2 dr \\
&\quad + \int_t^T  e^{As-\int_s^T f(r) dr } |\Delta Z(s)|^2 dr\\
=&\quad 2 \int_t^T e^{As-\int_s^T f(r) dr } \Delta Y [\alpha \Delta Y + \beta \Delta Z + \tE[g(s, \tc(s)) \Delta \tty] ] ds\\
&\quad  - 2 \int_t^T e^{As-\int_s^T f(r) dr } \Delta Y \Delta Z dW(s)\\
\le &\quad C \int_t^T e^{As-\int_s^T f(r)dr } |\Delta Y|^2 ds + C \int_t^T e^{As-\int_s^T f(r) dr } |\Delta Y|^2 ds \\
&\quad + \frac12 \int_t^T e^{As-\int_s^T f(r) dr } |\Delta Z|^2 ds+  2 \int_t^T  e^{As-\int_s^T f(r)dr } |\Delta Y| |g(s,0)| \|\Delta \tty \|_{L^2} ds \\
& \quad +
2 \int_t^T e^{As-\int_s^T f(r)dr } |\Delta Y| \|c \|_{L^2} \|\Delta \tty \|_{L^2} ds -
2\int_t^T e^{As-\int_s^T f(r) dr } \Delta Y \Delta Z dB(s)\\
\le  &\quad  C \int_t^T e^{As-\int_s^T f(r)dr } |\Delta Y|^2 ds + \frac12 \int_t^T e^{As-\int_s^T f(r) dr } |\Delta Z|^2 ds\\
 &\quad  -2\int_t^T e^{As-\int_s^T f(r) dr } \Delta Y \Delta Z dB(s) +  2\int_t^T  e^{As-\int_s^T f(r) dr } (|\Delta Y| |g(s,0)| \|\Delta \tty \|_{L^2}  ) ds\\
  &\quad  + \int_t^T  e^{As-\int_s^T f(r) dr } (|\Delta Y|^2 \|c \|^2_{L^2} +  \|\Delta \tty \|^2_{L^2}) ds.
\end{align*}
Taking expectation on both sides of the above inequality, we have
\begin{align*}
&\int_t^T (A-C+ f(s) ) e^{As-\int_s^T f(r) dr } \|\Delta Y(s)\|_{L^2}^2 dr \\
 \le&\ 2\int_t^T  e^{As-\int_s^T f(r) dr } \E[|\Delta Y| |g(s,0)|] \|\Delta \tty \|_{L^2}   ds \\
 &+ \int_t^T  e^{As-\int_s^T f(r) dr } (\|\Delta Y\|_{L^2}^2 \|c \|^2_{L^2} +  \|\Delta \tty \|^2_{L^2}) ds \\
\le&\int_t^T  e^{As-\int_s^T f(r) dr }\Big[ \Big(\E[|\Delta Y| |g(s,0)|]\Big)^2  + \| \Delta \tty \|_{L^2}^2 \Big] ds \\
&+ \int_t^T  e^{As-\int_s^T f(r) dr } (\|\Delta Y\|_{L^2}^2 \|c \|^2_{L^2} +  \|\Delta \tty \|^2_{L^2}) ds \\
\le&\int_t^T  e^{As-\int_s^T f(r) dr } \Big( \|g(s,0)\|_{L^2}^2 +\|c \|^2_{L^2}  \Big) \|\Delta Y\|_{L^2}^2 ds \\
&+ \int_t^T  e^{As-\int_s^T f(r) dr }   \| \Delta \tty \|_{L^2}^2 ds.
\end{align*}
Therefore,  choosing a sufficiently  large number $A$ such that $A-C>1$, we obtain a contraction and then the well-posedness of \eqref{e7}.

\ \ \ \ \ \ Now BSDE \eqref{e7} can be written as the following classical BSDE
\be
Y(s) = \xi + \int_s^T \Big( \alpha(r) Y(r) + \beta(r) Z(r)  + h'(r)  \Big) dr - \int_s^T Z(r) dB(r),
\ee
with
$
h'(r)= \tE[g(r, \tc(r)) \tY(r) ]+ h(r).
$
Thus it is standard that
\be
\begin{split}
\|Y\|_{\bS^2}^2 + \|Z\|_{\bH^2}^2  &\le C(\|\xi\|_{L^2}^2+ \|\int_t^T|h'(r)|dr\|^2_{L^2})\\
&\le C(\|\xi\|_{L^2}^2+ \|\int_t^T|h(r)|dr\|^2_{L^2} +\|\int_t^T| \tE[g(r, \tc(r)) \tY(r) ] |dr\|^2_{L^2} ).
\end{split}
\ee
Furthermore, similar to the proof of inequality \eqref{gy}, we have
\be
\|\int_t^T| \tE[g(r, \tc(r)) \tY(r) ] |dr\|^2_{L^2} \le C[ \int_t^T \|g(r,0)\|_{L^2}^2 \|Y \|^2_{\bS^2, [t,r]} dr +  \int_t^T \|c(r)\|_{L^2}^2 \|Y \|^2_{\bS^2, [t,r]} dr ].
\ee
Then, using Gronwall's inequality,  we obtain the desired estimate~\eqref{l-bd}.

%
%
%

\end{proof}

\ \ \ \ \ \ To study the differentiability of corresponding FBSDEs, for any $(t,\eta)\in[0,T]\times \MM^D_2$, we denote by $(Y^{\eta_t}, Z^{\eta_t})$ the solution of the following path-dependent mean-field BSDE, for $s\in [t,T],$
\be \label{e1}
Y (s)= \Phi(B_T^{\eta_t}, \ML_{B_T^{\eta_t} }) + \int_s^T f( B_r^{\eta_t}, Y (r),  Z (r)  , \ML_{B_r^{\eta_t}},  \ML_{Y (r) }) dr- \int_s^T Z (r) dB(r).
\ee
On the other hand, for any $\gamma \in \D,$ let $(Y^{\gamma_t, \eta_t} , Z^{\gamma_t,\eta_t} )$ solve the associated path-dependent BSDE: for $s\in [t,T],$
\be \label{e2}
\MY{(s)}= \Phi(B_T^{\gamma_t}, \ML_{B_T^{\eta_t} }) + \int_s^T\!\!\! f( B_r^{\gamma_t}, \MY(r), \MZ(r) ,  \ML_{B_r^{\eta_t}}, \ML_{Y^{\eta_t}(r)}  ) dr -  \int_s^T \MZ(r) dB(r).
\ee

Note that under Assumption {\bf (H0)}, the functional
\be
\hat{f}(r, y, z, \nu):= f( B_r^{\eta_t}, y , z , \ML_{B_r^{\eta_t}}, \nu ), \quad  (r,y,z,\nu)\in[t,T]\times \R \times \R^d \times \op_2(\R),
\ee
is uniformly Lipschitz continuous in $(y,z)\in \R \times \R^d$. According to \cite[Theorem 4.23]{CD1}, BSDE \eqref{e1} is well posed with $(Y^{\eta_t},Z^{\eta_t}, \ML_{Y^{\eta_t}}) \in \bS^2 \times \bH^2 \times \op_2(\R).$ Then \eqref{e2} is a well-defined  classical BSDE with $(Y^\gl, Z^\gl) \in \bS^p \times \bH^p$ for any $p\ge 1.$ In the following, we write $\Theta^{\eta_t}_r:=( B_r^{\eta_t}, Y^{\eta_t}(r),  Z^{\eta_t}(r)),  \Theta^{\gamma_t, \eta_t}_r:=(B_r^{\gamma_t}, Y^{\gamma_t, \eta_t}(r), Z^{\gamma_t,\eta_t}(r))$, $\ML_{\Theta^{\eta_t}_r}:=(\ML_{B_r^{\eta_t}}, \ML_{Y^{\eta_t}(r)})$ and $(Y,Z):=(Y(t), Z(t))$ if no confusion is raised.
Then we have the following basic estimates for BSDEs \eqref{e1} and \eqref{e2}.

\begin{lem}\label{unibd}
Assume that $(\Phi, f)$ satisfies {\bf (H0)}. For any $K>0$ and
 $(\gamma, \eta), ( \gamma' , \eta')  \in \D \times  \MM_2^D $ such that $|||\ML_{\eta_t}|||, |||\ML_{\eta'_t}||| \le K,$ we have for any $p\ge 1,$
\bea \label{eta1}
&&\|  (Y^{  \eta_t} , Z^{ \eta_t}) \|_{\bS^2 \times \bH^2} \le C(1+    \| \eta_t  \|_{\bS^2} ), \\ \label{gam1}
&&\|  (Y^{\gamma_t, \eta_t} ,  Z^{\gamma_t, \eta_t}) \|_{\bS^p \times \bH^p} \le C_p(1+ \| \gamma_t\| +  \| \eta_t  \|_{\bS^2} ), \\
 \label{eta2}
 &&\| ( Y^{  \eta_t} -Y^{  \eta'_t}  , Z^{ \eta_t} - Z^{  \eta'_t}  ) \|_{\bS^2 \times \bH^2} \le C_K  \| \eta_t-\eta'_t  \|_{\bS^2},  \quad\quad \text{and }\\ \label{gam2}
&&\| ( Y^{ \gamma_t, \eta_t} -Y^{ \gamma'_t,  \eta'_t},  Z^{ \gamma_t, \eta_t} - Z^{ \gamma'_t,  \eta'_t}  ) \|_{\bS^p \times \bH^p} \le C_{K,p} ( \|\gamma_t-\gamma'_t\| + W_2(\ML_{\eta_t} , \ML_{\eta'_t}  ) ),
\eea
where $(C ,C_p)$ does not depend on $(\gam,\eta)$, and $(C_K, C_{K,p})$ does not depend on $(\gam,\gam')$.

\end{lem}

\begin{rem}\label{law-dep}
According to inequality \eqref{gam2}, $( Y^{ \gamma_t, \eta_t}, Z^{ \gamma_t, \eta_t})$ and $(Y^{ \gamma_t, \eta'_t}, Z^{ \gamma_t, \eta'_t})$ are indistinguishable if $\ML_{\eta_t}=\ML_{\eta'_t} ,$ which implies the following definition is well-posed
\be
(Y^{ \gamma_t, \ML_{\eta_t}}, Z^{ \gamma_t, \ML_{\eta_t}} ) := ( Y^{ \gamma_t, \eta_t}, Z^{ \gamma_t, \eta_t}).
\ee

\end{rem}

The proof of Lemma \ref{unibd} is rather lengthy, which follows from Lemma \ref{linear} and is left in the appendix.

\subsection{First-order differentiability}

In this subsection we assume that {\bf(H1)} holds for $(\Phi, f).$
For any $(\gam,\eta)\in[0,T]\times \MM^D_2,$ we consider the first order differentiability of $Y^{\gl}=Y^{\gam_t,\ML_{\eta_t}}$ with respect to $\gam_t$ and $\ML_{\eta_t}.$ For the differentiability in $\gam_t,$
let
\be
\begin{split}
&\hat{f}(\ome_s,y,z  ):= f (\ome_s, y , z , \ML_{B_s^{\eta_t}}, \ML_{Y^{\eta_t}(s)}  ),\\
&   \hat{\Phi}(\ome_T  )  :=  \Phi(\ome_T , \ML_{B_T^{\eta_t} }), \  \forall (s,\ome,y,z)\in[t,T] \times \D \times \R \times \R^d,
 \end{split}
\ee
and then the solution $Y^{\gamma_t, \eta_t}{(s)}$ to equation \eqref{e2} solves the following path-dependent BSDE
\be\label{e3}
\hY{(s)}= \hat{\Phi}(B_T^{\gamma_t} )  + \int_s^T \hat{f}(B^{\gamma_t}_r, \hY(r ), \hZ(r  )  ) dr-  \int_s^T \hZ (r) dB(r).
\ee
Define
$
\hat{u}_{ \eta_t } (t , {\gamma} ) := Y^{\gamma_t, \eta_t} (t).
$ If $f$ and $\Phi$ are regular enough, according to \cite[Theorem 4.5]{PW16} , $\hat{u}_{\eta_t} (t  , {\gamma} ) $ is twice vertically differentiable at $(t, \gamma),$ and moreover for any $s \ge t,$
\be
\hu_{ \eta_t }(s, B^{\gamma_t})= Y^{\gamma_t, \eta_t} (s ) , \ \ \  \partial_{\gamma_t} \hu_{ \eta_t }(s, B^{\gamma_t}) = Z^{\gamma_t, \eta_t} (s ) .
\ee
Furthermore, $\hu_{ \eta_t }(t, \gamma)$ is the unique solution to the following semilinear PPDE
\be
\left\{
\begin{array}{l}
\partial_t \hu_{ \eta_t }(t, \gamma) + \frac12 \text{Tr}\left[  \partial_{\omega }^2 \hu_{ \eta_t }(t, \gamma)\right] + \hat{f}(\gamma_t , \hu_{ \eta_t }(t, {\gamma}) , \partial_\omega \hu_{ \eta_t }(t,  {\gamma})  )=0, \\[2mm]
 \hu_{ \eta_t }(T, \gamma ) = \hat{\Phi}( \gamma ),\ \ \ (t,\gam)\in [0,T]\times \mC.
\end{array}
\right.
\ee

In the following, we denote by $\partial_{(t,\ome,y,z,\mu,\nu, {\ome_\tau}, {\mu_\tau})}f$ the derivative vector 
$$(\partial_t f, \partial_\ome f, \partial_y f, \partial_z f, \partial_\mu f,\partial_\nu f, \partial_{\ome_\tau} f, \partial_{\mu_\tau}f ).$$

\begin{rem}\label{H2}
Assume that $\Phi: \D \mapsto \R$ is twice continuously strongly vertically differentiable and satisfies the following locally Lipschitz continuous condition: for any $t\in[0,T]$ and $\phi= \Phi, \partial_{\ome_t} \Phi,  \partial_{\ome_t}^2 \Phi $,
\be\label{weakphi}
|\phi(\ome_T)-\phi(\ome'_T)| \le C(1+\| \ome_T |^k+\|\ome'_T\|^k) \|\ome_T- \ome'_T\|, \quad \forall \ (\ome, \ome') \in \D^2.
\ee
Then, the main result \cite[Theorem 4.5]{PW16}  is still true.  
For the reader's convenience,   the proof is sketched  in the appendix,  using  our partial It\^o-Dupire formula.

\end{rem}

\begin{prop}\label{s-bdd1}
Let  $(f,\Phi)$ satisfy Assumption {\bf(H1)}. Then for any $\tau \le t,$ $(Y^{\gl}(s),Z^{\gl}(s))$ is strongly vertically differentiable at $(\tau,t,\gam)$. The derivative $(\partial_{\ome_\tau}Y^{\gl}, \partial_{\ome_\tau}Z^{\gl}) \in \bS^p([t,T],\R^d) \times \bH^p([t,T],\R^{d\times d}), \  \forall\ p\ge1$, is the unique solution to BSDE
\be\label{s-e4}
\begin{split}
\MY(s)  =\ &  \partial_{\ome_\tau} \Phi (B^{\gam_t}, \ML_{B^{\eta_t}})  + \int_s^T   \partial_{\ome_\tau} f(\Theta^\gl_r, \ML_{\Theta_r^{\eta_t}}) dr + \int_s^T  \partial_y f(\Theta^\gl_r, \ML_{\Theta_r^{\eta_t}}) \MY (r) dr\\
&+ \int_s^T \partial_z f(\Theta^\gl_r, \ML_{\Theta_r^{\eta_t}})\MZ  (r) dr
- \int_s^T  \MZ (r) d B(r), \ \ s\in[t,T].
\end{split}
\ee
Furthermore, since $(\partial_{\ome_\tau}Y^{\gl}, \partial_{\ome_\tau}Z^{\gl})$ is independent of $\MF_t$, we have that for any $K>0,$ and any $(  \gam , \eta), (\gam' , \eta')\in  \D \times \MM_2^D \text{ such that } |||\ML_{\eta_t}|||, |||\ML_{\eta'_t}||| \le K,$
\be\label{3.7}
\begin{split}
 &\|( \partial_{\ome_\tau}Y^{\gl} , \partial_{\ome_\tau}Z^{\gl}) \|_{\bS^p \times \bH^p}  <C_p ,\\
 &\| (\partial_{\ome_\tau}Y^{\gl}- \partial_{\ome_\tau} Y^{\gam'_t, \eta'_t}, \partial_{\ome_\tau}Z^{\gl}-\partial_{\ome_\tau} Z^{\gam'_t, \eta'_t}) \|_{\bS^p \times \bH^p}  < C_{K,p}(\|\gam_t-\gam'_t\| + W_2(\ML_{\eta_t} , \ML_{\eta'_t}  ) ),
\end{split}
\ee
for some positive constants $C_p$ and $C_{K,p}$.

\end{prop}

\begin{proof}
In view of Assumption {\bf(H1)} and Lemma \ref{linear},
we see that equation \eqref{s-e4} has a unique solution $(\partial_{\ome_\tau}Y, \partial_{\ome_\tau}Z) \in \bS^p \times \bH^p, \ \forall p\ge 1.$
Here,  we  consider the one-dimensional case for simplicity. For any $h>0,$ recall that $\gam^{\tau,h}=\gam+h 1_{[\tau,T]}$.  Set
\be
\begin{split}
&\gam':=\gam^{\tau,h}, \quad \Delta_h Y:=\frac1h (Y'-Y)  := \frac1h (Y^{\gam^{\tau,h}_t,\eta_t} - Y^{\gl}), \quad \text{and}\\
& \quad   \Delta_h Z:=\frac1h (Z'-Z)  := \frac1h (Z^{\gam^{\tau,h}_t,\eta_t} - Z^{\gl}).
\end{split}
\ee
Then we know that $(\Delta_h Y, \Delta_h Z)$ solves the following BSDE
\be\nonumber
\begin{split}
\Delta_h Y (s) &=  \frac1h(\Phi'-\Phi) + \frac1h \int_s^T [f(\Theta^{\gam',\eta}_r, \ML_{\Theta^{ \eta}_r} ) - f(\Theta^{\gam,\eta}_r, \ML_{\Theta^{  \eta}_r} ) ] dr - \int_s^T \Delta_h Z(r) d B(r) \\
&= : \Delta_h \Phi + \int_s^T \Big( a_r \Delta_h Y(r) + b_r \Delta_h Z(r)  + \Delta_h  f \Big) dr  - \int_s^T \Delta_h Z(r) d B(r) ,
\end{split}
\ee
where
\beaa
&&\Phi':=\Phi(B^{\gam'},\ML_{B^\eta}),   \ \ \Phi:=\Phi(B^{\gam},\ML_{B^\eta}), \ \
\Delta_h \Phi:= \int_0^1 \partial_{\gam_{\tau} } \Phi(B^{\gam^{\tau,h\theta}}, \ML_{B^{\eta}} )\ d\theta, \\
&&a_r := \int_0^1 \partial_y {f} (B^{\gam'}_r , Y+ \theta(Y'-Y), Z', \ML_{\Theta^\eta_r} )\ d\theta,  \ \
b_r :=  \int_0^1 \partial_z {f} (B^{\gam'}_r , Y , Z+\theta(Z'- Z), \ML_{\Theta^\eta_r} )\ d\theta ,\\
&& \quad \text{and} \quad \Delta_h f :=\frac{1}{h} f( B^{\ome}_r ,  Y, Z, \ML_{B^{\eta}_r} , \ML_{Y}  )\Big|_{\ome=\gam}^{\ome=\gam'}   = \int_0^1 \partial_{\ome_\tau} f(B^{\gam^{\tau,h\theta}}, Y, Z, \ML_{B^{\eta}_r} , \ML_{Y} )\ d\theta.
\eeaa
Then $(\delta Y, \delta Z):=(\Delta_h Y - \partial_{\ome_\tau} Y, \Delta_h Z - \partial_{\ome_\tau} Z)$ satisfies BSDE
\be\nonumber
\begin{split}
\delta Y(s)=&\ (\Delta_h \Phi - \partial_{\ome_\tau}\Phi)+ \int_s^T \left(a_r \delta Y + b_r \delta Z+ (\Delta_h f- \partial_{\ome_\tau} f(\Theta^\gl_r, \ML_{\Theta_r^{\eta_t}}))\right) dr\\
\ \ & + \int_s^T [ (a_r-\partial_y f(\Theta^\gl_r, \ML_{\Theta_r^{\eta_t}}))\partial_{\ome_\tau} Y + (b_r-\partial_z f (\Theta^\gl_r, \ML_{\Theta_r^{\eta_t}})) \partial_{\ome_\tau}Z ] dr- \int_s^T \delta Z dB(r).
\end{split}
\ee

According to standard estimate for BSDEs (or Lemma \ref{linear} for $p=2$) and Lemma \ref{unibd}, we have
\be\nonumber
\begin{split}
\|\delta  Y     \|_{\bS^p}^p+ \|\delta Z \|_{\bH^p}^p & \le C \| \Delta_h \Phi - \partial_{\ome_\tau}\Phi \|_{L^p}^p   + \|\int_t^T |\Delta_h f- \partial_{\ome_\tau} f(\Theta^\gl_r, \ML_{\Theta_r^{\eta_t}})| dr\|_{L^p}^p + O(|h|)\\
& \le   O(|h|),
\end{split}
\ee
and thus the strongly vertical differentiability.

\end{proof}

\ \ \ \ \ \ To show the differentiability of $Y^{\gamma_t, \eta_t}$ with respect to $\eta_t,$ we follow a similar argument as in the state-dependent case for SDEs made in \cite{BLP17}. Firstly we show that $Y^{\gam_t,\eta_t}$ is G\^ateaux differentiable in $\eta_t$ in the sense of \eqref{gat} and Remark \ref{d-ext}. To this end, we need to prove that for any $\xi \in L^2(\MF_t, \R^d)$ and $\eta_t^{\lambda \xi}:= \eta_t+ \lambda \xi 1_{[t,T]}$, $\lambda >0,$ the following limit exits in $\bS^2([t,T],\R^d),$
\be
\partial_\eta Y^{\gamma_t, \eta_t} (\xi) := \lim_{\lambda \rightarrow 0} \frac{1}{\lambda}(Y^{\gamma_t, \eta_t^{\lambda \xi}}- Y^{\gam_t, \eta_t}).
\ee
Then we show that $\partial_\eta Y^{\gamma_t, \eta_t} (\cdot): L^2(\MF_t,\R^d) \mapsto \bS^2([t,T],\R^d)$ is a bounded linear operator, and moreover, it is continuous in the following sense: for any $\zeta \in L^2(\MF_t,\R^d), $ $\partial_\eta Y^{\gam_t,\eta_t+ \zeta 1_{[t,T]}}$ converges to $\partial_\eta Y^{\gam_t,\eta_t}$ in the sense of operators as $\zeta$ goes to zero. In view of Remark \ref{fregat}, we see that $ Y^{\gam_t,\eta_t} $ is Fr\'echet (vertically) differentiable in the sense of \eqref{fre} and Remark \ref{d-ext}. To this end, consider the following linear BSDE
\be\label{e5'}
\begin{split}
\MY^{\gl ,\xi}(s) &= \tE [ \partial_{\mu_t} \Phi (B^{\gam_t}, \ML_{B^{\eta_t}}, \tB^{\teta_t}) \txi ] + \int_s^T \tE [ \partial_{\mu_t} f(\Theta^\gl_r, \ML_{\Theta_r^{\eta_t}}, \tB^{\teta_t} ) \txi ] dr \\
&\ \ \ + \int_s^T  \partial_y f(\Theta^\gl_r, \ML_{\Theta_r^{\eta_t}}) \MY^{\gl ,\xi} (r) dr+ \int_s^T  \partial_z f(\Theta^\gl_r, \ML_{\Theta_r^{\eta_t}}) \MZ^{\gl ,\xi}  (r) dr\\
&\ \ \ +  \int_s^T   \tE [ \partial_{\nu} f(\Theta^\gl_r, \ML_{\Theta_r^{\eta_t}}, \tY^{\teta_t} ) ( \partial_{\ome_t} \tY^{\teta_t, \ML_{\eta_t}} \txi  + \tilde{\MY}^{\teta_t, \txi})(r) ] dr\\
& \ \ \
- \int_s^T  \MZ^{\gl ,\xi}(r) d B(r), \quad s\in [t,T].
\end{split}
\ee
Here,  $(\tB, \teta, \txi, \tY^{\teta},\partial_{\ome_t} \tY^{\teta_t, \ML_{\eta_t}},\tMY^{\teta_t, \txi} )$ is an independent copy of $(B, \eta, \xi, Y^{\eta_t},\partial_{\ome_t} Y^{\gam_t, \ML_{\eta_t}}|_{\gam=\eta}, \MY^{\eta_t, \xi})$,
and $\MY^{\eta_t, \xi}$ satisfies the following linear mean-field BSDE
\be\label{e6}
\begin{split}
\MY^{\eta_t, \xi}(s) &= \tE [ \partial_{\mu_t} \Phi (B^{\eta_t}, \ML_{B^{\eta_t}}, \tB^{\teta_t}) \txi ] + \int_s^T \tE [ \partial_{\mu_t} f(\Theta^{\eta_t}_r, \ML_{\Theta_r^{\eta_t}}, \tB^{\teta_t}_r ) \txi ] dr \\
&\ \ \  + \int_s^T  \partial_y f(\Theta^{\eta_t}_r, \ML_{\Theta_r^{\eta_t}}) \MY^{{\eta_t} ,\xi} (r) dr+ \int_s^T  \partial_z f(\Theta^{\eta_t}_r, \ML_{\Theta_r^{\eta_t}}) \MZ^{{\eta_t} ,\xi}  (r) dr\\
&\ \ \ + \int_s^T   \tE [ \partial_{\nu} f(\Theta^{\eta_t}_r, \ML_{\Theta_r^{\eta_t}}, \tY^{\teta_t} ) ( \partial_{\ome_t} \tY^{\teta_t, \ML_{\eta_t}} \txi  + \tilde{\MY}^{\teta_t, \txi})(r) ] dr\\
&\ \ \
- \int_s^T  \MZ^{{\eta_t} ,\xi}(r) d B(r), \quad s\in [t,T].
\end{split}
\ee



\begin{lem}\label{e6est}
For any $\xi \in L^2(\MF_t,\R^d),$
there exits a unique solution $(\MY^{\eta_t, \xi}, \MZ^{\eta_t, \xi}) \in \bS^2([t,T]) \times \bH^2([t,T],\R^d)$ to BSDE \eqref{e6}. Moreover, $(\MY^{\eta_t, \xi}, \MZ^{\eta_t, \xi})$ is linear in $\xi$,  and we have
\be\label{est}
\| (\MY^{\eta_t, \xi} , \MZ^{\eta_t, \xi}) \|_{\bS^2 \times \bH^2} \le C \|\xi \|_{L^2}
\ee
for some constant $C$.

\end{lem}

\begin{proof}

By Lipschitz continuity of $(\partial_{\mu_t} \Phi, \partial_{\mu_t} f), $ we have
$$
\tE [ \partial_{\mu_t} \Phi (B_T^{\eta_t}, \ML_{B_T^{\eta_t}}, \tB_T^{\teta_t}) \txi ]\in L^2(\MF_T),\ \ \
\tE [ \partial_{\mu_t} f(\Theta^{\eta_t}_r, \ML_{\Theta_r^{\eta_t}}, \tB^{\teta_t}_r ) \txi ] \in  L^2(\MF_r).
$$
Since $f$ is uniformly Lipschitz continuous in $(y,z),$ $\partial_{(y,z)} f(\Theta^{\eta_t}_r, \ML_{\Theta_r^{\eta_t}})$ is uniformly bounded.
Set $g(r,x):= \partial_{\nu} f(\Theta^{\eta_t}_r, \ML_{\Theta_r^{\eta_t}},x )$. In view of Lemma \ref{unibd} and Assumption {\bf(H1)}, we see that $g(\cdot,0)\in \bH^2.$
Then by Lemma \ref{linear}, to show the well-posedness of linear mean-field BSDE \eqref{e6},
we only need to check the following
$$\int_t^T \left|\tE [ \partial_{\nu} f(\Theta^{\eta_t}_r, \ML_{\Theta_r^{\eta_t}}, \tY^{\teta_t} ) ( \partial_{\ome_t} \tY^{\teta_t, \ML_{\eta_t}} \txi)] \right|dr \in L^2(\MF_T).$$
 Let
 \be
 F_2(t, x,y,z, \mu, \nu  ):= \tE [ \partial_{\nu} f(t, x,y,z, \mu, \nu , \tY^{\teta_t}(r) ) ( \partial_{\ome_t} \tY^{\teta_t, \ML_{\eta_t}}(r) \txi)].
 \ee
Then by Lipschitz continuity of $\partial_\nu f$ and Proposition \ref{s-bdd1}, we have
\begin{align*}
 &F_2(t, x,y,z, \mu, \nu  )\\
 &\ \ = \tE \Big[ \tE_{\tMF_t} [ \partial_{\nu} f(t, x,y,z, \mu, \nu , \tY^{\gam_t, \ML_{\eta_t}}(r) ) ( \partial_{\ome_t} \tY^{\gam_t, \ML_{\eta_t}} (r) )]\Big|_{\gam_t=\teta_t}   \txi \Big]\\
 &\ \  \le C \tE \Big[ \tE_{\tMF_t} [  |\tY^{\gam_t, \ML_{\eta_t}}| | \partial_{\ome_t} \tY^{\gam_t, \ML_{\eta_t}}(r) | ]\left|_{\gam_t=\teta_t} \right. \txi \Big] + \partial_{\nu} f(t, x,y,z, \mu,\nu, 0)\tE\Big[ \tE_{\tMF_t} [\tY^{\gam_t, \ML_{\eta_t}}(r) | ]\left|_{\gam_t=\teta_t} \right. \txi \Big]  \\
 &\ \ \le C \tE\left[ \tE_{\tMF_t}[ (1+\|\gamma_t\|) ]  \Big|_{\gam_t=\teta_t} \txi   \right] + C \partial_{\nu} f(t, x,y,z, \mu,\nu, 0)\\
  &\ \ \le C  + C \partial_{\nu} f(t, x,y,z, \mu,\nu, 0),
\end{align*}
where we have applied Lemmas \ref{unibd} in the second  inequality. Then according to Lemma \ref{unibd} again, we have $\partial_{\nu} f(\Theta^{\eta_t}_r, \ML_{\Theta^{\eta_t}_r},0 ) \in \bH^2,$  and thus the well-posedness of \eqref{e6}. For inequality \eqref{est}, similar to the proof of Lemma \ref{unibd}, we have
\begin{align*}
&\| \MY^{\eta_t, \xi} \|_{\bS^2}^2+\| \MZ^{\eta_t, \xi} \|^2_{\bH^2}\\
& \ \ \le C   \E\Big[|\tE [ \partial_{\mu_t} \Phi (B^{\eta_t}, \ML_{B^{\eta_t}}, \tB^{\teta_t}) \txi ]|^2 +\int_s^T |\tE [ \partial_{\mu_t} f(\Theta^{\eta_t}_r, \ML_{\Theta_r^{\eta_t}}, \tB^{\teta_t} ) \txi ]|^2 dr   \\
& \ \ \ \ \ \ + \int_s^T   |\tE  \partial_{\nu} f(\Theta^{\eta_t}_r, \ML_{\Theta_r^{\eta_t}}, \tY^{\teta_t} ) ( \partial_{\ome_t} \tY^{\teta_t, \ML_{\eta_t}} \txi)|^2 dr \Big]\\
& \ \ \le C \Big(  (\tE\|\tB^{\teta_t}\| \ |\txi|)^2 + \|\xi \|_{L^2}^2 \E|\partial_{\mu_t} \Phi (B^{\eta_t}, \ML_{B^{\eta_t}},  0 )|^2\\
& \ \ \ \ \ \ + \|\xi \|_{L^2}^2 \E \int_s^T |\partial_{\mu_t} f(\Theta^{\eta_t}_r, \ML_{\Theta_r^{\eta_t}},0  ) |^2 dr + \E \int_s^T | \tE[\tY^{\teta_t}  \partial_{\ome_t} \tY^{\teta_t, \ML_{\eta_t}} \txi ] |^2 dr\\
  &\ \ \ \ \ \ + \|\xi \|_{L^2}^2 \E \int_s^T |\partial_{\nu} f(\Theta^{\eta_t}_r, \ML_{\Theta_r^{\eta_t}},0)|^2 dr \Big)
 \le C  \|\xi \|_{L^2}^2 .
\end{align*}

\end{proof}

Since BSDE \eqref{e6} is well-posed, so is BSDE \eqref{e5'}.
In conclusion, we have
\begin{coro}
There exits a unique solution $(\MY^{\gam_t,\eta_t, \xi}, \MZ^{\gam_t,\eta_t, \xi}) \in \bS^2([t,T]) \times \bH^2([t,T],\R^d)$ to BSDE \eqref{e5'}. Moreover,
\be
(\MY^{ \eta_t, \xi}, \MZ^{\eta_t, \xi})=(\MY^{\gam_t,\eta_t, \xi}, \MZ^{\gam_t,\eta_t, \xi})|_{\gam=\eta}.
\ee
\end{coro}

\begin{lem}\label{gat1} The map 
$\xi\mapsto \MY^{\gl, \xi}$ is a bounded linear operator from $L^2({\MF_t},\R^d)$ to $\bS^2([t,T])$. Moreover, it is the G\^ateaux derivative of $Y^{\gamma_t, \eta_t}$ with respect to $\eta_t $ in the following sense
\be
\MY^{\gl, \xi} =  \lim_{\lambda \rightarrow 0} \frac{1}{\lam}(Y^{\gamma_t, \eta_t^{\lambda \xi}}- Y^{\gam_t, \eta_t}) \quad  \text{strongly in $\bS^2([t,T])$}.
\ee
In particular, $\MY^{\gl, \xi}(s)$ is the G\^ateaux derivative of $Y^{\gamma_t, \eta_t}(s)$ in the sense of \eqref{gat}.
\end{lem}

\begin{proof} Since $\MY^{\eta_t, \xi} $ is linear in $\xi,$ we see that $(\MY^{\gl ,\xi},\  \MZ^{\gl ,\xi})$ is also linear in $\xi$.  Moreover,  we have the following estimate
	\be\label{y-linear}
	\| (\MY^{\gl, \xi} , \MZ^{\gl, \xi} )\|_{\bS^2\times \bH^2} \le C \|\xi \|_{L^2}.
	\ee
Therefore,  we have the first assertion. 
	
\ \ \ \ \ \ 	In the following, we omit the fixed subscript $t$ and write $(Y,Z):=(Y(r),Z(r))$ if no confusion raised. Besides, the constant $C$ may change from line to line.  Set
\be
\begin{split}
&\Delta_\lambda Y:= \frac{1}{\lam}(Y^{\gam, \eta^{\lambda\xi}} - Y^{\gam,\eta}),\ \ \Delta_\lambda Z:= \frac{1}{\lam}(Z^{\gam, \eta^{\lambda\xi}} - Z^{\gam,\eta}), \quad \quad \text{and }\\
  &\quad \Delta_\lambda \Phi := \frac{1}{\lambda} [\Phi(B_T^{\gamma}, \ML_{B_T^{\eta^{\lambda \xi}} })- \Phi(B_T^{\gamma}, \ML_{B_T^{\eta } } )].
\end{split}
\ee
Then according to Lemma \ref{unibd}, we have
\be\label{bd-dy}
\| \Delta_\lambda Y \|_{\bS^2} + \| \Delta_\lambda Z \|_{\bH^2} \le C \frac{1}{\lam} \|\eta_t^{\lam \xi} - \eta_t \|_{\bS^2}\le C \|\xi \|_{L^2}.
\ee
In view of BSDE \eqref{e2}, we see that $(\Delta_\lambda Y, \Delta_\lambda Z)$ satisfies the following linear mean-field BSDE
\be\label{e-diff}
\begin{split}
\Delta_\lambda Y  &=   \Delta_\lambda \Phi + \int_s^T \left[ \alpha(r) \Delta_\lambda Y + \beta(r) \Delta_\lambda Z + \tE[ \tilde{g}(r) \frac{1}{\lam}(\tY^{\tilde{\eta}^{\lam \txi}} - \tY^{\teta})] + \Delta_\lambda f \right] dr \\
&\quad - \int_s^T \Delta_\lambda Z dB(r),
\end{split}
\ee
where
\begin{align*}
&\alpha(r):= \int_0^1 \partial_y {f} (B^{\gamma}_r , Y^{\gam,\eta}+ \theta(Y^{\gam,\eta^{\lambda \xi}}-Y^{\gam,\eta}), Z^{\gam,\eta^{\lambda \xi}}, \ML_{\Theta^{\eta^{\lambda \xi}}_r} ) d\theta, \\
&\beta(r):=  \int_0^1 \partial_z {f} (B^{\gamma}_r , Y^{\gam,\eta} , Z^{\gam,\eta}+ \theta(Z^{\gam,\eta^{\lambda \xi}}-Z^{\gam,\eta}), \ML_{\Theta^{\eta^{\lambda \xi}}_r} ) d\theta, \\
&\tilde{g}(r):= \int_0^1 \partial_{\nu} f ( \Theta^{\gam,\eta}, \ML_{B^{\eta^{\lambda \xi}}}, \ML_{Y^{ \eta}+ \theta(Y^{ \eta^{\lambda \xi}}-Y^{ \eta})} ,   \tY^{ \teta}+ \theta(\tY^{ \teta^{\lambda \txi}}-\tY^{ \teta}) ) d \theta, \quad\quad \text{and} \\
&\Delta_\lam f(r): = \frac{1}{\lam} [f( \Theta^{\gam,\eta}, \ML_{B^{\eta^{\lambda \xi}}_r} , \ML_{Y^{ \eta}}  ) - f( \Theta^{\gam,\eta}, \ML_{B^{\eta}_r} , \ML_{Y^{ \eta}}  ) ] .
\end{align*}
According to estimate \eqref{eta2} in Lemma \ref{unibd}, we have
\be\label{bd-yeta}
\| \Delta_\lam \tY^{\teta}\|_{\bS^2}:= \|\frac{1}{\lam}(\tY^{\tilde{\eta}^{\lam \txi}} - \tY^{\teta}) \|_{\bS^2} \le C \frac{1}{\lam} \|\eta_t^{\lam \xi} - \eta_t\|_{\bS^2} \le C \| \xi \|_{L^2}.
\ee
Then,  in view of Assumption {\bf(H1)}, we have
\be
\|\int_t^T \tE[ \tilde{g}(r) \Delta_\lam \tY^{\teta}] dr\|_{L^2} + \| \int_t^T \Delta_\lambda f dr \|_{L^2} \le C\|\xi\|_{L^2}.
\ee
Thus BSDE \eqref{e-diff} has a unique solution $(\Delta_\lam Y, \Delta_\lam Z)$, and then $(\Delta_\lambda Y - \MY^{\gl, \xi}, \Delta_\lambda Z -\MZ^{\gl, \xi})$ is the unique  solution of the following BSDE
\begin{align*}
Y(s) =& \  ( \Delta_\lambda \Phi - \tE [ \partial_{\mu_t} \Phi (B^{\gam_t}, \ML_{B^{\eta_t}}, \tB^{\teta_t}) \txi ])  +  \int_s^T  \partial_y f(\Theta^\gl_r, \ML_{\Theta_r^{\eta_t}}) Y dr \\ \nonumber
&\ + \int_s^T  \partial_z f(\Theta^\gl_r, \ML_{\Theta_r^{\eta_t}}) Z  (r) dr + \int_s^T ( \Delta_\lam f - \tE [ \partial_{\mu_t} f(\Theta^\gl_r, \ML_{\Theta_r^{\eta_t}}, \tB^{\teta_t} ) \txi ] ) dr \\ \nonumber
&\ +  \int_s^T   \tE [ \partial_{\nu} f(\Theta^\gl_r, \ML_{\Theta_r^{\eta_t}}, \tY^{\teta_t} ) (\Delta_\lam \tY^{\teta} - \partial_{\ome_t} \tY^{\teta_t, \ML_{\eta_t}} \txi  - \tilde{\MY}^{\teta_t, \txi})  ] dr\\
&\ +  \int_t^T R_1(r) dr - \int_s^T  Z d B(r)
\end{align*}
with
\begin{align*}
R_1(r) :=& \left(\alpha(r)- \partial_y f(\Theta^\gl_r, \ML_{\Theta_r^{\eta_t}}) \right) \Delta_\lam Y + \left(\beta(r) - \partial_z f(\Theta^\gl_r, \ML_{\Theta_r^{\eta_t}}) \right) \Delta_\lam Z \\
& + \tE\left[( \tilde{g}(r)- \partial_{\nu} f(\Theta^\gl_r, \ML_{\Theta_r^{\eta_t}}, \tY^{\teta_t} )) \Delta_\lam \tY^{\teta} \right].
\end{align*}
Since $\partial_{(y,z)} f$ is bounded, from the standard estimate for solutions of BSDEs, we have
\be\label{bd-dely}
\|\Delta_\lambda Y - \MY^{\gl, \xi} \|^2_{\bS^2} \le C ( \|A_1\|^2_{L^2} + \|A_2\|^2_{L^2} + \|A_3\|^2_{L^2}+ \|A_4\|^2_{L^2})
\ee
with
\begin{align*}
A_1 &:= \Delta_\lam \Phi - \tE [ \partial_{\mu_t} \Phi (B^{\gam_t}, \ML_{B^{\eta_t}}, \tB^{\teta_t}) \txi ] , \quad
A_2 := \int_t^T |R_1(r)| dr,  \\
A_3 &:= \int_t^T  \left| ( \Delta_\lam f - \tE [ \partial_{\mu_t} f(\Theta^\gl_r, \ML_{\Theta_r^{\eta_t}}, \tB^{\teta_t} ) \txi ] )\right| dr, \quad \text{and}\\
A_4 &:= \int_t^T \left|  \tE [ \partial_{\nu} f(\Theta^\gl_r, \ML_{\Theta_r^{\eta_t}}, \tY^{\teta_t} ) (\Delta_\lam \tY^{\teta} - \partial_{\ome_t} \tY^{\teta_t, \ML_{\eta_t}} \txi  - \tilde{\MY}^{\teta_t, \txi})(r) ] \right| dr.
\end{align*}

For $A_1, $ according to the Lipschitz continuity of $\partial_{\mu_t} \Phi,$ we have
\be\label{bd-A1}
\begin{split}
\E|A_1|^2 =&\ \E \Big|  \int_0^1 \tE [\partial_{\mu_t} \Phi(B^\gam, \ML_{B^{\eta}+\theta(B^{\eta^{\lam \xi}} - B^{\eta} )} ,\tB^{\teta}+\theta(\tB^{\teta^{\lam \txi}} - \tB^{\teta} ) )\txi\\
 &\   - \partial_{\mu_t} \Phi (B^{\gam_t}, \ML_{B^{\eta_t}}, \tB^{\teta_t}) \txi]d\theta \Big|^2\\
\le&\  C \left( [\bE\|\bB^{\bareta^{\lam \bxi}}- \bB^{\bareta}\|^2]^{\frac12}\|\xi\|_{L^2} +\tE[\|\tB^{\teta^{\lam \txi}} - \tB^{\teta}\| |\txi|] \right)^2
 \le  C \lam^2   \|\xi \|_{L^2}^4  ,
\end{split}
\ee
for a constant $C$ independent of $\gam$ and $\eta.$ Term  $A_2$ is estimated as follows:
\be\label{A2}
\begin{split}
 | A_2  |^2 \le&\  C   \Big| \int_t^T (\alpha(r)- \partial_y f(\Theta^\gl_r, \ML_{\Theta_r^{\eta_t}})) \Delta_\lam Y dr \Big|^2\\
 & +  C\Big|\int_t^T \left(\beta(r) - \partial_z f(\Theta^\gl_r, \ML_{\Theta_r^{\eta_t}}) \right) \Delta_\lam Z dr \Big|^2\\
  &  + C \Big| \int_t^T \tE [( \tilde{g}(r)- \partial_{\nu} f(\Theta^\gl_r, \ML_{\Theta_r^{\eta_t}}, \tY^{\teta_t} )) \Delta_\lam \tY^{\teta}  ] dr \Big|^2  .
  \end{split}
\ee
For the first two terms on the right hand side of the above inequality, by the Lipschitz continuity of $\partial_{(y,z)} f $ and inequality \eqref{bd-dy}, we obtain
\begin{align*}\nonumber
&\E\, \Big| \int_t^T \left(\alpha(r)- \partial_y f(\Theta^\gl_r, \ML_{\Theta_r^{\eta_t}}) \right) \Delta_\lam Y dr \Big|^2\\
&\quad +  \Big|\int_t^T \left(\beta(r) - \partial_z f(\Theta^\gl_r, \ML_{\Theta_r^{\eta_t}}) \right) \Delta_\lam Z dr \Big|^2
 \le C\lam^2   \|\xi \|_{L^2}^4.
\end{align*}

For the third term, we claim that
\be
\E \Big| \int_t^T  \tE\left[( \tilde{g}(r)- \partial_{\nu} f(\Theta^\gl_r, \ML_{\Theta_r^{\eta_t}}, \tY^{\teta_t} )) \Delta_\lam \tY^{\teta} \right]dr \Big|^2 \le  C \lam^2   \|\xi \|_{L^2}^4  ,
\ee
with $C$ depending only on $\|\eta_t\|_{\bS^2},$
and therefore we have
\be\label{bd-A2}
\E|A_2|^2 \le C \lam^2   \|\xi \|_{L^2}^4,
\ee
in view of \eqref{A2} and above estimates.
Indeed, by the H\"older inequality and estimate \eqref{bd-yeta}, we have
\begin{align*}
&\E\Big| \int_t^T \tE[ (\tilde{g}(r)- \partial_\nu f) \Delta_\lam \tY^{\teta}] dr  \Big|^2 \le  \E[\int_t^T \tE|\tilde{g} - \partial_\nu f|^2 dr] \int_t^T \tE|\Delta_\lam \tY^{\teta}|^2dr\\
& \quad \le  C \|\xi \|_{L^2}^2 \E \int_t^T \tE\Big| \int_0^1 (\partial_{\nu} f ( \Theta^{\gam,\eta}, \ML_{B^{\eta^{\lambda \xi}}}, \ML_{Y^{ \eta}+ \theta(Y^{ \eta^{\lambda \xi}}-Y^{ \eta})} ,   \tY^{ \teta}+ \theta(\tY^{ \teta^{\lambda \txi}}-\tY^{ \teta}) ) \\
 &\quad  \quad  -\partial_{\nu} f(\Theta^\gl_r, \ML_{\Theta_r^{\eta_t}}, \tY^{\teta_t} ) )d \theta \Big|^2 dr\\
&\quad \le C \|\xi \|_{L^2}^2 (\|\tY^{ \teta^{\lambda \txi}}-\tY^{ \teta}\|_{\bS^2} + \|B^{\eta^{\lambda \xi}} - B^{\eta}\|_{\bS^2})^2 \le  C \lam^2 \|\xi \|_{L^2}^4  .
\end{align*}

For $A_3, $ from Lipschitz continuity of $\partial_{\mu_t}f$ in $(\mu,\nu,\tome)$, we have
\be\label{bd-A3}
\begin{split}
\E|A_3|^2 =&\ \E\big|\int_t^T \int_0^1 \tE [ \partial_{\mu_t}f(\Theta^{\gl}, \ML_{B^{\eta+\theta(\eta^{\lam \xi}-\eta)}} , \ML_{Y^\eta},\tB^{\teta+\theta(\teta^{\lam \txi}-\teta)} )\\
& \    - \partial_{\mu_t} f(\Theta^\gl_r, \ML_{\Theta_r^{\eta_t}}, \tB^{\teta_t} ) ]\txi d\theta   dr \big|^2\\
\le& \ C   \left| \tE[(\|\eta^{\lam \xi}_t- \eta_t \|_{\bS^2} + \| \teta^{\lam \txi} - \teta \|) |\txi|]\right|^2\\
\le& \ C \left|\tE[\lam \|\xi\|_{L^2 } \txi + \lam |\txi|^2 ] \right|^2
\le  C \lam^2 \| \xi \|_{L^2}^4.
\end{split}
\ee

We now estimate $A_4$.  Since
\be\label{A4}
\Delta_\lam \tY^{\teta} - \partial_{\ome_t} \tY^{\teta_t, \ML_{\eta_t}} \txi  - \tilde{\MY}^{\teta_t, \txi}=\  A_{41} + A_{42}
\ee
with 
\be 
\begin{split}
 A_{41}&:= \  [\frac{1}{\lam}(\tY^{\teta^{\lam \txi}, \ML_{\teta^{\lam \txi}}} - \tY^{\teta, \ML_{\teta^{\lam \txi}}}) - \partial_{\ome_t} \tY^{\teta_t, \ML_{\eta_t}} \txi ], \quad\quad \text{and} \\
 A_{42}&:=[\frac{1}{\lam}(\tY^{\teta, \ML_{\teta^{\lam \txi}}} - \tY^{\teta, \ML_{\teta}}) - \tilde{\MY}^{\teta_t, \txi} ], 
\end{split}
\ee
then,  from boundedness of $\partial_\nu f,$ we have
\be\label{bd-A4}
\begin{split}
\E|A_4|^2 &=  \E\Big| \int_t^T   \tE \Big[ \partial_{\nu} f(\Theta^\gl_r, \ML_{\Theta_r^{\eta_t}}, \tY^{\teta_t} ) (A_{41}(r)+ A_{42}(r))  \Big] dr \Big|^2\\
&\le  C \Big(   \Big| \int_t^T \tE[  A_{41}] dr \Big|^2 +    \Big| \int_t^T \tE [ A_{42}] dr \Big|^2 \Big).
\end{split}
\ee
From Proposition \ref{s-bdd1}, we have
\be\label{bd-A41}
\begin{split}
  \left| \int_t^T \tE[  A_{41}] dr \right|^2 &\le C \int_t^T \left[ \tE \int_0^1 |( \partial_{\ome_t}\tY^{\teta^{\lam \theta \txi},\ML_{\eta^{\lam \xi}}} - \partial_{\ome_t}\tY^{\teta, \ML_{\eta}} )\txi| d\theta  \right]^2 dr \\
&\le C \|\xi \|_{L^2}^2  \int_t^T   \int_0^1  \tE |\partial_{\ome_t}\tY^{\teta^{\lam \theta \txi},\ML_{\eta^{\lam \xi}}} - \partial_{\ome_t}\tY^{\teta, \ML_{\eta}} |^2 d\theta dr\\
&\le C \|\xi \|_{L^2}^2  \int_t^T   \int_0^1 \tE (\|\teta_t^{\lam \theta \txi}- \teta_t \|)^2 d\theta dr
\le C \lam^2 \|\xi \|_{L^2}^4
\end{split}
\ee
for a constant $C$ only depending on $\|\eta_t\|_{\bS^2}.$
Since
\bea\label{bd-A42}
\left| \int_t^T \tE [ A_{42}] dr \right|^2
\le  \int_t^T \tE|A_{42}|^2 dr
\le C \sup_{\gam_t} \int_t^T \tE   |\Delta_\lambda Y - \MY^{\gl, \xi}  |^2  dr,
\eea
for a constant $C$ independent of $(\gam,\eta),$ we have
\be\label{bdd-A4}
\E|A_4|^2 \le C\left(\lam^2 \|\xi \|_{L^2}^4   + \sup_{\gam_t} \int_t^T \tE   |\Delta_\lambda Y - \MY^{\gl, \xi}  |^2  dr \right).
\ee

\ \ \ \ \ \ Finally, in view of inequalities \eqref{bd-A1}, \eqref{bd-A2}, \eqref{bd-A3}, \eqref{bdd-A4} and \eqref{bd-dely}, we have
 \beaa
\|\Delta_\lambda Y - \MY^{\gl, \xi} \|^2_{\bS^2} \le C (\lam^2 \|\xi \|_{L^2}^4 +   \sup_{\gam_t} \int_t^T     \|\Delta_\lambda Y - \MY^{\gl, \xi}  \|_{\bS^2}^2  dr),
\eeaa
where $C$ only depends on $\|\eta_t\|_{\bS^2}$.
Then,  using Gronwall's inequality, we have
\be
\|\Delta_\lambda Y - \MY^{\gl, \xi} \|^2_{\bS^2} \le C \lam^2 \|\xi \|_{L^2}^4 \rightarrow 0, \ \ \text{ as } \lam \rightarrow 0.
\ee

\end{proof}

%
%
%
%

\ \ \ \ \ \ To show the strongly vertical differentiability of $Y^{\gam_t,\eta_t}$ in $\eta_t,$ in view of Definition \ref{svdm}, for any $\tau \le t$ and $\xi \in L^2(\MF_\tau,\R^d),$ consider $\eta_t^{\tau,\lambda \xi}:= \eta_t+ \lambda \xi 1_{[\tau,T]}$. Similar as the vertical differentiable case, we firstly need to show the following limit exits in $\bS^2([t,T]),$
\be
\partial_{\eta_\tau} Y^{\gamma_t, \eta_t,\xi}: = \lim_{\lambda \rightarrow 0} \frac{1}{\lam}(Y^{\gamma_t, \eta_t^{\tau,\lambda \xi}}- Y^{\gam_t, \eta_t}).
\ee
Indeed, $\partial_{\eta_\tau} Y^{\gamma_t, \eta_t, \xi} $ is the unique solution of the following BSDE: for $s\in[t,T],$
\be\label{s-e5'}
\begin{split}
 \partial_{\eta_\tau}Y^{\gl ,\xi}(s)  &= \tE [ \partial_{\mu_\tau} \Phi (B^{\gam_t}, \ML_{B^{\eta_t}}, \tB^{\teta_t}) \txi ] + \int_s^T \tE [ \partial_{\mu_\tau} f(\Theta^\gl_r, \ML_{\Theta_r^{\eta_t}}, \tB^{\teta_t} ) \txi ] dr \\
&\ \ \   + \int_s^T  \partial_y f(\Theta^\gl_r, \ML_{\Theta_r^{\eta_t}}) \partial_{\eta_\tau}Y^{\gl ,\xi} (r) dr\\
&\ \ \  +  \int_s^T   \tE [ \partial_{\nu} f(\Theta^\gl_r, \ML_{\Theta_r^{\eta_t}}, \tY^{\teta_t} ) ( \partial_{\ome_\tau} \tY^{\teta_t, \ML_{\eta_t}} \txi  + \partial_{\eta_\tau}\tilde{Y}^{\teta_t, \txi})(r) ] dr\\
&\ \ \  + \int_s^T  \partial_z f(\Theta^\gl_r, \ML_{\Theta_r^{\eta_t}}) \partial_{\eta_\tau}Z^{\gl ,\xi}  (r) dr\\
&\ \ \   - \int_s^T  \partial_{\eta_\tau}Z^{\gl ,\xi}(r) d B(r),
\end{split}
\ee
where $\partial_{\eta_\tau} {Y}^{\eta_t, \xi} $ solves the following  mean-field BSDE
\be\label{s-e6}
\begin{split}
\partial_{\eta_\tau}Y^{\eta_t, \xi}(s) &= \tE [ \partial_{\mu_\tau} \Phi (B^{\eta_t}, \ML_{B^{\eta_t}}, \tB^{\teta_t}) \txi ] + \int_s^T \tE [ \partial_{\mu_\tau} f(\Theta^{\eta_t}_r, \ML_{\Theta_r^{\eta_t}}, \tB^{\teta_t} ) \txi ] dr \\
&\ \ \  + \int_s^T  \partial_y f(\Theta^{\eta_t}_r, \ML_{\Theta_r^{\eta_t}}) \partial_{\eta_\tau}Y^{{\eta_t} ,\xi} (r) dr\\
&\ \ \  +  \int_s^T   \tE [ \partial_{\nu} f(\Theta^{\eta_t}_r, \ML_{\Theta_r^{\eta_t}}, \tY^{\teta_t} ) ( \partial_{\ome_\tau} \tY^{\teta_t, \ML_{\eta_t}} \txi  + \partial_{\eta_\tau}\tilde{Y}^{\teta_t, \txi})(r) ] dr\\
&\ \ \  + \int_s^T  \partial_z f(\Theta^{\eta_t}_r, \ML_{\Theta_r^{\eta_t}}) \partial_{\eta_\tau}Z^{{\eta_t} ,\xi}  (r) dr - \int_s^T  \partial_{\eta_\tau}Z^{{\eta_t} ,\xi}(r) d B(r).
\end{split}
\ee
According to Assumption \textbf{(H1)}, we see that BSDEs \eqref{s-e6} and \eqref{s-e5'} are well-posed. Moreover, following a
similar argument as in Lemma \ref{gat1}, for the G\^ateaux strong vertical differentiability, we have

\begin{lem}\label{s-gat1}

$\partial_{\eta_\tau}Y^{\gl, \cdot}$ is a bounded linear operator from $L^2({\MF_\tau},\R^d)$ to $\bS^2([t,T])$. Moreover,  $\partial_{\eta_\tau}Y^{\gl, \xi}$ is the G\^ateaux strong vertical derivative of $Y^{\gamma_t, \eta_t}$ at $(\tau,t,\eta) $:
\be
\partial_{\eta_\tau} Y^{\gl, \xi} =   \lim_{\lambda \rightarrow 0} \frac{1}{\lam}(Y^{\gamma_t, \eta_t^{\tau,\lambda \xi}}- Y^{\gam_t, \eta_t}), \quad  \text{strongly in   $\bS^2([t,T])$}.
\ee
In particular, $\partial_{\eta_\tau}Y^{\gl, \cdot}(s)$ is the G\^ateaux derivative of $Y^{\gamma_t, \eta_t}(s)$ at $(\tau,t,\eta) $ in the sense of \eqref{sgat}.

\end{lem}

\ \ \ \ \ \ To give an explicit representation of the vertical derivative $ Y^{\gam_t, \ML_{\eta_t}}(\cdot) $ with respect to $\ML_{\eta_t}$ in view of \eqref{rep-dmu}, we need to find out a measurable random field $U^{\gam_t, \ML_{\eta_t}}(\cdot): \D \mapsto \bS^2([t,T],\R^d)$, such that for any $  s\ge t$ and $\xi \in L^2(\MF_t,\R^d),$
\be\label{des-dy}
\MY^{\gl, \xi}(s)  = \bE[U^{\gam_t, \ML_{\eta_t}}(\bareta_t)(s)\bxi],
\ee
where $(\bareta, \bxi)$ is an independent copy of $(\eta,\xi).$
If \eqref{des-dy} holds and moreover we show that $Y^{\gam_t,  {\eta_t}}$ is Fr\'echet differentiable with respect to $\eta_t$ in the sense of \eqref{fre} and Remark \ref{d-ext}, we have that
\be
\partial_{\mu_t}Y^{\gam_t, \ML_{\eta_t}}(x_t) := U^{\gam_t, \ML_{\eta_t}}(x_t),\ \  \forall \ x\in \D,
\ee
is the vertical derivative of $Y^{\gam_t, \ML_{\eta_t}}$ at $\ML_{\eta_t}$. Here and in the following,
we write $\partial_\mu$ instead of $\partial_{\ML_{\eta}}$.
In view of \eqref{e5'} and \eqref{des-dy}, we formally deduce that $ (U^{\gam_t, \ML_{\eta_t}}(x_t),V^{\gam_t, \ML_{\eta_t}}(x_t))$ solves the following BSDE: for any $s\in[t,T],$
\be\label{e-u}
\begin{split}
U^{\gl  ,x_t}(s) &= \tE [ \partial_{\mu_t} \Phi (B^{\gam_t}, \ML_{B^{\eta_t}}, \tB^{x_t} ) ]   + \int_s^T \tE [ \partial_{\mu_t} f(\Theta^\gl_r, \ML_{\Theta_r^{\eta_t}}, \tB^{x_t} )   ] dr \\
&\ \ \  + \int_s^T  \partial_y f(\Theta^\gl_r, \ML_{\Theta_r^{\eta_t}}) U^{\gl   ,x_t}(r) dr\\
 &\ \ \  +  \int_s^T   \tE [ \partial_{\nu} f(\Theta^\gl_r, \ML_{\Theta_r^{\eta_t}}, \tY^{x_t, \ML_{\eta_t}} )   \partial_{\ome_t} \tY^{x_t, \ML_{\eta_t}} (r) ]  dr\\
 &\ \ \   +  \int_s^T   \tE [ \partial_{\nu} f (\Theta^\gl_r, \ML_{\Theta_r^{\eta_t}}, \tY^{ \teta_t } )  \tilde{U}^{\teta_t ,x_t} (r) ] dr \\
 &\ \ \  + \int_s^T  \partial_z f(\Theta^\gl_r, \ML_{\Theta_r^{\eta_t}}) V^{\gl ,x_t} (r) dr - \int_s^T  V^{\gl ,x_t}(r) d B(r),
\end{split}
\ee
where $ {U}^{\eta_t ,x_t}$ solves the associated mean-field BSDE:
\be\label{e-ueta}
\begin{split}
U^{\eta_t  ,x_t}(s) &= \tE [  \partial_{\mu_t} \Phi (B^{\eta_t}, \ML_{B^{\eta_t}}, \tB^{x_t} ) ]   + \int_s^T \tE [ \partial_{\mu_t} f(\Theta^\eta_r, \ML_{\Theta_r^{\eta_t}}, \tB^{x_t} )   ] dr \\
&\ \ \ + \int_s^T  \partial_y f(\Theta^\eta_r, \ML_{\Theta_r^{\eta_t}}) U^{\eta_t   ,x_t} (r) dr\\
&\ \ \  +  \int_s^T   \tE [ \partial_{\nu} f(\Theta^{\eta_t}_r, \ML_{\Theta_r^{\eta_t}}, \tY^{x_t, \ML_{\eta_t}} )  \partial_{\ome_t} \tY^{x_t, \ML_{\eta_t}}(r) ] dr \\
 & \ \ \ +  \int_s^T   \tE [ \partial_{\nu} f (\Theta^\gl_r, \ML_{\Theta_r^{\eta_t}}, \tY^{ \teta_t } )  \tilde{U}^{\teta_t ,x_t})(r) ] dr\\
 & \ \ \ + \int_s^T  \partial_z f(\Theta^{\eta_t}_r, \ML_{\Theta_r^{\eta_t}}) V^{\eta_t  ,x_t}  (r) dr  - \int_s^T  V^{\eta_t  ,x_t}(r) d B(r).
\end{split}
\ee
\ \ \ \ \ \ According to Lemma \ref{linear}, we see that mean-field BSDE \eqref{e-ueta} is well posed with $(U^{\eta_t, x_t} , V^{\eta_t, x_t} ) \in \bS^2([t,T],\R^d) \times \bH^2([t,T],\R^{d\times d}).$ Then BSDE \eqref{e-u} also has a unique solution $(U^{\gl, x_t} , V^{\gl, x_t} ) \in \bS^2([t,T],\R^d) \times \bH^2([t,T],\R^{d\times d})$. Moreover, according to the uniqueness of solutions for BSDEs \eqref{e-ueta}, we see 
$
U^{\eta_t, x_t} = U^{\gam_t, {\eta_t},x_t} |_{\gam=\eta} .
$ Concerning the regularity of $U^{\gl, x_t} $ and $U^{\eta_t, x_t}$ with respect to $(\gam, \eta, x)$,  we have

\begin{lem}\label{reg-U}
For any $x,x', \gam, \gam' \in \D,$ and $\eta, \eta' \in \MM_2^D,$ we have 
\begin{align}\label{d-u}
&\|U^{\eta_t, x_t} - U^{\eta'_t, x'_t}   \|_{\bS^2} \le C (  \|\eta_t-\eta'_t \|_{\bS^2}+ \|x_t-x'_t \|  ),  \\ \label{d-ueta}
&\|U^{\gl ,x_t} - U^{\gam'_t,\eta'_t ,x'_t}   \|_{\bS^2} \le C (\|\gam_t-\gam'_t\|+ W_2(\ML_{\eta_t}, \ML_{\eta'_t} )       + \|x_t-x'_t \| ),
\end{align}
with $C$ only depending on $ \|\eta_t\|_{\bS^2}+\|\eta'_t\|_{\bS^2}$.
\end{lem}

\begin{proof}

In the following we omit the subscript $t$ and write $(U,V,Y,Z):=(U(r),V(r),Y(r),Z(r)) $. Moreover, we only show the proof for \eqref{d-u} since \eqref{d-ueta} follows from \eqref{d-u} in a similar way. Set
\beaa
&&(\Delta U, \Delta V ):= (U^{\eta,x  } -U^{\eta', x'} , V^{\eta, x }-V^{\eta',x' }),\\
&&\Delta \partial_{\mu} \Phi := \partial_{\mu_t} \Phi(B^\eta, \ML_{B^\eta}, B^x)-\partial_{\mu_t} \Phi(B^{\eta'}, \ML_{B^{\eta'}}, B^{x'}), \\
&& \Delta \partial_{\mu} f:= \partial_{\mu_t} f(\Theta^\eta_r, \ML_{\Theta_r^{\eta}}, \tB^{x_t}) - \partial_{\mu_t} f(\Theta^{\eta'}_r, \ML_{\Theta_r^{\eta'}}, \tB^{x'_t}),\\
&& \Delta \partial_{\nu} f^{(1)}:= \partial_{\nu} f(\Theta^\eta_r, \ML_{\Theta_r^{\eta}}, \tY^{x_t, \ML_{\eta_t}})  - \partial_{\nu} f(\Theta^{\eta'}_r, \ML_{\Theta_r^{\eta'}}, \tY^{x'_t, \ML_{\eta'_t}} ),\\
&& \Delta \partial_{\nu} f^{(2)}:= \partial_{\nu} f(\Theta^\eta_r, \ML_{\Theta_r^{\eta}}, \tY^{\teta_t}) - \partial_{\nu} f(\Theta^{\eta'}_r, \ML_{\Theta_r^{\eta'}}, \tY^{\teta'_t} ),\\
&&\Delta \partial_{(y,z )}f  := \partial_{(y,z )} f(\Theta^{\eta}_r, \ML_{\Theta_r^{\eta} }) - \partial_{(y,z )} f(\Theta^{\eta'}_r, \ML_{\Theta_r^{\eta'}} ), \\
&&\Delta \partial_{\ome} \tY:= \partial_{\ome_t} \tY^{x_t, \ML_{\eta_t}} - \partial_{\ome_t} \tY^{x'_t, \ML_{\eta'_t}}.
\eeaa
Then,  $(\Delta U , \Delta V )$ is the unique solution of BSDE
\be
\begin{split}
\Delta U (s)  &= \Delta \partial_\mu \Phi + \int_s^T \tE[\Delta \partial_\mu f] dr + \int_s^T \tE[\partial_\nu f(\Theta^{\eta}_r, \ML_{\Theta_r^{\eta} } , \tY^{\teta} ) \Delta \tU ] dr\\
&\ \ \  + \int_s^T (\partial_y f(\Theta^{\eta}_r, \ML_{\Theta_r^{\eta} }) \Delta U + \partial_z f(\Theta^{\eta}_r, \ML_{\Theta_r^{\eta} } ) \Delta V )dr- \int_s^T \Delta V dB(r)\\
&\ \ \   + \int_s^T  \tE[ (\Delta \partial_{\nu}f^{(1)} ) \partial_{\ome} \tY^{ x',\teta'} +  (\Delta \partial_{\nu}f^{(2)} )  \tU^{\teta',x'}  ]  dr\\
&\ \ \   + \int_s^T   \tE[  \partial_{\nu}f(\Theta^{\eta}_r, \ML_{\Theta_r^{\eta} }, \tY^{x, \ML_{\eta}}) \Delta \partial_\ome \tY^{x, \ML_{\eta} } ] dr\\
&\ \ \   +  \int_s^T \left( (\Delta \partial_y f )  U^{\eta', x'}  +    (\Delta \partial_z f)  V^{\eta', x'}  \right) dr.
\end{split}
\ee
By Lipschitz continuity of $(\partial_{(\mu,\nu,y,z)} f, \partial_\mu \Phi ),$ and boudnedness of $\partial_{(y,z)} f,$ we see that
\bea \nonumber
&&\| \Delta \partial_\mu \Phi \|^2_{L^2} + \| \int_t^T \tE[\Delta \partial_\mu f] dr  \|^2_{L^2} +   \E[ \int_t^T \tE (|  \Delta \partial_\nu f^{(1)}  |^2+|  \Delta \partial_\nu f^{(2)}  |^2) dr  ] \\ \label{3132}
&&\ \ \ \ \   +  \E [ \int_t^T |  \Delta \partial_y f  |^2 dr] + \E [ \int_t^T   |  \Delta \partial_z f  |^2 dr ]  \\ \nonumber
&&\ \ \ \ \ \ \ \ \ \ \le  C (  \|\eta_t-\eta'_t \|^2_{\bS^2}+ \|x_t-x'_t \|^2  ).
\eea
Moreover, since $\partial_{\ome} \tY^{ x',\teta'} ,  \tU^{\eta',x'}, \tY^{x, \ML_{\eta} }, \tY^{\teta} \in \bS^2$, and $ V^{\eta', x'}  \in \bH^2, $ from the above estimate and the Cauchy-Schwartz inequality, we obtain  
\be \label{bd-dpf}
\begin{split}
&\| \int_s^T  \tE\left[ (\Delta \partial_{\nu}f^{(1)} ) \partial_{\ome} \tY^{ x',\teta'} +  (\Delta \partial_{\nu}f^{(2)} )  \tU^{\teta',x'}  \right]  dr  \|_{L^2}\\
 &\ \ \ \ \   +  \|  \int_s^T \left( (\Delta \partial_y f )  U^{\eta', x'}  +    (\Delta \partial_z f)  V^{\eta', x'}  \right) dr  \|_{L^2} \\
&\ \ \ \ \ \ \ \   \le  C (  \|\eta_t-\eta'_t \|_{\bS^2}+ \|x_t-x'_t \|  ).
\end{split}
\ee
According to estimates given in Lemma \ref{s-bdd1} and boundedness of $\partial_\nu f,$ we check that
\be\label{3131}
\|        \int_s^T  | \tE[  \partial_{\nu}f(\Theta^{\eta}_r, \ML_{\Theta_r^{\eta} }, \tY^{x, \ML_{\eta}}) \Delta \partial_\ome \tY^{x, \ML_{\eta} } ] | dr \|_{L^2}  \le C (  \|\eta_t-\eta'_t \|_{\bS^2}+ \|x_t-x'_t \|  ).
\ee
Then in view of Lemma \ref{linear}, inequalities \eqref{3132}, \eqref{bd-dpf} and \eqref{3131}, we obtain the desired inequality \eqref{d-u}.

\end{proof}

\begin{rem}

Similar to Lemma \ref{unibd}, according to estimate \eqref{d-ueta},
$
U^{\gam_t, \ML_{\eta_t},x_t}:= U^{\gl, x_t }
$
is well-defined.

\end{rem}

\ \ \ \ \ \ Concerning the SVD $\partial_{\mu_\tau} Y^{\gl,\cdot}$ of $Y^{\gam_t,\ML_{\eta_t}} $ at $(\tau,t,\ML_{\eta})$, $\tau \le t,$ in view of Definition \ref{svdm} and BSDE \eqref{s-e5'}, we deduce that it is the unique solution of the following BSDE: for any $x\in \D,$ $s\in[t,T],$
\be\label{s-e-u}
\begin{split}
 \partial_{\mu_\tau} Y^{\gl  ,x_t}(s) &= \tE [ \partial_{\mu_\tau} \Phi (B^{\gam_t}, \ML_{B^{\eta_t}}, \tB^{x_t} ) ]   + \int_s^T \tE [ \partial_{\mu_\tau} f(\Theta^\gl_r, \ML_{\Theta_r^{\eta_t}}, \tB^{x_t} )   ] dr \\
&\ \ \ + \int_s^T  \partial_y f(\Theta^\gl_r, \ML_{\Theta_r^{\eta_t}}) \partial_{\mu_\tau} Y^{\gl   ,x_t}(r) dr\\
&\ \ \  +  \int_s^T   \tE [ \partial_{\nu} f (\Theta^\gl_r, \ML_{\Theta_r^{\eta_t}}, \tY^{ \teta_t } )  \partial_{\mu_\tau} \tilde{Y}^{\teta_t ,x_t} (r) ] dr\\
 &\ \ \ +  \int_s^T   \tE [ \partial_{\nu} f(\Theta^\gl_r, \ML_{\Theta_r^{\eta_t}}, \tY^{x_t, \ML_{\eta_t}} )   \partial_{\ome_\tau} \tY^{x_t, \ML_{\eta_t}} (r) ]  dr\\
 & \ \ \  + \int_s^T  \partial_z f(\Theta^\gl_r, \ML_{\Theta_r^{\eta_t}}) \partial_{\mu_\tau} Z^{\gl ,x_t} (r) dr\\
& \ \ \   - \int_s^T  \partial_{\mu_\tau} Z^{\gl ,x_t}(r) d B(r), 
\end{split}
\ee
where $\partial_{\mu_\tau} Y^{\eta_t ,x_t}$ sloves the mean-field BSDE below
\be\label{s-e-ueta}
\begin{split}
\partial_{\mu_\tau} Y^{\eta_t  ,x_t}(s)& = \tE [  \partial_{\mu_\tau} \Phi (B^{\eta_t}, \ML_{B^{\eta_t}}, \tB^{x_t} ) ]   + \int_s^T \tE [ \partial_{\mu_\tau} f(\Theta^\eta_r, \ML_{\Theta_r^{\eta_t}}, \tB^{x_t} )   ] dr \\
 &\ \ \ + \int_s^T  \partial_y f(\Theta^\eta_r, \ML_{\Theta_r^{\eta_t}}) \partial_{\mu_\tau} Y^{\eta_t   ,x_t} (r) dr\\
 &\ \ \  +  \int_s^T   \tE [ \partial_{\nu} f(\Theta^{\eta_t}_r, \ML_{\Theta_r^{\eta_t}}, \tY^{x_t, \ML_{\eta_t}} )  \partial_{\ome_\tau} \tY^{x_t, \ML_{\eta_t}}(r) ] dr\\
 &\ \ \  +  \int_s^T   \tE [ \partial_{\nu} f (\Theta^\gl_r, \ML_{\Theta_r^{\eta_t}}, \tY^{ \teta_t } )  \partial_{\mu_\tau} \tilde{Y}^{\teta_t ,x_t})(r) ] dr\\
&\ \ \ + \int_s^T  \partial_z f(\Theta^{\eta_t}_r, \ML_{\Theta_r^{\eta_t}}) \partial_{\mu_\tau} Z^{\eta_t  ,x_t}  (r) dr
- \int_s^T  \partial_{\mu_\tau} Z^{\eta_t  ,x_t}(r) d B(r).
\end{split}
\ee
Thanks to Lemma \ref{linear} again, mean-field BSDE \eqref{s-e-ueta} has a unique solution $(\partial_{\mu_\tau} Y^{\eta_t ,x_t}, \partial_{\mu_\tau} Z^{\eta_t ,x_t}) \in \bS^2 \times \bH^2.$ Then the well-posedness of equation \eqref{s-e-u} follows similarly. Moreover we have that if $\tau=t,$
\be
\partial_{\mu_t } Y^{\gl ,x_t}= U^{\gl,x_t}, \ \ \ \partial_{\mu_t } Y^{\eta_t ,x_t}=U^{\eta_t ,x_t} ,  
\ee
 and $\partial_{\mu_\tau } Y^{\eta_t ,x_t} = \partial_{\mu_\tau } Y^{\gam_t, \ML_{\eta_t} ,x_t}|_{\gam=\eta}.$ Thus the following lemma follows similarly as Lemma \ref{reg-U}.

\begin{lem}\label{s-reg-U}
For any $x,x', \gam, \gam' \in \D,$ and $\eta, \eta' \in \MM_2^D,$ we have
\begin{align}\label{s-d-u}
&\|\partial_{\mu_\tau } Y^{\eta_t ,x_t} - \partial_{\mu_\tau } Y^{\eta'_t ,x'_t}   \|_{\bS^2} \le C (  \|\eta_t-\eta'_t \|_{\bS^2}+ \|x_t-x'_t \|  ), \\ \label{s-d-ueta}
&\|\partial_{\mu_\tau } Y^{\gl ,x_t} - \partial_{\mu_\tau } Y^{\gam'_t, \eta'_t ,x'_t}   \|_{\bS^2} \le C (\|\gam_t-\gam'_t\|+ W_2(\ML_{\eta_t}, \ML_{\eta'_t} )       + \|x_t-x'_t \| ),
\end{align}

with $C$ only depending on $\|\eta_t\|_{\bS^2}+\|\eta'_t\|_{\bS^2}.$
\end{lem}

\ \ \ \ \ \ Recall that $\MY^{\gl, \xi}$ is the solution of BSDE \eqref{e5'} and $U^{\gam_t,  {\eta_t},x_t }$ solves equation \eqref{e-u}.
The following lemma implies that $U^{\gam_t, \ML_{\eta_t},\cdot} :=  U^{\gam_t,  {\eta_t},\cdot } $ is the derivative of $Y^{\gam_t, \ML_{\eta_t}}$ with respect to $\ML_{\eta_t}.$

\begin{lem}\label{e-rep}
For any $\xi \in L^2(\MF_t,\R^d),$ we have
\be\label{des-dy2}
\MY^{\gl, \xi}(s)  = \bE[U^{\gam_t, \ML_{\eta_t},\bareta_t}(s)\bxi],
\ee
where $(\bareta, \bxi)$ is an independent copy of $(\eta, \xi)$.
\end{lem}

\begin{proof}

Substitute $ \bareta_t$ for $x_t$ in equation \eqref{e-ueta} and multiply the equation by $\bxi$. Then we take the expectation $
\bE$ on both sides of the relation, and obtain
\be\label{e-urep}
\begin{split}
\bE [U^{\eta_t, \bareta_t} (s)\bxi] &=\bE\left[ \tE [  \partial_{\mu_t} \Phi (B^{\eta_t}, \ML_{B^{\eta_t}}, \tB^{\bareta_t} ) ] \bxi\right]  + \int_s^T \bE\left[\tE [ \partial_{\mu_t} f(\Theta^\eta_r, \ML_{\Theta_r^{\eta_t}}, \tB^{\bareta_t} )   ]  \bxi\right]   dr \\
&\ \ \ + \int_s^T  \partial_y f(\Theta^\eta_r, \ML_{\Theta_r^{\eta_t}}) \bE [U^{\eta_t,  \bareta_t}   (r) \bxi] dr\\
&\ \ \ +  \int_s^T  \bE[  \tE [ \partial_{\nu} f(\Theta^{\eta_t}_r, \ML_{\Theta_r^{\eta_t}}, \tY^{\bareta_t, \ML_{\eta_t}} )  \partial_{\ome_t} \tY^{\bareta_t, \ML_{\eta_t}}(r) ] \bxi] dr \\
 &\ \ \ +  \int_s^T   \tE [ \partial_{\nu} f (\Theta^\gl_r, \ML_{\Theta_r^{\eta_t}}, \tY^{ \teta_t } ) \bE [\tilde{U}^{\teta_t,\bareta_t}(r)\bxi ] ]  dr\\
&\ \ \  + \int_s^T  \partial_z f(\Theta^{\eta_t}_r, \ML_{\Theta_r^{\eta_t}}) \bE [V^{\eta_t ,\bareta_t}  (r) \bxi]dr - \int_s^T \bE[  V^{\eta_t, \bareta_t}(r)\bxi] d B(r).
\end{split}
\ee
Since random vectors $(\bareta, \bxi)$, $(\tB, \tY, \tU)$, $(B, Y, Z, \eta)$ are mutually independent, and $(B^{\gam_t},$ $ Y^{\gam_t, \ML_{\eta_t}},$ $\partial_{x_t} Y^{\gam_t, \ML^{\eta_t}})$ is independent of $\MF_t$, we have
\be\nonumber
\begin{split}
&\bE\left[ \tE [  \partial_{\mu_t} \Phi (B^{\eta_t}, \ML_{B^{\eta_t}}, \tB^{\bareta_t} ) ] \bxi \right]= \tE [  \partial_{\mu_t} \Phi (B^{\eta_t}, \ML_{B^{\eta_t}}, \tB^{\teta_t} ) \txi ],\\
&\bE\left[\tE [ \partial_{\mu_t} f(\Theta^\eta_r, \ML_{\Theta_r^{\eta_t}}, \tB^{\bareta_t} )   ]  \bxi\right] =  \tE [ \partial_{\mu_t} f(\Theta^\eta_r, \ML_{\Theta_r^{\eta_t}}, \tB^{\teta_t} )  \txi ],  \quad \text{and}\\
&\bE\left[  \tE [ \partial_{\nu} f(\Theta^{\eta_t}_r, \ML_{\Theta_r^{\eta_t}}, \tY^{\bareta_t, \ML_{\eta_t}} )  \partial_{x_t} \tY^{\bareta_t, \ML_{\eta_t}}(r) ] \bxi\right] = \tE [ \partial_{\nu} f(\Theta^{\eta_t}_r, \ML_{\Theta_r^{\eta_t}}, \tY^{\teta_t} )  \partial_{x_t} \tY^{\teta_t, \ML_{\eta_t}}(r)  \txi].
\end{split}
\ee
Then identity \eqref{e-urep} is equivalent to
\be
\begin{split}
 \bE [U^{\eta_t ,\bareta_t}(s)\bxi] = &\tE [  \partial_{\mu_t} \Phi (B^{\eta_t}, \ML_{B^{\eta_t}}, \tB^{\teta_t} ) \txi ] + \int_s^T \tE [ \partial_{\mu_t} f(\Theta^\eta_r, \ML_{\Theta_r^{\eta_t}}, \tB^{\teta_t} )  \txi ]   dr \\    \nonumber
&\   + \int_s^T  \partial_y f(\Theta^\eta_r, \ML_{\Theta_r^{\eta_t}}) \bE [U^{\eta_t , \bareta_t} (r) \bxi] dr- \int_s^T \bE[  V^{\eta_t, \bareta_t}(r)\bxi] d B(r)\\
&\ +  \int_s^T  \tE [ \partial_{\nu} f(\Theta^{\eta_t}_r, \ML_{\Theta_r^{\eta_t}}, \tY^{\teta_t} )  \partial_{\ome_t} \tY^{\teta_t, \ML_{\eta_t}}(r)  \txi] dr \\ \nonumber
&\   +  \int_s^T   \tE [ \partial_{\nu} f (\Theta^\gl_r, \ML_{\Theta_r^{\eta_t}}, \tY^{ \teta_t } ) \bE [\tilde{U}^{\teta_t, \bareta_t} (r)\bxi ] ]  dr\\
\nonumber & \ + \int_s^T  \partial_z f(\Theta^{\eta_t}_r, \ML_{\Theta_r^{\eta_t}}) \bE [V^{\eta_t, \bareta_t}  (r) \bxi]dr,
\end{split}
\ee
and therefore $(\bE [U^{\eta_t, \bareta_t}(s)\bxi], \bE [V^{\eta_t, \bareta_t}(s)\bxi])$ satisfies BSDE \eqref{e6}. In view of uniqueness of solutions of BSDE \eqref{e6}, we see  
$
\MY^{\eta_t, \xi} = \bE [U^{\eta_t ,\bareta_t}(s)\bxi].
$
Then identity \eqref{des-dy2} follows in a similar way.

\end{proof}


\begin{thm}\label{dfre}

Suppose that $(\Phi, f)$ satisfies Assumption {\bf(H1)}. For any $(t,\gam, \eta)\in [0,T]\times \D \times  \MM^D_2  $, $Y^\gl$
is Fr\'echet differentiable with respect to $\eta_t$ in the sense of \eqref{fre} and Remark \ref{d-ext}. Moreover, the Fr\'echet derivative $D_{\eta_t}Y^{\gl}$ has the following representation: for any $\xi \in L^2(\MF_t,\R^d),$
\be\label{fre-rep}
D_{\eta_t}Y^{\gl}(s)  ( \xi) = \MY^{\gl, \xi}(s) = \bE[U^{\gam_t, \ML_{\eta_t}, \bareta_t}(s)\bxi],
\ee
where $\MY^{\gl, \xi} $ is the solution of BSDE \eqref{e5'} and $U^{\gam_t, \ML_{\eta_t}, x_t}, $ $x\in \D,$ is the solution of BSDE \eqref{e-u}. In particular, $U^{\gam_t, \ML_{\eta_t}, \cdot}$ is the vertical derivative of $Y^{\gam_t, \ML_{\eta_t}}$ at $\ML_{\eta_t}$ in the sense of \eqref{rep-dmu} and Remark \ref{d-ext}.

\end{thm}

\begin{proof}

According to inequality \eqref{y-linear} and argument therein, we see that $\MY^{\gl, \cdot} $ is a bounded linear operator from $L^2(\MF_t,\R^d) $ to $\bS^2([t,T]).$
Moreover, in view of Lemma \ref{gat1}, for any $\xi \in L^2(\MF_t,\R^d),$ $\MY^{\gl, \xi}$ is the  G\^ateaux derivative of $Y^{\gl}$ with respect to $\eta_t.$ To show $\MY^{\gl, \cdot}$ is the Fr\'echet derivative of $Y^{\gl}$, it suffices to prove that $\MY^{\gl, \cdot}$ is continuous in $\eta_t \in \bS^2([0,t])$ as a linear bounded operator from $  L^2(\MF_t)$ to $\bS^2([t,T])$.  Indeed, due to the representation \eqref{des-dy2} and estimate \eqref{d-ueta}, we have that for any $\eta, \eta' \in \MM_2^D,$
$$
\| \MY^{\gl, \xi}- \MY^{\gam_t, \eta'_t, \xi}  \|_{\bS^2}^2 = \E\|  \bE[U^{\gam_t, \ML_{\eta_t},\bareta_t}(s)\bxi]- \bE[U^{\gam_t, \ML_{\eta'_t},\bareta'_t}(s)\bxi]   \|^2
\le  C \| \bxi \|^2_{L^2}  \| \bareta_t - \bareta'_t \|^2_{\bS^2}.
$$
Thus we have the following estimate and complete our proof
$$
\|  \MY^{\gl, \cdot }- \MY^{\gam_t, \eta'_t, \cdot }  \|_{L(L^2(\MF_t), \bS^2)} \le C \| \bareta_t - \bareta'_t \|^2_{\bS^2}.
$$

\end{proof}

\ \ \ \ \ \ For the strong vertical differentiability of $Y^{\gl}$ at $(\tau, t, \ML_{\eta}),$ similar as the proof of Lemma \ref{e-rep}, we have that for any $\xi \in L^2(\MF_\tau,\R^d),$
\be\label{s-rep}
\partial_{\eta_\tau}Y^{\gl, \xi}(s)  = \bE[\partial_{\mu_\tau} Y^{\gam_t, \ML_{\eta_t}, \bareta_t}(s)\bxi].
\ee
Moreover, the following proposition implies that $\partial_{\mu_\tau} Y^{\gam_t, \ML_{\eta_t}, \cdot}$ is  the SVD of $Y^\gl$ at $(\tau,t,\ML_\eta),$ the proof of which follows from Lemma \ref{s-gat1}, Lemma \ref{s-reg-U} and identity \eqref{s-rep}.

\begin{prop}\label{s-dfre}
For any $(t,\gam, \eta)\in [0,T]\times \D \times  \MM^D_2  $ and $\tau\in [0,t]$, $Y^\gl$
is Fr\'echet differentiable with respect to $\eta_\tau$ in the sense of \eqref{sfre} and Remark \ref{ue-sdmu}. Moreover, the Fr\'echet derivative $D_{\eta_\tau}Y^{\gl} $ has the representation: for any $\xi \in L^2(\MF_\tau, \R^d)$
\be\label{s-fre-rep}
D_{\eta_\tau}Y^{\gl}(s) (\xi) = \partial_{\eta_\tau} Y^{\gl, \xi}(s) = \bE[\partial_{\mu_\tau}Y^{\gam_t, \ML_{\eta_t},\bareta_t}(s)\bxi],
\ee
where $\partial_{\eta_\tau} Y^{\gl, \xi}$ is the solution of BSDE \eqref{s-e5'} and $\partial_{\mu_\tau}Y^{\gam_t, \ML_{\eta_t},x_t},$ $x\in \D,$ is the solution of BSDE \eqref{s-e-u}.

\end{prop}

\subsection{Second-order differentiability}

\ \ \ \ \ In this section, results are written in $d=1$ case for simplicity of notations.
For the second order differentiability of $Y^\gl$, we assume that $(\Phi,f)$ satisfies assumption {\bf(H2)}.

\ \ \ \  \  According to Proposition \ref{s-bdd1}, $Y^\gl$ is strongly vertically differentiable at $( t,\gam),$ and the derivative $\partial_{\ome_\tau}Y^{\gl}$ at $(\tau,t,\gam)$ solves the linear BSDE \eqref{s-e4}. Similarly, in view of {\bf(H2)}, we see that $\partial_{\ome_\tau}Y^{\gl}$ is strongly vertically differentiable at $(\tau,t,\gam),$ and moreover, the derivative $\partial_{\ome_\tau}^2 Y^{\gl}$ is the unique solution of BSDE in the form of \eqref{e7}.
To apply Theorem \ref{itoformula} on $Y^{\gam_t,\eta_t}$, it remains to study
the differentiability of $\partial_{\mu_t} Y^{\gl}(x_t)=U^{\gl,x_t}$ with respect to $x_t \in \D $. Since $U^{\gl,x_t}$ is the unique solution of BSDE \eqref{e-u}, by formally taking vertical derivative at $(t,x), $ we obtain the following linear BSDE: for $s\in[t,T],$
\be\label{e-durx}
\begin{split}
&\partial_{\tome_t}U^{\gl ,x_t}(s)\\
&\ \  = \tE [\partial_{\tome_t} \partial_{\mu_t} \Phi (B^{\gam_t}, \ML_{B^{\eta_t}}, \tB^{x_t} ) ]
 + \int_s^T \tE [ \partial_{\tome_t} \partial_{\mu_t} f(\Theta^\gl_r, \ML_{\Theta_r^{\eta_t}}, \tB^{x_t} )   ] dr \\
&\ \ \ \ \ \ + \int_s^T  \partial_y f(\Theta^\gl_r, \ML_{\Theta_r^{\eta_t}}) \partial_{\tome_t}U^{\gl  ,x_t} (r) dr\\
&\ \ \ \ \ \ +  \int_s^T   \tE [\partial_{\ty} \partial_{\nu} f(\Theta^\gl_r, \ML_{\Theta_r^{\eta_t}}, \tY^{x_t, \ML_{\eta_t}} ) \partial_{\ome_t} \tY^{x_t, \ML_{\eta_t}}  (\partial_{\ome_t} \tY^{x_t, \ML_{\eta_t}})^T(r) ]  dr\\
&\ \ \ \ \ \ +  \int_s^T   \tE [ \partial_{\nu} f(\Theta^\gl_r, \ML_{\Theta_r^{\eta_t}}, \tY^{x_t, \ML_{\eta_t}} ) \partial^2_{\ome_t} \tY^{x_t, \ML_{\eta_t}} (r) ]  dr \\
&\ \ \ \ \ \  +  \int_s^T   \tE [ \partial_{\nu} f (\Theta^\gl_r, \ML_{\Theta_r^{\eta_t}}, \tY^{ \teta_t } ) \partial_{\tome_t}  \tilde{U}^{\teta_t ,x_t} (r) ] dr\\
& \ \ \ \ \ \  + \int_s^T  \partial_z f(\Theta^\gl_r, \ML_{\Theta_r^{\eta_t}}) \partial_{\tome_t}  V^{\gl  ,x_t}  (r) dr - \int_s^T \partial_{\tome_t}  V^{\gl  ,x_t}(r)\ d B(r),
\end{split}
\ee
where $\partial_{\tome_t}   {U}^{\eta_t ,x_t}$ solves a mean-field linear BSDE
\be\label{e-dux}
\begin{split}
&\partial_{\tome_t}U^{\eta_t ,x_t}(s)\\
 &\ \ = \tE [\partial_{\tome_t} \partial_{\mu_t} \Phi (B^{\eta_t}, \ML_{B^{\eta_t}}, \tB^{x_t} ) ]
 + \int_s^T \tE [\partial_{\tome_t} \partial_{\mu_t} f(\Theta^{\eta_t}_r, \ML_{\Theta_r^{\eta_t}}, \tB^{x_t} )   ] dr \\
&\ \ \ \ \ \ + \int_s^T  \partial_y f(\Theta^{\eta_t}_r, \ML_{\Theta_r^{\eta_t}}) \partial_{\tome_t}U^{{\eta_t} ,x_t} (r) dr  \\
&\ \ \ \ \ \ +  \int_s^T   \tE [\partial_{\ty} \partial_{\nu} f(\Theta^{\eta_t}_r, \ML_{\Theta_r^{\eta_t}}, \tY^{x_t, \ML_{\eta_t}} ) \partial_{\ome_t} \tY^{x_t, \ML_{\eta_t}}  (\partial_{\ome_t} \tY^{x_t, \ML_{\eta_t}})^T (r) ]  dr\\
&\ \ \ \ \ \ +  \int_s^T   \tE [ \partial_{\nu} f(\Theta^{\eta_t}_r, \ML_{\Theta_r^{\eta_t}}, \tY^{x_t, \ML_{\eta_t}} ) \partial^2_{\ome_t} \tY^{x_t, \ML_{\eta_t}} (r) ]  dr \\
&\ \ \ \ \ \  +  \int_s^T   \tE [ \partial_{\nu} f (\Theta^{\eta_t}_r, \ML_{\Theta_r^{\eta_t}}, \tY^{ \teta_t } ) \partial_{\tome_t}  \tilde{U}^{\teta_t ,x_t} (r) ] dr\\
&\ \ \ \ \ \ + \int_s^T  \partial_z f(\Theta^{\eta_t}_r, \ML_{\Theta_r^{\eta_t}}) \partial_{\tome_t}  V^{{\eta_t}  ,x_t}  (r) dr
- \int_s^T \partial_{\tome_t}  V^{{\eta_t}  ,x_t}(r)\ d B(r).
\end{split}
\ee

\begin{lem}\label{dxu}

There exist unique solutions $(\partial_{\tome_t}U^{\eta_t ,x_t}, \partial_{\tome_t}V^{\eta_t ,x_t}) \in \bS^2([t,T]) \times \bH^2([t,T])$ and $(\partial_{\tome_t}U^{\gl ,x_t}, \partial_{\tome_t}V^{\gl ,x_t}) \in \bS^p([t,T]) \times \bH^p([t,T])$ of equations \eqref{e-dux} and \eqref{e-durx}, respectively. Moreover,  $\partial_{\tome_t}U^{\gl ,x_t} $ is the vertical derivative of $U^{\gl ,x_t}$ at $(t,x)$, and for any $K>0$ and $ (\gam,\eta,x),\ (\gam',\eta',x')\in \D \times \MM_2^D \times \D \text{ such that } |||\ML_{\eta_t}|||, |||\ML_{\eta'_t}||| \le K,$
\be
\begin{split}
&\|(\partial_{\tome_t}U^{\gl ,x_t}   ,\partial_{\tome_t}V^{\gl ,x_t} )  \|_{\bS^p \times \bH^p} \le C_p ,\\
&\|(\partial_{\tome_t}U^{\gl ,x_t}-  \partial_{\tome_t}U^{ \gam'_t,\eta'_t ,x'_t} ,\partial_{\tome_t}V^{\gl ,x_t} -  \partial_{\tome_t}V^{ \gam'_t,\eta'_t ,x'_t}  )  \|_{\bS^p \times \bH^p} \\
&\ \ \ \le C_{K,p} ( \|x_t-x'_t\| + W_2(\ML_{\eta_t}, \ML_{\eta'_t}) + \|\gam_t-\gam'_t\| ),
\end{split}
\ee
with some constants $C_p$ and $C_{K,p}.$

\end{lem}

\begin{proof}

To show
the well-posedness of \eqref{e-dux}, according to Lemma \ref{linear}, it remains to check the following terms belong to $L^2(\MF_T)$,
\beaa
&& \Big| \tE [\partial_{\tome_t} \partial_{\mu_t} \Phi (B^{\eta_t}, \ML_{B^{\eta_t}}, \tB^{x_t} ) ]\Big|, \ \int_t^T \Big|\tE [\partial_{\tome_t} \partial_{\mu_t} f(\Theta^{\eta_t}_r, \ML_{\Theta_r^{\eta_t}}, \tB^{x_t} )   ]\Big| dr   ,\\
&&\int_t^T  \Big| \tE [ \partial_{\nu} f(\Theta^{\eta_t}_r, \ML_{\Theta_r^{\eta_t}}, \tY^{x_t, \ML_{\eta_t}} ) \partial^2_{\ome_t} \tY^{x_t, \ML_{\eta_t}} (r) ]\Big|  dr,\quad \text{and}\\
&&\int_t^T  \Big| \tE [\partial_{\ty} \partial_{\nu} f(\Theta^{\eta_t}_r, \ML_{\Theta_r^{\eta_t}}, \tY^{x_t, \ML_{\eta_t}} ) \partial_{\ome_t} \tY^{x_t, \ML_{\eta_t}}  (\partial_{\ome_t} \tY^{x_t, \ML_{\eta_t}} (r))^T ] \Big|,
 \\
\eeaa
which follows easily by the boundedness of $(\partial_{\tome_t} \partial_{\mu_t} \Phi , \partial_{\tome_t}\partial_{\mu_t} f, \partial_{\nu} f, \partial_{\ty} \partial_{\nu} f)$ and Proposition \ref{s-bdd1}.
Moreover, we have  
\be\nonumber 
\begin{split}
&\|\partial_{\tome_t}U^{\eta_t ,x_t}   \|_{\bS^2} + \|\partial_{\tome_t}V^{\eta_t ,x_t}   \|_{\bH^2} \le C_p , \\
&\|\partial_{\tome_t}U^{\eta_t ,x_t}-  \partial_{\tome_t}U^{\eta'_t ,x'_t} \|_{\bS^2} + \|\partial_{\tome_t}V^{\eta_t ,x_t} -  \partial_{\tome_t}V^{\eta'_t ,x'_t}    \|_{\bH^2}  \le C_{K,p} ( \|x_t-x'_t\| + \|\eta_t-\eta'_t\|_{\bS^2}).	
\end{split}
\ee
Concerning the well-posedness of \eqref{e-durx}, since $(\partial_{\tome_t}U^{\eta_t ,x_t}, \partial_{\tome_t}V^{\eta_t ,x_t}) \in \bS^2 \times \bH^2$, we have
$$
 \int_t^T  \Big| \tE [ \partial_{\nu} f (\Theta^\gl_r, \ML_{\Theta_r^{\eta_t}}, \tY^{ \teta_t } ) \partial_{\tome_t}  \tilde{U}^{\teta_t ,x_t} (r) ] \Big| dr \in L^p(\MF_T),
$$
and therefore, there exists a unique solution $(\partial_{\tome_t}U^{\gl ,x_t}, \partial_{\tome_t}V^{\gl ,x_t}) \in \bS^p \times \bH^p.$ In view of the boundedness of  $(\partial_{\tome_t}\partial_{\mu_t}\Phi, \partial_{\tome_t} \partial_{\mu_t}f, \partial_{\ty}\partial_\nu f, \partial_yf , \partial_z f)$ and standard estimate for BSDEs, we have our desired estimates.
\end{proof}

\ \ \ \ \ \ Recall that $\partial_{\mu_\tau} Y^{\gl,x_t}$ is the solution of BSDE \eqref{s-e-u}.
To prove the strong vertical differentiability of $\partial_{\mu_\tau} Y^{\gl,x_t}$ at $(\tau, t, x)$, we consider BSDE
\be\label{e-durx2}
\begin{split}
&\partial_{\tome_\tau}\partial_{\mu_\tau}Y^{\gl ,x_t}(s) \\
&\ \ = \tE [\partial_{\tome_\tau} \partial_{\mu_\tau} \Phi (B^{\gam_t}, \ML_{B^{\eta_t}}, \tB^{x_t} ) ]
 + \int_s^T \tE [ \partial_{\tome_\tau} \partial_{\mu_\tau} f(\Theta^\gl_r, \ML_{\Theta_r^{\eta_t}}, \tB^{x_t} )   ] dr \\
&\ \ \ \ \ \ + \int_s^T  \partial_y f(\Theta^\gl_r, \ML_{\Theta_r^{\eta_t}}) \partial_{\tome_\tau}\partial_{\mu_\tau}Y^{\gl ,x_t} (r) dr\\
&\ \ \ \ \ \ +  \int_s^T   \tE [\partial_{\ty} \partial_{\nu} f(\Theta^\gl_r, \ML_{\Theta_r^{\eta_t}}, \tY^{x_t, \ML_{\eta_t}} ) \partial_{\ome_\tau} \tY^{x_t, \ML_{\eta_t}}  (\partial_{\ome_\tau} \tY^{x_t, \ML_{\eta_t}} (r))^T ]  dr\\
&\ \ \ \ \ \ +  \int_s^T   \tE [ \partial_{\nu} f(\Theta^\gl_r, \ML_{\Theta_r^{\eta_t}}, \tY^{x_t, \ML_{\eta_t}} ) \partial^2_{\ome_\tau} \tY^{x_t, \ML_{\eta_t}} (r) ]  dr \\
&\ \ \ \ \ \  +  \int_s^T   \tE [ \partial_{\nu} f (\Theta^\gl_r, \ML_{\Theta_r^{\eta_t}}, \tY^{ \teta_t } ) \partial_{\tome_\tau}\partial_{\mu_\tau}\tY^{\teta_t ,x_t} (r) ] dr\\
&\ \ \ \ \ \ + \int_s^T  \partial_z f(\Theta^\gl_r, \ML_{\Theta_r^{\eta_t}}) \partial_{\tome_\tau}\partial_{\mu_\tau}Z^{\gl ,x_t}  (r) dr\\
&\ \ \ \ \ \ - \int_s^T \partial_{\tome_\tau}\partial_{\mu_\tau}Z^{\gl ,x_t}(r)\ d B(r),\ s\in [t,T],
\end{split}
\ee
where $\partial_{\tome_\tau}\partial_{\mu_\tau}Y^{\eta_t ,x_t}$ solves the following mean-field linear BSDE
\be\label{e-dux2}
\begin{split}
&\partial_{\tome_\tau}\partial_{\mu_\tau}Y^{\eta_t ,x_t}(s)\\
 &\ \ = \tE [\partial_{\tome_\tau} \partial_{\mu_\tau} \Phi (B^{\eta_t}, \ML_{B^{\eta_t}}, \tB^{x_t} ) ]
 + \int_s^T \tE [\partial_{\tome_\tau} \partial_{\mu_\tau} f(\Theta^{\eta_t}_r, \ML_{\Theta_r^{\eta_t}}, \tB^{x_t} )   ] dr \\
&\ \ \ \ \ \ + \int_s^T  \partial_y f(\Theta^{\eta_t}_r, \ML_{\Theta_r^{\eta_t}}) \partial_{\tome_\tau}\partial_{\mu_\tau}Y^{\eta_t ,x_t} (r) dr\\
&\ \ \ \ \ \ +  \int_s^T   \tE [\partial_{\ty} \partial_{\nu} f(\Theta^{\eta_t}_r, \ML_{\Theta_r^{\eta_t}}, \tY^{x_t, \ML_{\eta_t}} ) \partial_{\tome_\tau} \tY^{x_t, \ML_{\eta_t}}  (\partial_{\tome_\tau} \tY^{x_t, \ML_{\eta_t}} (r))^T ]  dr\\
&\ \ \ \ \ \ +  \int_s^T   \tE [ \partial_{\nu} f(\Theta^{\eta_t}_r, \ML_{\Theta_r^{\eta_t}}, \tY^{x_t, \ML_{\eta_t}} ) \partial^2_{\ome_\tau} \tY^{x_t, \ML_{\eta_t}} (r) ]  dr\\
 &\ \ \ \ \ \  +  \int_s^T   \tE [ \partial_{\nu} f (\Theta^{\eta_t}_r, \ML_{\Theta_r^{\eta_t}}, \tY^{ \teta_t } ) \partial_{\tome_\tau}\partial_{\mu_\tau}\tY^{\teta_t ,x_t} (r) ] dr\\
&\ \ \ \ \ \ + \int_s^T  \partial_z f(\Theta^{\eta_t}_r, \ML_{\Theta_r^{\eta_t}}) \partial_{\tome_\tau}\partial_{\mu_\tau}Z^{\eta_t ,x_t} (r) dr
- \int_s^T  \partial_{\tome_\tau}\partial_{\mu_\tau}Z^{\eta_t ,x_t}(r)\ d B(r).
\end{split}
\ee
Then we have the following lemma via a similar proof of Lemma \ref{dxu}.

\begin{lem}\label{s-dxu}

There exists a unique solution $(\partial_{\tome_\tau}\partial_{\mu_\tau}Y^{\gl ,x_t}, \partial_{\tome_\tau}\partial_{\mu_\tau}Z^{\gl ,x_t})$ of BSDE \eqref{e-durx2}. Moreover, $\partial_{\tome_\tau}\partial_{\mu_\tau}Y^{\gl ,x_t}$ is the SVD of $\partial_{\mu_\tau}Y^{\gl ,x_t} $ at $(\tau,t,x)$, and
for any $K>0$,
\be
\begin{split}
&\|(\partial_{\tome_\tau}\partial_{\mu_\tau}Y^{\gl ,x_t}   ,\partial_{\tome_\tau}\partial_{\mu_\tau}Z^{\gl ,x_t} )  \|_{\bS^p \times \bH^p} \le C_p ,\ \ \ \ \  \\
&\|(\partial_{\tome_\tau}\partial_{\mu_\tau}Y^{\gl ,x_t}-  \partial_{\tome_\tau}\partial_{\mu_\tau}Y^{\gam_t',\eta'_t ,x'_t} , \partial_{\tome_\tau}\partial_{\mu_\tau}Z^{\gl ,x_t} -  \partial_{\tome_\tau}\partial_{\mu_\tau}Z^{\gam_t',\eta'_t ,x'_t}    )\|_{\bS^p \times \bH^p} \\
&\ \ \ \le C_{K,p} ( \|x_t-x'_t\| + W_2(\ML_{\eta_t}, \ML_{\eta'_t}) + \|\gam_t-\gam'_t\| ),\\
&\ \ \ \ \ \ \ \forall \ (\gam,\eta,x),\ (\gam',\eta',x')\in \D \times \MM_2^D \times \D \text{ such that } |||\ML_{\eta_t}|||, |||\ML_{\eta'_t}||| \le K,
\end{split}
\ee
with some constants $C_p$ and $C_{K,p}.$

\end{lem}


\section{Appendix}

\subsection{Proof of Lemma \ref{unibd}}

We omit the proof of inequality \eqref{eta1} since it is similar to that of~\eqref{gam1} for $p=2.$ 
Now suppose that \eqref{eta1} is true, and we show inquality \eqref{eta2} first. In what follows,  we write $(\MY,\MZ):=( Y^{  \eta_t}, Z^{   \eta_t} )$ and notations such as $(\MY', \MZ')$ and $( \Phi, \Phi')$ are defined in a similar way. Set $(\delta \MY, \delta \MZ):= (\MY-\MY', \MZ-\MZ')$. We see that $(\delta \MY, \delta \MZ)$ solves the following linearized BSDE
\be\label{linear-app}
\begin{split}
\delta \MY (s) &=  \Phi-\Phi' + \int_s^T [f(\Theta^{\eta}_r, \ML_{\Theta^{\eta}_r} ) - f(\Theta^{\eta'}_r, \ML_{\Theta^{\eta'}_r} ) ] dr - \int_s^T \delta \MZ(r) d B(r) \\
&= : \delta \Phi + \int_s^T \Big( a_r \delta \MY(r) + b_r \delta \MZ(r) + \tE[ \tilde{c}_r \delta \tMY(r)] + \delta h_r \Big) dr  - \int_s^T \delta \MZ(r) d B(r) ,
\end{split}
\ee
where
\beaa
&&a_r:= \int_0^1 \partial_y {f} (B^{\eta_t}_r , \MY'+ \theta(\MY-\MY'), \MZ, \ML_{\Theta^\eta_r} ) d\theta, \\
&&b_r:=  \int_0^1 \partial_z {f} (B^{\eta_t}_r , \MY' , \MZ'+\theta(\MZ-\MZ'), \ML_{\Theta^{\eta_t}_r} ) d\theta ,\\
&&\tilde{c}_r:= \int_0^1 \partial_{\nu} f ( B^{\eta_t}_r , \MY' , \MZ', \ML_{B_r^{\eta_t}}, \ML_{\MY'+ \theta(\MY -\MY')} ,   \tMY'+ \theta(\tMY -\tMY' ) ) d \theta,\quad \text{and} \quad \\
&&\delta h_r := f( B^{\eta_t}_r ,  \MY', \MZ', \ML_{B^{\eta_t}_r} , \ML_{\tMY'}  ) - f( B^{{\eta'_t}}_r ,  \MY', \MZ', \ML_{B^{\eta'_t}_r} , \ML_{\tMY'}  ) .
\eeaa

Let
$$F(\Theta^\eta_t, y_1 , y_2  ):= \int_0^1 \partial_{\nu} f ( B^{\eta_t}_r , \MY' , \MZ', \ML_{B_r^{\eta_t}}, \ML_{\MY'+ \theta(\MY -\MY')} ,   y_1 + \theta y_2 ) d \theta.$$  Then,  $F$ is uniformly Lipschitz continuous in $(y_1, y_2)$ in view of Assumption {\bf(H0)}. On the other hand, since $\MY', \MZ' \in \bH^2,$ we deduce that $F(\Theta^\eta_t,0,0) \in \bH^2$, and moreover, we have
\bea
\| \MY \|_{\bH^2} + \|\MY' \|_{\bH^2} + \| F(\Theta^\eta_t,0,0) \|_{\bH^2} \le C (1+\|\eta_t\|_{\bS^2} + \|\eta'_t\|_{\bS^2} ).
\eea

Then applying estimates of Lemma \ref{linear} to BSDE \eqref{linear-app}, we have
\be\label{ine1-app}
\|\delta \MY \|^2_{\bS^2} + \| \delta \MZ \|^2_{\bH^2} \le C (\|\Phi-\Phi'  \|^2_{L^2} + \|\delta h \|^2_{\bH^2}) e^{C(\| F(\Theta^\eta_t, 0 , 0)\|_{\bH^2}+ \|(\MY, \MY')\|_{\bH^2} )}.
\ee
Furthermore, using the Lipschitz continuity of $\Phi$ and $f,$ we have
\be
\|\Phi-\Phi'  \|_{L^2} + \|\delta h \|_{\bH^2} \le C \|\eta_t-\eta'_t\|_{\bS^2},
\ee
and thus the desired estimate \eqref{eta2} in view of inequality \eqref{ine1-app}.

\ \ \ \ \ \ Now we show inequalities \eqref{gam1}. In what follows,  we omit the superscript $(\gamma_t, \eta_t)$ for simplicity.
Without loss of generality, we assume $p= 2q, \ q \in \Z^+.$ Otherwise,  we replace $|Y|$ with $(|Y|^2+\vep)^{\frac12 }$ in the following argument and then take the limit $\vep \rightarrow 0$. Applying It\^o's formula to $|Y|^p$ on $ [s,T],$  we have
\be \label{ito1}
\begin{split}
&|Y(s)|^p + \frac12 p (p-1) \int_s^T |Y|^{p-2} |Z|^2 dr\\
& \quad =\ | \Phi(B_T^{\gamma_t}, \ML_{B_T^{\eta_t} }) |^p + p \int_s^T  |Y(r)|^{p-1} f(\Theta^{\gamma_t, \eta_t}_r, \ML_{\Theta^{\eta_t}_r} ) dr -  p \int_s^T |Y(r)|^{p-1} Z(r) dB(r) .
\end{split}
\ee
Since $\Phi$ and $f$ are Lipschitz continuous, we have $$
\Phi(B_T^{\gamma_t}, \ML_{B_T^{\eta_t} })  , \int_t^T f( B_r^{\gamma_t}, 0, 0 ,  \ML_{B_r^{\eta_t}}, \ML_{Y^{\eta_t}(r)}  ) dr \in L^q(\MF_T), \quad \forall q \ge 1.
$$
Then,  using standard estimates of BSDEs, we obtain
$\|Y\|_{\bS^q} + \|Z\|_{\bH^q} < \infty. $
Taking the expectation on both sides of identity \eqref{ito1},  we have
\be\label{ine1}
\E[|Y(s)|^p   + \frac12 p (p-1) \int_s^T |Y|^{p-2} |Z|^2 dr  ] \le \E[|\Phi|^p] + p \int_s^T  |Y(r)|^{p-1} f(\Theta^{\gamma_t, \eta_t}_r, \ML_{\Theta^{\eta_t}_r} ) dr .
\ee

Applying Young's inequality to the last integral, we have
\be \label{ineq}
\begin{split}
 &\int_s^T  |Y(r)|^{p-1} f(\Theta^{\gamma_t, \eta_t}_r, \ML_{\Theta^{\eta_t}_r} ) dr\\
   &\ \ \le  \int_s^T \Big[ |Y|^{p-1} f( B^{\gamma_t}_r, 0,0, \ML_{B_r^{\eta_t}}, \delta_0   ) dr + C |Y|^p
   + C |Y|^{p-1} [\tE[|\tY^{\teta_t}(r) |^2]  ]^{\frac12} + C|Y|^{p-1} |Z| \Big] dr \\
& \ \ \le  (C_p + \frac{C}{\vep} ) \int_s^T |Y|^{p} dr + \frac{1}{p} \int_s^T | f(B^{\gamma_t}_r, 0,0, \ML_{B_r^{\eta_t}}, \delta_0 )|^p dr\\
&\ \ \ \ \ \  + \vep C \int_s^T |Y|^{p-2} |Z|^2 dr + \frac{C}{2p} \int_s^T  \| \tY^{\teta_t}(r)\|_{L^2}^{p} dr, \quad \forall \vep >0.
\end{split}
\ee
Then by choosing a small enough $\vep$ such that $\frac12 p (p-1)- \vep C>0$, we obtain
\be\label{app-1}
\begin{split}
&\E[|Y(s)|^p + C_p \int_s^T |Y(r)|^{p-2} |Z(r)|^2 dr  ]\\
&\quad \le   \E  [ |\Phi(B_T^{\gamma_t}, \ML_{B_T^{\eta_t} })|^p ] +  C_p\int_s^T [ | f(B^{\gamma_t}_r, 0,0, \ML_{B_r^{\eta_t}}, \delta_0 )|^p dr + \int_s^T |Y|^{p} dr  ]\\
&\quad \quad +   \int_s^T  \| \tY^{\teta_t}(r)\|_{L^2}^{p} dr.
\end{split}
\ee
Apply Gronwall's inequality to \eqref{app-1}, and we obtain
\be\label{ine3}
\begin{split}
& \E [|Y(s)|^p + C_p \int_s^T |Y(r)|^{p-2} |Z(r)|^2 dr ]\\
&\ \ \ \le C_p \E [|\Phi(B_T^{\gamma_t}, \ML_{B_T^{\eta_t} })|^p +  \int_s^T  | f(B^{\gamma_t}_r, 0,0, \ML_{B_r^{\eta_t}}, \delta_0 )|^p dr]  +   \int_s^T  \| \tY^{\teta_t}(r)\|_{L^2}^{p} dr.
\end{split}
\ee

\ \ \ \ \ \  Then in view of inequalities \eqref{ineq} and \eqref{ito1}, choosing $\vep$ sufficiently small, we have
\be\label{ine2}
\begin{split}
|Y(s)|^{p} &\le |\Phi(B_T^{\gamma_t}, \ML_{B_T^{\eta_t} })|^p +   C_p \Big[  \int_s^T | f(B^{\gamma_t}_r, 0,0, \ML_{B_r^{\eta_t}}, \delta_0 )|^p dr\\
&\ \ \ + \int_s^T |Y(r)|^p dr + \int_s^T \|Y^{\eta_t}(r)\|^p_{L_2} dr \Big]  -p \int_s^T |Y|^{p-1} Z(r) dB(r).
\end{split}
\ee
Applying Burkholder-Davis-Gundy inequality to the right hand side of inequality \eqref{ine2}, we obtain
\be\label{ine2-app}
\begin{split} 
\E [\sup_{s\in[t,T] } |Y(s)|^p]
  \le&\ \E\left[ |\Phi(B_T^{\gamma_t}, \ML_{B_T^{\eta_t} })|^p\right]  + C_p \, \E \int_t^T  | f(B^{\gamma_t}_r, 0,0, \ML_{B_r^{\eta_t}}, \delta_0 )|^p dr \\ 
&+C_p\, \E \int_s^T \|Y^{\eta_t}(r)\|^p_{L_2} dr+C_p\, \E\int_t^T  |Y(r)|^p dr \\
&+  C_p\, \E\left[\Big( \int_t^T |Y|^{2p-2} |Z|^2  dr   \Big)^\frac12\right] \\ 
\le& \ \E \left[ |\Phi(B_T^{\gamma_t}, \ML_{B_T^{\eta_t} })|^p\right]  + C_p\, \E  \int_t^T  | f(B^{\gamma_t}_r, 0,0, \ML_{B_r^{\eta_t}}, \delta_0 )|^p dr\\
&+C_p\, \E \int_s^T \|Y^{\eta_t}(r)\|^p_{L_2} dr
  + C_p\, \E  \int_t^T  |Y(r)|^p dr \\
  &+ \vep\, \E\left[\sup_{s\in[t,T]} |Y(s)|^p\right] + \frac{1}{4\vep} \E\left[\int_t^T |Y|^{p-2} |Z|^2 dr \right].
 \end{split}
 \ee

Then in view of \eqref{ine3}, \eqref{eta1} and \eqref{ine2-app}, for sufficiently small  $\vep$, we have
\be\label{iney}
\begin{split}
 &\E\left[\sup_{s\in[t,T]} | Y(s)|^p\right ]\\
  \le &\ \ C_p\, \E \left[  |\Phi(B_T^{\gamma_t}, \ML_{B_T^{\eta_t} })|^p \right] +     C_p\, \E \int_t^T  | f(B^{\gamma_t}_r, 0,0, \ML_{B_r^{\eta_t}}, \delta_0 )|^p dr \\
  &\ \ +C_p \, \E \left[\int_s^T \|Y^{\eta_t}(r)\|^p_{L_2} dr \right] \\
\le&\ \   C_p \left( 1+ \E [  \| B^{\gamma_t}\|^p  ]+ \| \| B^{\eta_t}  \| \|_{L_2}^p +\| \| \eta_t \| \|^p_{L_2} \right)\\
\le&\ \   C_p  \left( 1+ \E\left[ ( \|B\| +\| \gamma_t \|  )^p\right] +  \E\left[ ( \|B\| +\| \eta_t \|  )^2 \right]^\frac{p}{2}+\| \| \eta_t \| \|^p_{L_2} \right)\\
 \le&\ \   C_p  \left( 1+ \| \gamma_t \|^p + \| \| \eta_t\| \|^{p}_{L_2} \right).
\end{split}
\ee
Let $\tilde{f}( Y, Z ):= f(B_r^{\gamma_t}, Y , Z  ,  \ML_{B_r^{\eta_t}}, \ML_{Y^{\eta_t}(r)} )$. Using a standard argument of BSDEs, we have
\begin{equation*}
\begin{split}
 \E[|\int_t^T |Z|^2 dr|^{\frac{p}{2} } ] \le&\, C_p \E \left[ |\Phi(B_T^{\gamma_t}, \ML_{B_T^{\eta_t} })|^p  +   \int_t^T  | \tilde{f}( 0,0 )|^p dr +C_p \int_s^T \|Y (r)\|^p_{L_2} dr\right]\\
 & \, + C_p\, \E \left[\sup_{s\in[t,T]} |Y(s)|^p \right] 
\le\,   C_p  ( 1+ \| \gamma_t \|^p + \| \| \eta_t \| \|^p_{L_2}  ),
\end{split}
\end{equation*}
and thus \eqref{gam1}.

\ \ \ \ \ \ It remains to prove \eqref{gam2}.  Note that $(\delta Y, \delta Z):= (Y-Y', Z-Z')$ solves the following linearized BSDE
\be
\begin{split}
\delta Y (s) &=  \Phi-\Phi' + \int_s^T [f(\Theta_r, \ML_{\Theta^{\eta_t}_r} ) - f(\Theta'_r, \ML_{\Theta^{\eta'_t}_r} ) ] dr - \int_s^T \delta Z(r) d B(r) \\
&= : \delta \Phi + \int_s^T \Big( \alpha_r \delta Y(r) + \beta_r \delta Z(r) + \delta h_r + \delta f_r \Big) dr  - \int_s^T \delta Z(r) d B(r) ,
\end{split}
\ee
where
\beaa
&&\alpha_r := \int_0^1 \partial_y {f} (B^{\gamma_t}_r , Y'+ \theta(Y-Y'), Z, \ML_{\Theta_r} ) d\theta, \\
&&\beta_r :=  \int_0^1 \partial_z {f} (B^{\gamma_t}_r , Y' , Z+\theta(Z-Z'), \ML_{\Theta_r} ) d\theta ,\ \ \delta \tY^\eta := \tY^{\teta}-\tY^{\teta'}, \\
&&\delta{h}_r := f( B^{\gamma_t}_r ,  Y', Z', \ML_{B^{\eta_t}_r} , \ML_{Y^{\eta_t}}  ) - f( B^{\gamma_t}_r ,  Y', Z', \ML_{B^{\eta_t}_r} , \ML_{Y^{\eta'_t}}  ),\quad \text{and} \quad \\
&&\delta f_r := f( B^{\gamma_t}_r ,  Y', Z', \ML_{B^{\eta_t}_r} , \ML_{Y^{\eta'_t}}  ) - f( B^{\gamma'_t}_r ,  Y', Z', \ML_{B^{\eta'_t}_r} , \ML_{Y^{\eta'_t}}  ) .
\eeaa
Using the Lipschitz continuity of $f$ in $(y, z)$ and standard estimate for linear BSDEs (see e.g. \cite{BDH03}), we have
\be\label{ine6}
\E\Big[ \sup_{s\in[t,T] }   | \delta Y (s) |^p +  (\int_t^T | \delta Z(s) |^2 ds )^{\frac{p}{2}}  \Big] \le C_p \Big(  \E| \delta \Phi |^p + \E \Big| \int_t^T (  \delta h_r + \delta f_r  ) dr \Big|^p  \Big).
\ee
Then,  in view of estimate \eqref{eta2}, we have
\be\label{ine4}
\E | \delta \Phi |^p +\E | \int_t^T     \delta h_r   dr |^p + \E | \int_t^T     \delta f_r   dr |^p \le C_K ( \| \gamma_t - \gamma'_t \|^{p} + W_2(\ML_{\eta_t}, \ML_{\eta'_t})^p  ),
\ee
and thus the desired estimate \eqref{gam2}.

\subsection{An extension of \cite[Theorem 4.5]{PW16} without assumption of local Lipschitz continuity in time}

\begin{lem}  Let non-anticipative functional $f: [0,T] \times \D \times \R \times \R^d \mapsto \R$ lie in $\FC^{0,2,2,2}_s.$  Assume that for any $t   \in [0,T]  $ and $\tau \le t,$ $f  $ and all its derivatives are locally Lipschitz continuous on $\D$: for $\phi(t, \cdot )= (I, \partial_{\ome_\tau},   \partial_{\ome_\tau}^2) f (t,\cdot, 0,0), $
	\be
	|\phi(t, \ome)-  \phi(t, \ome') | \le C(1+\|\ome_t\|^k +\|\ome'_t\|^k) (\|\ome_t - \ome'_t \|) , \quad \forall \ (\ome,\ome') \in \D^2,
	\ee 
	for some constant $C$ and integer $k.$ Moreover, suppose that the first-order derivatives in $(y, z)$, as well as their first-order derivatives w.r.t. $(\ome_{\tau}, y, z)$ are uniformly bounded.  If $\Phi $ satisfies \eqref{weakphi}, there is a unique classical solution of the following PPDE
\bea\label{ppde2}
\left\{
\begin{array}{l}
\partial_t u(t,\gam  ) + \frac12 \text{Tr}\ [ \partial_{\ome}^2 u(t,\gam)]
 + f(t,\gam, u(t,\gam), \partial_\ome u(t,\gam) )=0,\\
 \\
 u(T, \gam)= \Phi (\gam_T ),\ \ \ \ (t,\gam )\in [0,T] \times \mC.
\end{array}
\right.
\eea
\end{lem}

\begin{proof}  The uniqueness is a consequence of that of  the following non-Markovian BSDE
	\be\label{pureppde}
	Y^{\gam_t}(s)= \Phi (B^{\gam_t}_T) + \int_s^T f(r, B^{\gam_t}, Y^{\gam_t}(r), Z^{\gam_t}(r)) dr - \int_s^T Z^{\gam_t}(r) dB(r).
	\ee
We now sketch the proof of the existence.  Set
\be
u(t,\gam):= Y^{\gam_t}(t).
\ee
Similar to that of \cite[Theorems 3.9 and 3.10]{PW16}, we have $u \in \FC^{0,2}_{s,p}$ and moreover, 
\be
u(s, B^{\gam_t})= Y^{\gam_t}(s), \quad     \partial_{\ome_t} u(s, B^{\gam_t}) = Z^{\gam_t}(s), \quad s \ge t.
\ee
Then,  applying the partial It\^{o} formula \eqref{se-ito'}, we have that for any $\delta>0,$ 
\begin{align*}
&u(t+\delta , \gam_t) - u(t,\gam_t) \\
&\quad = u(t+\delta, \gam_t)- \E[u(t+\delta, B^{\gam_t})] + \E[u(t+\delta, B^{\gam_t})] - u(t,\gam_t)\\
& \quad = \E \left[ - \int_t^{t+\delta} \partial_{\ome_r} u(t+\delta, B^{\gam_t}_r) dB(r) - \frac12  \int_t^{t+\delta} \text{Tr} [ \partial_{\ome_r}^2 u(t+\delta, B^{\gam_t}_r ) ] dr  \right]\\
& \quad \quad + \E \left[ - \int_t^{t+\delta} f(r, B^{\gam_t}, Y^{\gam_t}(r), Z^{\gam_t}(r) ) dr \right].
\end{align*}
Dividing both sides of the above identity by $\delta$ and taking $\delta \rightarrow 0,$ we complete the proof.

\end{proof}




\end{document}